\documentclass[11pt]{amsart}

\usepackage{graphicx}
\usepackage{amssymb}
\usepackage{amsmath,amscd,color,pinlabel}
\usepackage[backref]{hyperref}

\newtheorem{thm}{Theorem}[section]
\newtheorem{prp}[thm]{Proposition}

\newtheorem{lma}[thm]{Lemma}

\newtheorem{proposition}[thm]{Proposition}

\newtheorem{lemma}[thm]{Lemma}

\theoremstyle{definition}

\theoremstyle{remark}
\newtheorem{remark}[thm]{Remark}
\newtheorem{rmk}[thm]{Remark}

\def\d{\partial}

\newcommand{\R}{{\mathbb{R}}}
\newcommand{\C}{{\mathbb{C}}}

\newcommand{\Z}{{\mathbb{Z}}}

\newcommand{\A}{{\mathcal{A}}}

\newcommand{\M}{{\mathcal{M}}}

\newcommand{\rk}{\operatorname{rank}}

\newcommand{\ev}{\operatorname{ev}}

\newcommand{\pa}{\partial}

\newcommand{\area}{\operatorname{area}}

\newcommand{\lk}{\operatorname{lk}}
\newcommand{\slk}{\operatorname{slk}}

\def\d{\partial}
\newcommand{\Aug}{\operatorname{Aug}}
\newcommand{\qAug}{\widehat{\operatorname{Aug}}}

\begin{document}

\title[Higher genus knot contact homology\dots]
{Higher genus knot contact homology and recursion for colored HOMFLY-PT polynomials}
\author{Tobias Ekholm}
\address{Department of mathematics, Uppsala University, Box 480, 751 06 Uppsala, Sweden;
Institut Mittag-Leffler, Aurav 17, 182 60 Djursholm, Sweden}
\email{tobias.ekholm\@@math.uu.se}
\author{Lenhard Ng}
\address{Mathematics Department, Duke University, Durham, NC 27708 USA}
\email{ng@math.duke.edu}
\thanks{TE is supported by the Knut and Alice Wallenberg Foundation and by the Swedish Research Council.
Parts of the paper are based upon work supported by the National Science Foundation under Grant No. DMS-1440140 while the author was in residence at the Mathematical Sciences Research Institute in Berkeley, California, during the 2018 spring  semester.}
\thanks{LN is partially supported by NSF grants DMS-1406371 and DMS-1707652.}

\begin{abstract}
We sketch a construction of Legendrian Symplectic Field Theory (SFT) for conormal tori of knots and links. Using large $N$ duality and Witten's connection between open Gromov--Witten invariants and Chern--Simons gauge theory, we relate the SFT of a link conormal to the colored HOMFLY-PT polynomials of the link. We present an argument that the HOMFLY-PT wave function is determined from SFT by induction on Euler characteristic, and also show how to, more directly, extract its recursion relation by elimination theory applied to finitely many noncommutative equations. The latter can be viewed as the higher genus counterpart of the relation between the augmentation variety and Gromov--Witten disk potentials established in \cite{AENV} by Aganagic, Vafa, and the authors, and, from this perspective, our results can be seen as an SFT approach to quantizing the augmentation variety. 
\end{abstract}

\maketitle

\section{Introduction}
We start this introduction by reviewing background material and then, in the light of this review, discuss the new results established in the paper. Before beginning, we want to make the important disclaimer that, from a strict mathematical rather than physical viewpoint, both some of the background results and the main constructions in this paper are not rigorously proved and should be considered as conjectures. We will indicate certain points throughout the paper where it is particularly true that rigorous mathematical steps are missing.

\subsection{Background}
Let $K\subset M$ be a link in an orientable 3-manifold $M$. The cotangent bundle $T^{\ast}M$ is a symplectic manifold with symplectic form $\omega=-d(p\,dq)$, where $p\,dq$ is the Liouville $1$-form naturally associated to $M$. The conormal $L_{K}$ of $K$ is the set of all covectors in $T^{\ast}M$ along $K$ that annihilate the tangent vectors to $K$ at the corresponding point in $M$. Then $L_{K}$ is a Lagrangian submanifold with one component for each component of $K$, each diffeomorphic to $S^{1}\times\R^{2}$.  The pair $(T^{\ast}M,L_{K})$ is non-compact but has ideal contact boundary $(ST^{\ast}M,\Lambda_{K})$, where $ST^{\ast}M$ denotes the spherical cotangent bundle, which can be represented as the unit conormal bundle of $T^{\ast}M$ with respect to some Riemannian metric with the contact form equal to the restriction of $p\,dq$, and where $\Lambda_{K}=ST^{\ast}M\cap L_{K}$ is a Legendrian torus (i.e., outside a compact subset, $(T^{\ast}M,L_{K})$ looks like $([0,\infty)\times ST^{\ast}M,[0,\infty)\times\Lambda_{K})$).

Seminal work of Witten \cite{Witten:1988hf} relates $U(N)$ Chern--Simons gauge theory in $M$, with insertion of the monodromy along the link $K\subset M$, to the colored HOMFLY-PT polynomial of $K$ and further to open topological string in $T^{\ast}M$ with $N$ branes along the 0-section $M$ and one brane along $L_{K}$. 

Let $M=S^{3}$ and $K=K_{1}\cup \dots \cup K_{k}$, where $K_{j}$ are the connected components of $K$. In this case, Ooguri--Vafa \cite{OV} found that the topological string in $T^{\ast}S^{3}$ with $N$ branes on $S^{3}$ and $1$ on each $L_{K_{j}}$ corresponds to open topological string in the resolved conifold $X$ (the total space of the bundle $\mathcal{O}(-1)^{\oplus 2}\to\C P^{1}$) with $1$ brane on each $L_{K_j}$ only (i.e., closing off all boundaries in the $N$ copies of the zero section by disks, see \cite{worldsheet}) provided $t=\area(\C P^{1})=Ng_{s}$, where $g_{s}$ denotes the string coupling constant. Here we abuse notation and use $L_{K_j}$ to denote a Lagrangian in $X$ that corresponds to the conormal $L_{K_j}$ in $T^{\ast}S^{3}$. The Lagrangian in $X$ is obtained by shifting the conormal in $L_{K_{j}}\subset T^{\ast}S^{3}$ off of the zero section along a closed 1-form defined in a tubular neighborhood of $K_{j}$ and dual to its tangent vector, see \cite{koshkin}.

For the purposes of this paper the outcome of these results is a relation between the colored HOMFLY-PT polynomials and open Gromov--Witten theory of $L_{K}\subset X$, see \cite{koshkin}, that takes the following form. Let $x=(x_{1},\dots,x_{k})$ and define the colored HOMFLY-PT \emph{wave function}
\[ 
\Psi_{K}(x)=\sum_{n=(n_{1},\dots,n_{k})} H_{K;n}(q,Q) e^{n\cdot x}
\] 
where $Q=q^{N}$ and $H_{K;n}(q,Q)$ is the $n$-colored HOMFLY-PT polynomial. Here $n$-colored means that the component $K_{j}$ is colored by the $n_{j}^{\rm th}$ symmetric representation of $U(N)$. 

We next consider the open Gromov--Witten potential. The relative homology $H_{2}(X,L_{K})$ is generated by classes $t$ corresponding to the generator of $H_{2}(X)$ and $x_{j}$ (well-defined up to adding multiples of $t$) corresponding to the generator of $H_{1}(L_{K_{j}})$. Let 
\[ 
F_{K}(x,Q,g_{s})= \sum_{\chi,r,n} C_{K;\chi,r,n}\, g_{s}^{-\chi} Q^{r}e^{n\cdot x}
\]
where $C_{K;\chi,r,n}$ counts (generalized) holomorphic curves of Euler characteristic $\chi$ that represent the relative homology class $rt+\sum_{j=1}^{k}n_{j}x_{j}\in H_{2}(X,L_{K})$. Then the Ooguri--Vafa result \cite{OV} says that
\[ 
\Psi_{K}(x)= e^{F_{K}(x)},\quad\text{ for }\quad q=e^{g_{s}},\; Q=e^{t}=e^{Ng_{s}},
\]
where the $e^{F_{K}(x)}$ can be interpreted as counting all disconnected curves.

In \cite{AV, AENV} the short wave asymptotics of the wave function $\Psi_{K}$ were related to knot contact homology. Knot contact homology is the homology of the Chekanov--Eliashberg dg-algebra $\A_{K}=CE(\Lambda_{K})$ of the unit conormal $\Lambda_{K}$ of a knot or link. This is a graded algebra over $\C[H_{2}(ST^{\ast}S^{3},\Lambda_{K})]$ (in degree $0$) generated by the Reeb chords $c$ of $\Lambda_{K}$, which in this case correspond to geodesics $\gamma$ in $S^{3}$ with endpoints on $K$ that are perpendicular to $K$ at their endpoints. One can position $\Lambda_K$ so that the grading $|c|$ of each Reeb chord $c$ coincides with the Morse index of the corresponding geodesic $\gamma$. Therefore, $\A_{K}$ is supported in degrees $\ge 0$. The differential $d\colon \A_{K}\to \A_{K}$ counts holomorphic disks $u\colon (D,\partial D)\to (\R\times ST^{\ast}S^{3},\R\times\Lambda_{K})$ with one positive and several negative punctures where the disk is asymptotic to Reeb chord strips. After choosing capping disks for the Reeb chords, such disks represent relative homology classes in $H_{2}(ST^{\ast}S^{3},\Lambda_{K})$. We pick generators for this homology group: $x_{j}$ and $p_{j}$ corresponding to the longitude and meridian in $H_{1}(\Lambda_{K_{j}})$ and $t$ corresponding to a generator of $H_{2}(ST^{\ast}_{\mathrm{pt}} S^{3})$, the second homology of the fiber. This way, $\A_{K}$ can be viewed as an algebra over the ring
\[ 
\C[e^{\pm x_{j}},e^{\pm p_{j}}, Q^{\pm 1}]_{j=1}^{k},\text{ where }Q=e^{t}.
\]    
The differential in $\A_{K}$ was computed explicitly from a braid presentation of $K$ in \cite{EENS}.

To describe the relation between open topological strings and knot contact homology we think of $Q\in\C^{\ast}$ as a deformation parameter and $\A_{K}$ as a family of algebras over the family of complex tori $(\C^{\ast})^{2k}\times\C^{\ast}$, where the coordinates in $(\C^{\ast})^{2k}$ correspond to $e^{\pm x_{j}}$ and $e^{\pm p_{j}}$, $j=1,\dots,k$. We consider the locus in the coefficient space $(\C^{\ast})^{2k}\times\C^{\ast}$ where $\A_{K}$ admits an augmentation; an augmentation is a unital chain map
\begin{equation}\label{eq:augmenation} 
\epsilon\colon \A_{K}\to\C,
\end{equation}
where $\C$ lies in degree $0$ and is equipped with the zero differential.
The closure of the highest-dimensional part of this locus is a variety $V_{K}\subset(\C^{\ast})^{2n}\times\C^{\ast}$ called the \emph{augmentation variety}. 

To see how $\A_{K}$ determines $V_{K}$, note that the differentials of the degree $1$ generators give a collection of polynomials $\{P_{c_{r}}(a_{1},\dots,a_{m})\}_{|c_{r}|=1}$, with coefficients in $\C[e^{\pm x_{j}},e^{\pm p_{j}},Q^{\pm 1}]_{j=1}^{k}$, in the degree $0$ generators $a_{1},\dots,a_{m}$. The chain map equation for $\epsilon$ reduces to $\epsilon\circ d=0$, and hence $V_{K}$ is the locus where the polynomials $\{P_{c_{r}}(a_{1},\dots,a_{m})\}_{|c_{r}|=1}$ have common roots, corresponding to $\epsilon(a_{j})$, $j=1,\dots,m$. The locus $V_{K}$ is thus an algebraic variety that can be determined by elimination theory. 

Consider the semi-classical approximation of the count $\Psi_{K}(x)$ of disconnected holomorphic curves in $X$ with boundary on $L_{K}$:
\[ 
\Psi_{K}(x) = \exp\left(g_{s}^{-1} F_{K;0}(x) + F_{K;1}(x) + \mathcal{O}(g_{s})\right).
\]       
From the open string expression for $\Psi_{K}$, we identify $F_{K;0}(x)=W_{K}(x)$ as the Gromov--Witten disk potential for the Lagrangian filling $L_{K}$, $F_{K;1}(x)$ the annulus potential, and so on for lower Euler characteristic. 

In \cite[Section 6.4]{AENV} it was argued that
\[ 
p_{j}=\frac{\partial W_{K}}{\partial x_{j}}, \quad j=1,\dots,k,
\]
gives a local parameterization of a branch of $V_{K}$. Hence $V_{K}$ is a complex Lagrangian variety with respect to the standard symplectic form $\sum_{j=1}^{k} dx_{j}\wedge dp_{j}$ on $(\C^{\ast})^{2k}$. In \cite[Section 8]{AENV}, the quantization of the augmentation variety $V_{K}$ was discussed from a physical point of view. Here $V_{K}$ appears as the characteristic variety for a $D$-module. The $D$-module arises from the ``D-model'', which is the A-model topological string in $(\C^{\ast})^{2n}$ with a space filling coisotropic brane and a Lagrangian brane on $V_{K}$. This theory has a unique ground state, leading to a wave function on $V_{K}$ that then generates the $D$-module. 

In this paper we discuss the quantization of $V_{K}$ from the point of view of A-model topological string in $X$ with one Lagrangian brane on $L_{K}$, i.e., the open Gromov--Witten potential of $L_{K}$. We study 
in particular how to determine the corresponding wave function and $D$-module in terms of the Symplectic Field Theory (SFT) of $\Lambda_{K}$, i.e., the holomorphic curves with boundary on $L_{K}$ at infinity in $X$. 

\subsection{Legendrian SFT}
\label{ssec:lsft}
In \cite{EGH}, a framework for holomorphic curve theories in symplectic cobordisms called Symplectic Field Theory (SFT) was introduced. Here the holomorphic curves are asymptotic to closed Reeb orbits in the closed string case and to Reeb chords in the open string case. For closed strings the full version of SFT including curves of all genera was constructed, but in the open string case the construction stopped at the most basic level of SFT, the Chekanov--Eliashberg dg-algebra discussed above. The obstructions to a higher genus generalization of $CE$ for open strings are related to the codimension $1$ boundary in the moduli space of holomorphic curves that corresponds to so-called boundary bubbling. 

Very briefly, the algebraic structures of closed string SFT come from identifying the boundary of a 1-dimensional moduli space of holomorphic curves in a sympletic cobordism as two-level holomorphic curves with a $0$-dimensional curve in the cobordism and a 1-dimensional $\R$-invariant curve in the symplectization either above or below, see Figure \ref{fig:closedsft}. Here interior bubbling, as in Figure \ref{fig:closedbubble}, has codimension $2$ and carries no homological information and can hence be neglected. In open string SFT one considers holomorphic curves with boundary on a Lagrangian submanifold, and here two-level curves do not account for the whole codimension $1$ boundary. Boundary bubbling has codimension $1$ and cannot be disregraded, see Figure \ref{fig:boundarybubble}.

\begin{figure}[htp]
	\centering
	\includegraphics[width=.35\linewidth]{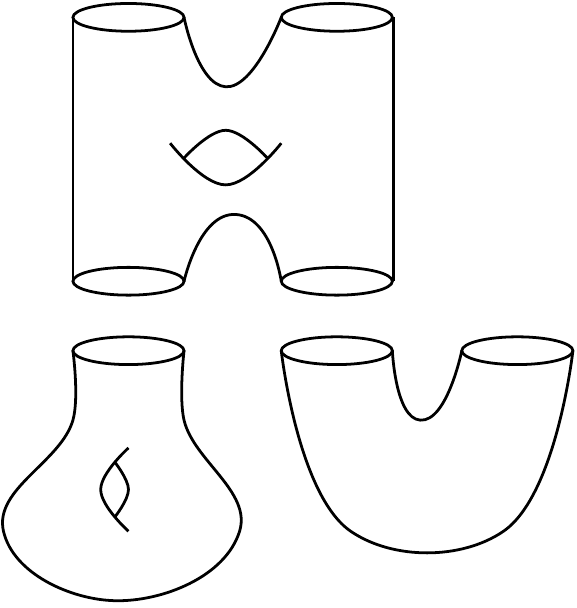}
	\caption{Two-level curves in closed string SFT. The top level is in the symplectization, the bottom in the cobordism.}
	\label{fig:closedsft}
\end{figure}

\begin{figure}[htp]
	\centering
	\includegraphics[width=.6\linewidth]{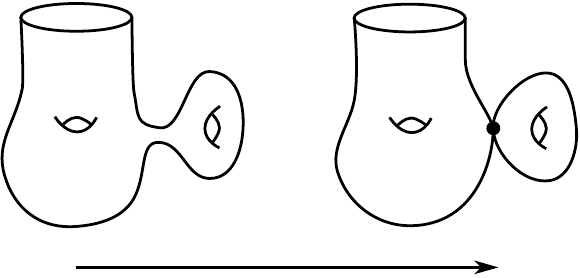}
	\caption{Interior bubbling has codimension two.}
	\label{fig:closedbubble}
\end{figure}

\begin{figure}[htp]
\centering
	\includegraphics[width=.6\linewidth]{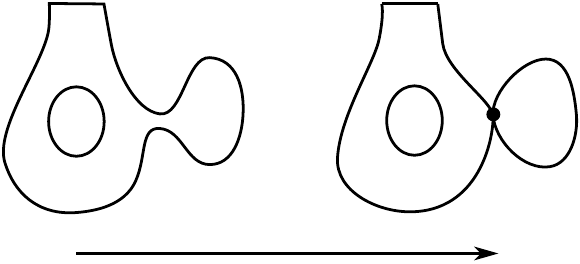}
	\caption{Boundary bubbling has codimension one.}
	\label{fig:boundarybubble}
\end{figure}

Partial generalizations of the Chekanov--Eliashberg dg-algebra are known in the open string case. More specifically, flavors of so-called rational SFT, incorporating disks with several positive punctures, were considered in \cite{Ng_rstf, Ekholm_rsft}. 

The main tool in this paper is a higher genus generalization of knot contact homology. More precisely, we sketch a construction of a holomorphic curve theory that includes curves of arbitrary genus for the conormal of a link $L_{K}\subset X$. The structure of the theory is analogous to SFT. We call it Legendrian SFT and use notation analogous to that in \cite{EGH} in the closed string case. 

Complete proofs of the main SFT equation require the use of an abstract perturbation scheme for holomorphic curves in combination with the extra structures that we introduce here (closely related to bounding cochains in Floer cohomology as introduced by Fukaya--Oh--Ohta--Ono \cite{FO3}). Perturbation schemes such as Kuransihi structures \cite{FO3}, polyfolds \cite{hofer1,hofer2}, or algebraic topologically defined virtual fundamental cycles \cite{pardon} give a framework for defining solution spaces, which then must have required properties with respect to the extra structures. Existence of suitable perturbation schemes in combination with geometric data related to the bounding cochains of \cite{FO3} and similar to that considered here was studied in the setting of Kuranishi structures in a series of works by Iacovino, see \cite{iacovino1, iacovino2, iacovino3, iacovino4}.  Here we will not discuss technical details of perturbation schemes for direct calculations. Rather, we focus on explaining how to add extra geometric data and how to define generalized holomorphic curves that allow us to remove boundary splitting from the moduli space boundary in such a way that the curve counts needed to extract information from Legendrian SFT become accessible. We study simple examples in detail computing the theory directly in a combinatorial way from a braid presentation. After the preparation of this paper, Ekholm and Shende gave an approach to curve counting in \cite{ESh} which we expect can be adapted to the SFT setting here, see Remark \ref{rmk:parallelchains}.

We next explain the structure of Legendrian SFT. The theory is defined in terms of what we call generalized holomorphic curves, defined in terms of ordinary holomorphic curves and additional geometric data. For details of the definition we refer to Section \ref{sec:gencurves}. 
Consider $\R\times\Lambda_{K}\subset \R\times ST^{\ast}S^{3}$. The main object is the \emph{Hamiltonian}, which counts rigid generalized holomorphic curves with arbitrary positive and negative punctures with boundary on $\R\times\Lambda_K$. Here we have an $\R$-invariant almost complex structure leading to an $\R$ action on the space of holomorphic curves, and ``rigid'' means that the relevant moduli space is $0$-dimensional once we mod out by the $\R$ action. 

To organize the count in the Hamiltonian, we associate to the Reeb chords $c_j$ of $\Lambda_K$ formal variables $c_{j}$ and dual operators $\partial_{c_{j}}$ such that
\[ 
\partial_{c_{j}}(c_{k})=\delta_{jk}, 
\] 
and such that the operators satisfy the usual graded sign commutative rules. For a word $\mathbf{c}=c_{1}\cdots c_{m}$ of Reeb chords we write $\ell(\mathbf{c})=m$ for its length, and  $\partial_{\mathbf{c}}=\partial_{c_{m}}\cdots\partial_{c_{1}}$ for the corresponding dual word of operators.

Write $\mathbf{H}_{K}$ for the generating function of rigid holomorphic curves with arbitrary positive and negative punctures with boundary on $\R\times\Lambda_{K}$:
\[ 
\mathbf{H}_{K}= \sum_{\chi,k,m,r,\mathbf{c}^{+},\mathbf{c}^{-}} H_{K;\chi,n,m,r,\mathbf{c}^{+},\mathbf{c}^{-}}\, g_s^{-\chi+\ell(\mathbf{c}^{+})}e^{n\cdot x}e^{m\cdot p}Q^{r}\,\mathbf{c}^{+}\partial_{\mathbf{c}^{-}}.
\]
Here the sum ranges over all words of Reeb chord asymptotics such that $|\mathbf{c}^{+}|-|\mathbf{c}^{-}|=1$ and $H_{K;\chi,k,m,r,\mathbf{c}^{+},\mathbf{c}^{-}}$ counts the number of rigid curves with positive asymptotics at $\mathbf{c}^{+}$ and negative asymptotics at $\mathbf{c}^{-}$, with Euler characteristic $\chi$, and in relative homology class $n\cdot x+m\cdot p+rt$, where $(x,p)=(x_{1},p_{1},\dots,x_{k},p_{k})$ is a basis in $H_{1}(\Lambda_{K})$ and $t$ a generator in $H_{2}(X)$, as above. 

We will use only a small piece of the algebraic structure of SFT and restrict attention to the most basic moduli spaces corresponding to a certain part of the Hamiltonian. More precisely, if $c$ is a Reeb chord of grading $|c|=1$, we write $\mathbf{H}^{c}_{K}$ for the sum of all terms in $\mathbf{H}_{K}$ with 
\[ 
\mathbf{c}^{+}= c\mathbf{a},
\]
where $\mathbf{a}$ is a word of degree 0 chords. Note that this forces $\mathbf{c}^{-}=\mathbf{a}'$, where also $\mathbf{a}'$ is a word of degree 0 generators.

We write $\mathbf{F}_{K}$ for the SFT-potential, the generating function of rigid curves with positive punctures and boundary on $L_{K}$:
\[ 
\mathbf{F}_{K}= \sum_{\chi,k,r,\mathbf{a}^{+}} F_{K;\chi,k,r,\mathbf{a}^{+}}\, g_s^{-\chi+\ell(\mathbf{a}^{+})}e^{kx}Q^{r}\,\mathbf{a}^{+},
\]
where the sum ranges over all Reeb chord words $\mathbf{a}$ with $|\mathbf{a}|=0$. The key equation of SFT that we will be using here takes the form
\begin{equation}\label{eq:masterconnected} 
e^{-\mathbf{F}_{K}} \ \mathbf{H}_{K}|_{p_{j}=g_{s}\frac{\partial}{\partial x_{j}}} \ e^{\mathbf{F}_{K}}=0, 
\end{equation}
and expresses the fact that the boundary of a compact 1-manifold has algebraically $0$ points, as follows. The exponential counts all disconnected curves and the negative exponential removes additonal 0-dimensional curves not connected to the 1-dimensional piece. The differential operators $\partial_{a_{j}}$ in $\mathbf{H}_{K}$ act on Reeb chords and the substitutions $p_{j}=g_{s}\frac{\partial}{\partial x_{j}}$, $j=1,\dots,k$ have the enumerative meaning of attaching curves along bounding chains intersecting the boundary of the 1-dimensional curves at infinity. We point out that the latter also gives a direct enumerative interpretation of the standard quantization scheme where $\hat x_{j}$ is a multiplication operator and $\hat p_{j}=g_{s}\frac{\partial}{\partial x_{j}}$. Note also that \eqref{eq:masterconnected} implies the simpler equation
\begin{equation}\label{eq:master} 
\mathbf{H}_{K}|_{p_{j}=\frac{\partial}{\partial x_{j}}} \ e^{\mathbf{F}_{K}}=0.
\end{equation}
We will use both forms.

Equation \eqref{eq:master} is the quantized version of the chain map equation that relates the knot contact homology and the Gromov--Witten potential. We show that in basic examples, applying elimination theory in this noncommutative setting, we get a quantization 
\[ 
\hat A_{K;j}(e^{\hat x},e^{\hat p},Q)=0,\; j=1,\dots, s
\]
of the augmentation variety (which corresponds to the commutative $g_s=0$ limit of the defining equations for $V_{K}$ discussed above) such that
\[ 
\hat A_{K;j} \Psi_{K}=0,\; j=1,\dots,s.
\]
In these simple examples we also verify that $\hat A_{K;j}$ agree with the recursion relation for the colored HOMFLY-PT polynomial exactly as expected from large $N$ duality. We point out that our conjectural definition of SFT includes also a definition of the open Gromov--Witten invariant of $L_{K}$, see \cite{iacovino3} for similar results.

\begin{rmk}
In our construction of Legendrian SFT, we focus on the important special case of Lagrangian conormal fillings $L_{K}\subset X$ of $\Lambda_{K}$. Our construction utilizes the topology of the conormal $L_{K}\approx S^{1}\times\R^{2}$. Similar more involved constructions likely work for Lagrangian fillings of $\Lambda_{K}$ of more complicated topology. We leave possible generalizations to future work. 	
\end{rmk}

\subsection{Recursive formulas}
We also give a direct recursive calculation of the wave function $\Psi_{K}$, showing how to compute it genus by genus from data of holomorphic curves at infinity. In fact we further show that it is possible to choose an almost complex structure so that all relevant holomorphic curves at infinity have the topology of the disk. The main underlying principle of this recursion is a calculation of the linearized Legendrian contact homology (a linearized version of the Chekanov--Eliashberg dg-algebra) at a general point of the augmentation variety (where the linearized homology can be identified with the tangent space). We point out that that our recursion takes place in the A-model only (unlike so-called topological recursion, which is a B-model calculation). In fact, the first step of our recursion gives what is called the annulus kernel for general knots, a central ingredient in topological recursion.

\subsection{Augmentation varieties for knots in more general 3-manifolds}
We also consider analogues of the relation derived for the large $N$ dual of knots in the 3-sphere for knots in more general 3-manifolds. We show in particular that the expected connection between K\"ahler classes of the large $N$ dual and free homotopy classes of loops in the 3-manifold appears in knot contact homology in the coefficient ring. Here the coefficient ring is the orbit contact homology which in degree 0 can be shown to be generated by the free homotopy classes. 

\subsection{Examples}
In the last section of the paper we illustrate our results by working out a number of examples in 
detail. More precisely, we study the unknot, the Hopf link, and the trefoil in $S^{3}$, and the line $\R P^{1}$ in $\R P^{3}$.

\section{The SFT potential for conormals of links}
In this section we outline a definition of the open Gromov--Witten potential, or in the language of this paper the SFT potential, for conormals $\Lambda_{K}\subset ST^{\ast} S^{3}$ of links, filled by Lagrangian conormals $L_{K}$ in the resolved conifold $X$. (We get $T^{\ast}S^{3}$ as a special case when the area of the sphere in $X$ is set to zero.) The construction is phrased in terms of a special Morse function on $L_{K}$ and an additonal relative 4-chain $C_{K}$ with $\partial C_{K}=2\cdot L_{K}$. 

The section is organized as follows. In Section \ref{s:adddata} we describe the additional data we use in our main construction. In Section \ref{sec:boundingchains} we describe how to use the data to construct a version of bounding chains for holomorphic curves as in \cite{FO3}. In Section \ref{sec:gencurves} we define generalized holomorphic curves that are counted in the Gromov--Witten potential and in Section \ref{sec:sftpotential} we then define the SFT potential.

\subsection{Additional data for moduli spaces}\label{s:adddata}
Consider the conormal Lagrangian $L_{K}\subset X$, of a link $K=K_{1}\cup\dots\cup K_{k}$. We view the resolved conifold as a symplectic manifold with an (asymptotic) cylindrical end, the symplectization $[0,\infty)\times ST^{\ast} S^{3}$ of the unit cotangent bundle of $S^{3}$. Likewise we view the conormal as having (asymptotic) cylindrical end $[0,\infty)\times \Lambda_{K}\subset ST^{\ast}S^{3}$.

In \cite{AENV} the main relation between the augmentation variety and the Gromov--Witten disk potential (as well as the definition of the disk potential itself) was obtained using so-called bounding chains: non-compact 2-chains in $L_{K}$ that interpolate between boundaries of holomorphic disks and a multiple of a fixed longitude curve in $\Lambda_{K}$ at infinity. Here we need similar bounding chains for curves of arbitrary Euler characteristic. The main difference from the case of disks is the following: instances in a 1-parameter family of curves when the boundary of the curve self-intersects can be disregarded for disks but for higher genus curves the holomorphic curve with self-intersection can be glued to itself creating a curve of Euler characteristic $1$ below that of the original curve. Because of this phenomenon, it is not sufficient to use bounding chains only for rigid curves, as in the disk case. We deal with this by constructing dynamical bounding chains that move continuously with holomorphic curves varying in a 1-parameter family. 

\subsubsection{An additional Morse function on conormals}
Consider a Morse function $f\colon L_K\to\R$ 
of the following form:
\begin{itemize}
\item The critical points of $f$ lie on $K_{j}$ and are: a minimum (index $0$ critical point) $\kappa_{0}^{j}$ and an index  $1$ critical point $\kappa_{1}^{j}$ (in particular, $f$ has no maxima). 
\item The flow lines of $\nabla f$ connecting $\kappa^{0}_{j}$ to $\kappa^{1}_{j}$ lie in $K_{j}$.
\item Outside a small neighborhood of $K_{j}$, the function $f$ is radial and $\nabla f$ is the radial vector field along the fiber disks in $L_{K_{j}}\approx K_{j}\times\R^{2}$. 
\end{itemize}
See Figure \ref{fig:morsefunction}. Note that the unstable manifold $W^{\rm u}(\kappa^{1}_{j})$ of $\kappa^{1}_{j}$ is a disk that intersects $\Lambda_{K_{j}}$ in the meridian cycle $p_{j}$.

\begin{figure}[htp]
\labellist
\small\hair 2pt
\pinlabel $\kappa^{1}$ at 7 53
\pinlabel $\kappa^{0}$ at 106 6
\pinlabel $W^{\rm u}(\kappa^{1})$ at -10 60
\endlabellist
\centering
\includegraphics[width=.5\linewidth]{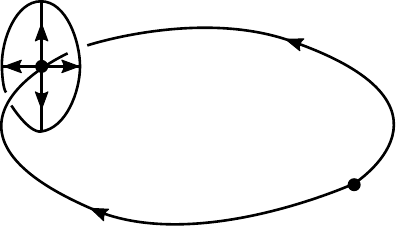}
\caption{The gradient flow of the Morse function $f\colon L_K\to\R$.}
\label{fig:morsefunction}
\end{figure}

\subsubsection{A 4-chain with boundary  
$2\cdot L_{K}$}
The Gromov--Witten potential will be defined using holomorphic curves with boundaries in general position with respect to the gradient vector field $\nabla f$. In 1-parameter families there are isolated instances when the boundaries become tangent to $\nabla f$. To keep curve counts invariant as we cross such instances we will use a certain 4-chain that we describe next. Our construction here was inspired by the study of self linking of real algebraic links (Viro's encomplexed writhe, \cite{viro}) from the point of view taken in \cite{shade}.

We first consider the topology of $X-L_{K}$. Let $N(L_{K})$ denote a small tubular neighborhood of $L_{K}$ and let $\partial N(L_{K})$ denote its boundary. Represent $K=K_{1}\cup\dots\cup K_{k}$ as a braid around the unknot and consider a point $q_{j}$, $j=1,\dots,k$ on each component of the link where the tangent line does not intersect the link except at the point of contact. Let $P_{j}\approx (0,\infty)\times\R^{2}$ denote the union of all parallel transports of the fiber of $L_{K_{j}}$ at the outer points along the half ray tangent to the knot at $q_{j}$.

\begin{lma}\label{l:homologydual}
The homology group $H_{3}(X-L_{K})$ equals $\Z^k$ and is generated by $\partial N(L_{K_{j}})|_{K_{j}}\approx K_{j}\times S^{2}$, $j=1,\dots,k$, and $P_{j}$ is a Poincar\'e dual of $K_{j}\times S^{2}$.  
\end{lma}

\begin{proof}
By homotopy $H_{3}(X-L_{K})\approx H_{3}(X-N(L_{K}))$. Using excision and the long exact sequence for $H_{\ast}(X,N(L_{K}))$ one finds that the map $H_{3}(\partial N(L_{K}))\to H_{3}(X-L_{K})$ induced by inclusion is an isomorphism. 
\end{proof}

We next construct an $\R$-invariant 4-chain at infinity. Outside a compact set, the gradient $\nabla f$  of the Morse function $f\colon L_{K}\to\R$ agrees with the radial vector field in $T^{\ast}S^{3}$ and hence $J\nabla f=R$, where $R$ is the Reeb vector field in $ST^{\ast}S^{3}$. We call a smooth chain $\sigma\subset ST^{\ast}S^{3}$ \emph{regular at the boundary} if the boundary $\partial\sigma$ is a smooth submanifold and if $\sigma$ in a neighborhood of $\partial\sigma$ agrees with an embedding of $\partial\sigma\times[0,\epsilon)$. For regular chains we define the inward normal vector field along $\partial\sigma$ as $\frac{\partial\sigma}{\partial t}|_{\partial\sigma\times\{0\}}$, where $t$ is the standard coordinate on $[0,\epsilon)$. 

We define two smooth 3-chains in $ST^{\ast} S^{3}$ with regular boundary $\Lambda_{K}$ and inward normals $\pm R$ as follows. Consider the Reeb flow lines parameterized by $[0,\epsilon]$ starting at $(q,p)\in\Lambda_{K}$. These flow lines project to geodesic arcs of length $\epsilon$ in $S^{3}$ starting at $q\in K$ and perpendicular to $K$ at the start point. The union of the flow lines is an embedded copy $F_{\epsilon}$ of $[0,\epsilon]\times\Lambda_{K}$ in $ST^{\ast} S^{3}$. Consider its boundary component $\partial_{\epsilon} F$ corresponding to the flow line endpoints at $\epsilon$. The union of tori $\partial_{\epsilon} F$ consists of the union of the lifts of the boundary circles of the  geodesic disks perpendicular to $K$ at $q\in K$. The lift of such a boundary circle projects to a curve in the unit cotangent fiber at $q$ which is close to the great circle perpendicular to the tangent vector of $K$. The great circle bounds two hemispheres, one containing the unit tangent of the knot and one containing its negative. Using these hemispheres over every point in $K$, we construct two unions of solid tori $G_{\pm}$ filling $\partial_{\epsilon}F$. Define 
\[ 
C_{K;\pm}^{\infty}= F_{\epsilon}\cup G_{\pm}.
\]

To further explain this construction we consider the following local model which approximates any knot up to first order. If $K$ corresponds to the $x_{1}$-axis in a neighborhood of $0\in \R^{3}$ with coordinates $(x_{1},x_{2},x_{3})$ then $F_{\epsilon}$, with coordinates $(x,y)=(x_{1},x_{2},x_{3},y_{1},y_{2},y_{3})$ in $T^{\ast}S^{3}$, is the subset
\[ 
F_{\epsilon}=\left\{(x,y)=(x_{1},s\cos\theta,s\sin\theta,0,\cos\theta,\sin\theta)\right\},\quad 0\le s\le\epsilon,\;0\le\theta\le 2\pi.
\]
The subsets $G_{\pm}$ are then
\begin{align*} 
G_{\pm}=&\left\{(x,y)=(x_{1},\rho\epsilon\cos\theta,\rho\epsilon\sin\theta,\pm\sqrt{1-\rho^{2}},\rho\cos\theta,\rho\sin\theta)\right\},\\
&\quad 0\le \rho\le 1,\;0\le \theta\le 2\pi.
\end{align*}

\begin{lma}
If $C_{K;\pm}^{\infty}$ is the 3-chain constructed above then $C_{K;\pm}^{\infty}$ is a smooth 3-chain with regular boundary $L_{K}$ and inward normal $\pm R$. Furthermore $C_{K,\pm}^{\infty}$ meets $L_{K}$ along the boundary only.
\end{lma}

\begin{proof}
Immediate from the construction.
\end{proof}

Let $v$ be the vector field defined outside the critical points of $f$ by $v(q)=\frac{1}{|\nabla f|}\nabla f$, $q\in L_{K}-\{\kappa_{j}^{0},\kappa_{j}^{1}\}_{j=1}^{k}$. Consider the subsets of $N(L_{K})$ given by the closures $G_{\pm v}$ of
\[ 
G_{\pm v}^{0}=\{(q,\pm tJv)\in N(L_{K})\colon 0\le t\le 1\}.
\]
Let $G=G_{v}\cup G_{-v}$. 

\begin{lma}
The union $G$ is a regular chain with boundary $2\cdot L_{K}\cup G'$ and inward normals $\pm J\nabla f$ along $2\cdot L_{K}$. Furthermore, if $G'$ denotes the boundary component of $G$ that does not lie in $L_K$, then $G'$
 is an embedding of two copies of $L_{K}$ joined by two $1$-handles for each component of $K$, and $G'\cup C_{K}^{\infty}$ is null homologous in $X-L_{K}$. 
\end{lma}

\begin{proof}
Note that the outer boundary of $G_{\pm v}$ consists of 
\[ 
G_{\pm v}=\left\{(q,\pm Jv)\in N(L_{K})\right\}\cup \bigcup_{j=1}^{m}D_{\kappa_{j}^{0}}\cup D_{\kappa_{j}^{1}},
\] 
where $D_{\kappa_{j}^{\sigma}}$ denotes the fiber disk at $\kappa_{j}^{\sigma}$.
Here the $D_{\kappa_{j}^{\sigma}}$'s come with the orientation determined by $\pm v$ and hence cancel in the boundary of the union $G=G_{v}\cup G_{-v}$ since $\dim(D_{\kappa_{j}^{\sigma}})=3$ is odd. The statement on the boundary of $G$ follows. For the last statement, we simply note that both $G'$ and $C_{K}^{\infty}$ intersect the Poincar\'e duals $P_{j}$ of the generator of the relavant homology group, see Lemma \ref{l:homologydual}, once and that the intersections cancel.  
\end{proof}

We next define $C_{K}$ as follows. Fix a $4$-chain $C_{K}^{0}$ in $X-L_K$ with boundary 
\[ 
\partial C_{K}^{0}=G'\cup C_{K}^{\infty}
\]
and let
\[
C_{K}= C_{K}^{\infty}\times[0,\infty) \cup C_{K}^{0}\cup G.
\]
The above lemmas then show that $C_{K}$ is a $4$-chain with regular boundary along $2\cdot L_{K}$ and inward normal $\pm J\nabla f$, and, furthermore, that $C_{K}$ intersects $L_{K}$ only along its boundary and is otherwise disjoint from it. See Figure~\ref{fig:shade}.

\begin{figure}
\labellist
\small\hair 2pt
\pinlabel ${\color{red} C_K}$ at 110 66
\pinlabel $L_K$ at 110 47
\pinlabel ${\color{blue} C_K}$ at 110 28
\pinlabel ${\color{red} J \nabla_f}$ at 42 60
\pinlabel $\nabla_f$ at 63 39
\pinlabel ${\color{blue} -J \nabla_f}$ at 40 19
\endlabellist
\centering
\includegraphics[width=.4\linewidth]{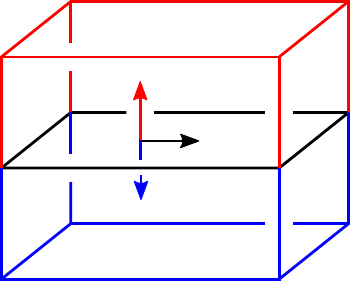}
\caption{The bounding chain $C_K$ near its boundary $L_K$.}
\label{fig:shade}
\end{figure}
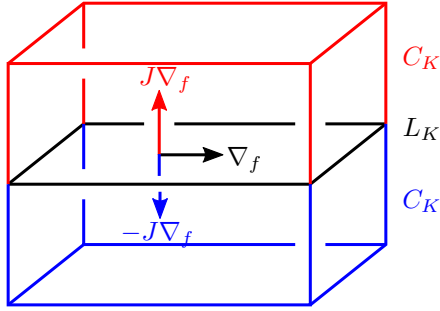

\subsubsection{Capping disks and general position with respect to trivial strips}
Our constructions of bounding chains below use general position. For this reason we need to alter the function $f$ and the 4-chain $C_{K}$ above slightly. The problem is that the chains constructed are not disjoint from Reeb chord holomorphic strips. Furthermore, as we shall see we will count intersections of the holomorphic curves with the bounding chains and combine with certain linking and self-linking numbers. For this to work, our perturbation of the bounding chains needs to be connected to our choice of capping paths in $\Lambda_{K}$, which we discuss next.

Fix a base point in each component of $\Lambda_{K}$, and for each Reeb chord endpoint, fix a real analytic arc connecting it to the base point. In the case where $K$ has multiple components, we also fix a path joining all of the base points together; for any pair of base points, some subset of this path will join that pair.

Assume that the derivative of the path at the Reeb chord endpoint is distinct from the local stable and unstable manifolds (i.e., the directions corresponding to the two K\"ahler angles).  
For each Reeb chord, consider the loop consisting of the Reeb chord, the paths between Reeb chord endpoints and the base points, and the path joining the two base points (if the endpoints of the Reeb chord lie on different components); then fix a 2-chain whose boundary is this loop.
We take this chain to be holomorphic along the boundary, i.e., to agree with the appropriate half of the complexification of the real analytic arc near the boundary.

Let $I$ be an index set enumerating all the Reeb chord endpoints. For each $i\in I$, fix a function $g_i$ supported in a small ball around the Reeb chord endpoint such that $\nabla g_i= v_i$, where $v_{i}$ is a nonzero vector in the contact plane at the endpoint. Rename the function $f$ considered above $\tilde f$ and let
\[
f= \tilde f + \epsilon \sum_{I} g_i.
\]
Now let the $4$-chain $C_{K}$ instead start out along $\pm J\nabla f$.

\begin{lemma}
Let $c$ be a Reeb chord of $\Lambda_K$; then there is a uniform neighborhood $\R\times U$ of the boundary in the trivial Reeb chord strip $\R\times c$ such that the interior of $\R\times U$ is disjoint from $C_K$.
\end{lemma}

\begin{proof}
Near the Reeb chord endpoint $p$ the strip looks like $p+s R$ whereas the chain looks like $p+s(R+J\nabla g_i)$, for small $s$, and $\nabla g_i=v_i\ne 0$ which lies in the contact plane (and hence so does $J\nabla g_i$).
\end{proof}

\subsection{Bounding chains for holomorphic curves}\label{sec:boundingchains}
We next associate a bounding chain to each holomorphic curve $u\colon (\Sigma,\partial\Sigma)\to (X,L_{K})$. Note that we allow $u$ to have positive punctures at Reeb chords. A bounding chain in this setting is a non-compact 2-chain $\sigma_{u}$ in $L_{K}$ such that $\partial \sigma_{u}$ equals $u(\partial\Sigma)$ completed by capping paths and such that the ideal boundary of $\sigma_{u}$ is a curve in $\Lambda_{K}$ that is homologous to a multiple of the longitude, i.e.~homologous to $\sum_{j} n_{j}x_{j}+m_{j}p_{j}$, where $m_{j}=0$, for all $j$. We define it as follows.

Consider first the case of a holomorphic curve $u\colon(\Sigma,\partial\Sigma)\to (X,L_{K})$ without punctures. Then its boundary $u(\partial\Sigma)$ is a collection of closed curves contained in a compact subset of $L_{K}$. By general position, $u(\partial\Sigma)$ does not intersect the stable manifold of the index $1$ critical points $\kappa_{j}^{1}$ of $f$. Let $\sigma'_u$ denote the set of all flow lines of $\nabla f$ starting on $u(\partial\Sigma)$. Then since $f$ has no index $2$ critical points and since $\nabla f$ is vertical (except for small disks around the Reeb chord endpoints) outside a compact set, we find that $\sigma'_u\cap(T\times\Lambda_{K})$ is a closed curve, independent of $T$ for all sufficiently large $T>0$. Write $\partial_{\infty}\sigma'_u\subset\Lambda_{K}$ for this curve.  

View $\sigma_{u}'$ as a non-compact 2-chain with $\partial\sigma_u'=u(\partial\Sigma)$ and with $\partial_{\infty}\sigma'_u$ a curve in the homology class $n\cdot x + m\cdot p$ in $H_{1}(\Lambda_{K})$. Let $W^{\mathrm{u}}(\kappa_{j}^{1})$ denote the unstable manifold of $\kappa^{1}_{j}$. Note that $W^{\mathrm{u}}(\kappa^{1}_{j})$ is a 2-disk, with ideal boundary representing the meridian class $p_{j}$. Define
\begin{equation}\label{eq:defboundingchain1}
\sigma_{u}=\sigma_{u}'-\sum_{j=1}^{k}m_{j} W^{\mathrm{u}}(\kappa^{1}_{j}),
\end{equation}
where $m=(m_{1},\dots,m_{k})$. Then $\sigma_{u}$ has the desired properties.

Consider next the general case when $u\colon(\Sigma,\partial\Sigma)\to(X,L_{K})$ has punctures at Reeb chords $c_{1},\dots, c_{m}$. Let $\delta_{j}$ denote the capping disk of $c_{j}$. 
The main difference in this case is that the boundary $u(\partial\Sigma)$ is not a closed curve. We use the capping disks to close it up as follows. Fix a sufficiently large $T>0$ and replace $u(\partial\Sigma)$ in the construction of $\sigma_u'$ above by the boundary of the chain
\[ 
u(\Sigma)\cap \left(X-([T,\infty)\times ST^{\ast}S^{3})\right) \ \cup \ \bigcup_{j=1}^{m}\delta_{j}
\]
and then proceed as there. This means that we cap off the holomorphic curve, keeping it holomorphic along the boundary by adding capping disks, and construct a bounding chain of this capped disk. Here we take $T$ sufficiently large so that the region where the disk is altered lies in the end where the whole family of curves in which the curve under consideration lives are close to Reeb chords.

\subsection{Generalized holomorphic curves, moduli spaces, and the SFT-potential}\label{sec:gencurves}
In this subsection we define generalized holomorphic curves; counts of these generalized curves will give the Gromov--Witten and SFT potentials.
We first consider the definition of certain linking numbers that will be used in defining generalized holomorphic curves. 

\subsubsection{Linking numbers of holomorphic curves}
Let $u$ and $v$ denote two distinct  holomorphic curves with bounding chains $\sigma_u$ and $\sigma_v$ as defined above in Section~\ref{sec:boundingchains}. Then the boundaries of $u$ and $v$ are oriented curves $\partial u$ and $\partial v$ in $L_K$.
Define the \emph{linking number} $\lk(u,v)$ as the following intersection number:
\[
\lk(u,v):=[\partial u]\cdot \sigma_v = [\partial v]\cdot\sigma_u. 
\] 
To see the second equality note that $\sigma_u\cap \sigma_v$ is an oriented 1-chain interpolating between $(\partial u)\cap \sigma_v$, $(\partial v)\cap \sigma_u$, and the intersection $(\partial_{\infty} u)\cap (\partial_{\infty}v)$ in $\Lambda_{K}$. Since the intersection number at infinity is zero by construction, it follows that $[\partial u]\cdot \sigma_v = [\partial v]\cdot\sigma_u$.

In order to define a similar self-linking number $\slk(u,u)$ between $u$ and itself, we pick a normal vector field $\nu$ along $\partial u$ in general position with respect to $\nabla f$. Let $\partial u_{\nu}$ denote the curve $\partial u$ shifted slightly along $\nu$. We then shift a neighborhood of $\partial u$ in $u$ along a small extension of the vector field $J\nu$. Then the shifted version $u_{J\nu}$ of $u$ is a 2-chain transverse to the 4-chain $C_{K}$, and $C_{K}$ and $u_{J\nu}$ have disjoint boundaries. Define the self linking number as
\[
\slk(u,u) = [\partial u_{\nu}]\cdot \sigma_u+[u_{J\nu}]\cdot C_{K}.
\]
Note that $\slk(u,u)$ is independent of the choice of $\nu$: if we change $\nu$ by a full twist then first term on the right hand side changes by $\pm 1$ and the second by $\mp 1$.

\subsubsection{Moduli spaces}
We define the interior of the moduli spaces that we will use. As the curves we consider may well be multiply covered, the actual definition of the moduli spaces requires the use of abstract perturbations. As mentioned previously, we will not discuss the details of the abstract perturbation scheme used but merely give an outline.  

We build the perturbation by induction on energy and Euler characteristic. The energy concept we use is the Hofer energy. Starting at the lowest energy level and Euler characteristic, we make all holomorphic curves transversely cut out and transvere with respect to the Morse data fixed. We also fix appropriate shifting vector fields for the curves. As we increase the energy level we keep the curves transverse to the Morse data as well as to curves constructed in earlier steps of the construction. More precisely, a holomorphic curve in general position has tangent vector independent from $\nabla f$ along its boundary and we take the shifting vector field $\nu$ to be everywhere independent to the normal vector field along the boundary defined by $\nabla f$, in such a way that $\nabla f$, $\nu$, and the tangent vector of the boundary form a positively oriented frame. Finally we assume that the $J\nu$-shifted holomorphic curve is transverse to $C_{K}$.

Elements in the moduli spaces we use consist
of the following data: 
\begin{itemize}
\item Begin with a finite oriented graph $\Gamma$ with vertex set $V(\Gamma)$ and edge set $E(\Gamma)$. 
\item To each $v\in V(\Gamma)$ is associated a (generic) holomorphic curve $u^v$ with boundary on $L_{K}$ (and possibly with positive punctures). 
\item To each edge $e\in E(\Gamma)$ that has its endpoints at distinct vertices, $\partial e=v_+-v_-$, $v_+\ne v_-$, is associated an intersection point of the boundary curve $\partial u_{v_{-}}$ and the bounding chain $\sigma_{u_{v_+}}$. 
\item
To each edge $e\in E(\Gamma)$ which has its endpoints at the same vertex $v_0$, $\partial e=v_0-v_0=0$, is associated either an intersection point in $\partial(u^{v_{0}})_{\nu}\cap \sigma_{u^{v_{0}}}$  or an intersection point in $(u^{v_{0}})_{J\nu}\cap C_{K}$.  
\end{itemize}
We call such a configuration a \emph{generalized holomorphic curve over $\Gamma$} and denote it $\Gamma_{\mathbf{u}}$, where $\mathbf{u}=\{u^{v}\}_{v\in V(\Gamma)}$ lists the curves at the vertices.

\begin{rmk}
Several edges of a generalized holomorphic curve may have the same intersection point associated to them.  
\end{rmk}

\begin{rmk}
Note that the convention for edges with endpoints at distinct vertices depends on the interplay between the shifting vector field for multiple copies of a given curve and the choice of obstruction chains. Our choice guarantees that there are no contributions to the linking number close to the boundary of a curve.
\end{rmk}

We define the Euler characteristic of a generalized holomorphic curve $\Gamma_{\mathbf{u}}$ as
\[
\chi(\Gamma_{\mathbf{u}})= \sum_{v\in V(\Gamma)} \chi(u_v) -\# E(\Gamma),
\]
where $\# E(\Gamma)$ denotes the number of edges of $\Gamma$, and the dimension of the moduli space containing $\Gamma_{\mathbf{u}}$ as 
\[
\dim(\Gamma_{\mathbf{u}})=\sum_{v\in V(\Gamma)} \dim(u^v),
\]
where $\dim(u^{v})$ is the formal dimension of $u^{v}$.
In particular, if $\dim(\Gamma_{\mathbf{u}})=0$ then $u^{v}$ is rigid for all $v\in V(\Gamma)$ and if $\dim(\Gamma_{\mathbf{u}})=1$ then $\dim(u^{v})=1$ for exactly one $v\in V(\Gamma)$ and $u^v$ is rigid for all other $v\in V(\Gamma)$.

As usual our moduli spaces are branched oriented orbifolds. In fact we will consider only moduli spaces of dimension 0 and 1 and hence we can think of them as branched manifolds rather than orbifolds. The weight of the moduli space at $\Gamma_{\mathbf{u}}$ is 
\[
w(\Gamma_{\mathbf{u}})=\frac{1}{N(\Gamma)}\,2^{-|E(\Gamma)|}\,\prod_{v\in V(\Gamma)} w(u^{v}),
\] 
where $w(u^{v})$ is the weight of the usual moduli space at $u^{v}$ and where $N$ is a symmetry factor coming from exchanging identical edges and vertices. We orient the moduli space using the product of the orientations over the vertices and the intersection signs at the edges.

The relative homology class represented by $\Gamma_{\mathbf{u}}$ is the sum of the homology classes of the curves $u^{v}$ at its vertices, $v\in V(\Gamma)$. 

\subsection{The SFT-potential}\label{sec:sftpotential}
We define the SFT-potential to be the generating function of rigid generalized curves (over graphs $\Gamma$) as just described:
\[
\mathbf{F}_{K} = \sum_{n,r,\mathbf{c^{+}}} F_{K;n,r,\chi,\mathbf{c}^{+}}\, g_{s}^{-\chi+\ell(\mathbf{c}^{+})}\, e^{n\cdot x} Q^{r}\, \mathbf{c}^{+},
\]
where $F_{K;n,r,\chi,\mathbf{c}^{+}}$ counts the algebraic number of generalized curves $\Gamma_{\mathbf{u}}$ in homology class $n\cdot x+rt\in H_{2}(X,L_{K})$ with $\chi(\Gamma_{\mathbf{u}})=\chi$ and with positive punctures according to the Reeb chord word $\mathbf{c}^{+}$.

We will sometimes use the decomposition of $\mathbf{F}_{K}$ according to the number of positive punctures:
\[ 
\mathbf{F}_{K} = \mathbf{F}_{K}^{0} + \mathbf{F}_{K}^{1} + \dots,
\]
where $\mathbf{F}^{j}_{K}$ counts the curves with $j$ positive punctures. We then define the open Gromov--Witten potential of $L_K$ to be 
the constant term $\mathbf{F}^{0}_{K}$, which counts configurations without positive punctures.
In particular, the wave function that counts disconnected curves without positive punctures is
\[ 
\Psi_{K}= e^{\mathbf{F}_{K}^{0}}.
\]

For computational purposes we next note that we can rewrite the sum for $\mathbf{F}_{K}$ in the following way. Instead of the complicated oriented graphs with many edges considered above, we look at unoriented graphs with at most one edge connecting every pair of distinct vertices and no edge connecting a vertex to itself. We call such graphs \emph{simple graphs}.

As before we have rigid curves at the vertices of our graphs. Above we defined the linking number between distinct holomorphic curves, $\lk(u_0,u_1)$, and the self linking number of one holomorphic curve, $\slk(u)$. 
If $e$ is an edge in a simple graph $\Delta$ with distinct endpoints $v_{0}$ and $v_{1}$, we define
\[
\lk(e)=\lk(u^{v_{0}},u^{v_{1}}),
\]
where $u^{v_{0}}$ and $u^{v_{1}}$ are the curves at the vertices at the end points of $e$.
If $v$ is a vertex in a simple graph $\Delta$ we define
\[
\slk(v)=\slk(u^v),
\] 
where $u^{v}$ is the curve at $v$.
With these definitions, we now weight each simple graph by a $g_{s}$-dependent weight:
\[
W(\Delta)=\frac{1}{N(\Delta)}\prod_{v\in V(\Delta)} w(u^v)\, g_{s}^{-\chi(u^v)}e^{\tfrac12g_s\slk(v)}\prod_{e\in E(\Delta)}\left(e^{g_s\lk(e)}-1\right),
\] 
where $N(\Delta)$ is a symmetry factor coming from exchanging identical vertices.
Furthermore, we can assign a sign to each simple graph $\Delta$ given by the product of the orientations of the curves at its vertices.

We then get a simplified formula for the SFT-potential:
\[
\mathbf{F}_{K}= \sum_{n,k} G_{n,r}\, e^{n\cdot x} Q^{r}\mathbf{c}^{+},
\]
where $G_{n,r}$ is the algebraic sum of $g_s$-dependent weights of all simple graphs in homology class $n\cdot x+nt$ with positive punctures according to $\mathbf{c}^{+}$. This follows from counting the contributions to the potential lying over a given simple graph, where we project from a complicated graph to a simple one by identifying all edges with the same pair of distinct endpoints and by deleting all edges with endpoints at the same vertex.

\section{Compactification of 1-dimensional moduli spaces and the SFT-equation}
The generalized holomorphic curves that we defined in Section \ref{sec:gencurves} constitute the  open strata of the 1-dimensional moduli spaces that underlie the SFT-equation. Such curves correspond to graphs $\Gamma$ with a generic curve of dimension 1 at exactly one vertex. Besides the usual holomorphic degenerations in 1-parameter families, there are new boundary phenomena arising from the 1-dimensional curve becoming non-generic. In this section we study this and argue that all degenerations, except for breaking at infinity, cancel out. This means that the boundary of each 1-dimensional moduli space corresponds to two-level curves only, which then leads to the SFT-equation.   

\subsection{Boundary phenomena in moduli spaces of dimension one}
Consider a generalized holomorphic curve $\Gamma_{\mathbf{u}}$ of dimension 1. We have the following boundary phenomena that come from degenerations of the holomorphic curves $u^{v}$ at the vertices $v\in V(\Gamma)$:  
\begin{itemize}
\item[$(1)$] Splitting at Reeb chords, see Figure \ref{fig:Reebsplit}.   
\item[$(2)$] Hyperbolic boundary splitting, see Figure \ref{fig:hyperbolic}.
\item[$(3)$] Elliptic boundary splitting, see Figure \ref{fig:elliptic}.  
\end{itemize}
In addition, since we require that the 1-dimensional curve in our graph is generic with respect to the auxiliary Morse function $f$ and the $4$-chain $C_{K}$, there are also the following degenerations: 
\begin{itemize}
\item[$(4)$] Crossing the stable manifold of $\kappa_{1}$: the boundary of the curve intersects the stable manifold of $\kappa_{1}$, see Figure \ref{fig:stablecross}.
\item[$(5)$] Boundary crossing: a point in the boundary mapping to a bounding chain moves out across the boundary of a bounding chain, see Figure \ref{fig:cross}.
\item[$(6)$] Interior crossing: An interior marked point mapping to $C_K$ moves across the boundary $L_{K}$ of $C_{K}$, see Figure \ref{fig:out}.
\item[$(7)$] Boundary kink: The boundary of a curve becomes tangent to $\nabla f$ at one point, see Figure \ref{fig:encomplex}.
\item[$(8)$] Interior kink: A marked point mapping to $C_K$ moves to the boundary in the holomorphic curve, see Figure \ref{fig:encomplex}.
\end{itemize}

\begin{figure}
\centering
\includegraphics[width=.4\linewidth]{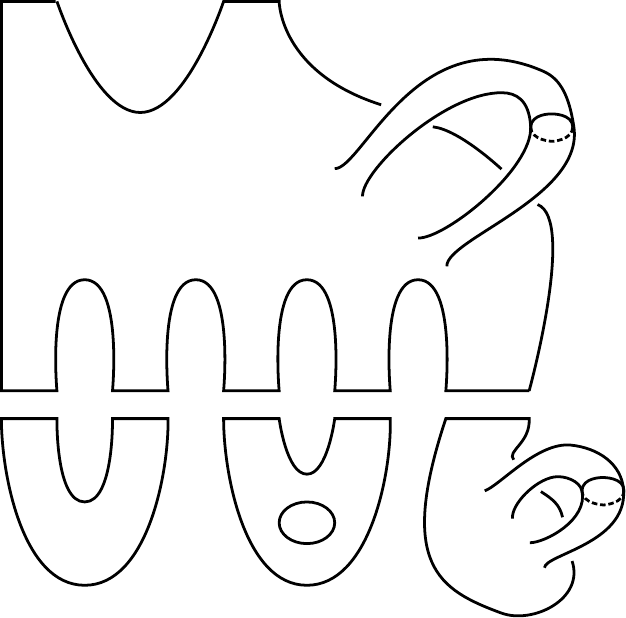}
\caption{Splitting at Reeb chords. The top part of the diagram is in the symplectization, the bottom in $X$.}
\label{fig:Reebsplit}
\end{figure}

\begin{figure}
\centering
\includegraphics[width=.5\linewidth]{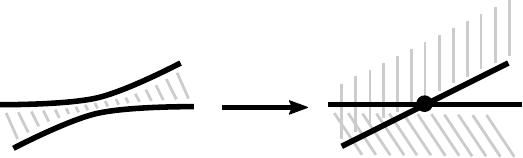}
\caption{Hyperbolic boundary splitting.}
\label{fig:hyperbolic}
\end{figure}

\begin{figure}
\labellist
\small\hair 2pt
\pinlabel ${\color{blue} L_K}$ at 128 33
\pinlabel ${\color{blue} L_K}$ at 336 33
\endlabellist
\centering
\includegraphics[width=.7\linewidth]{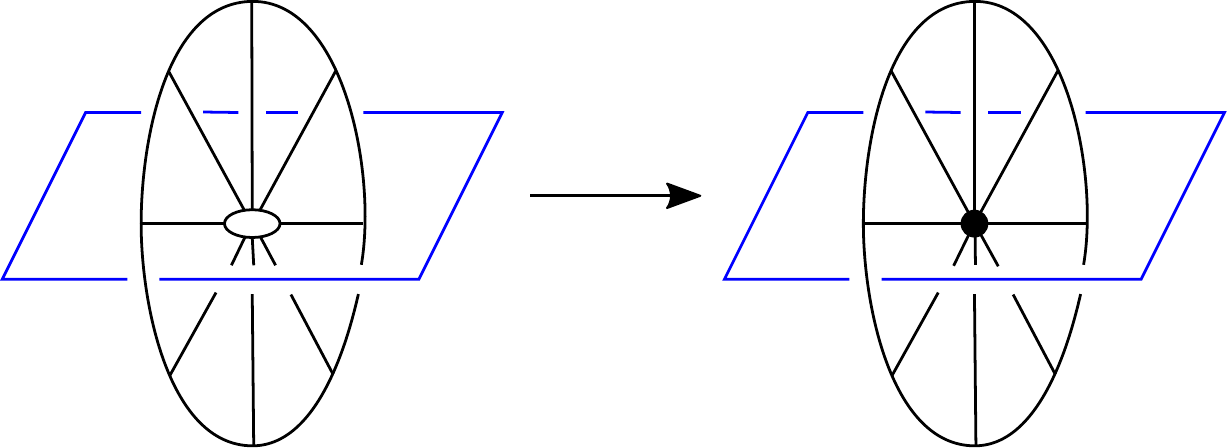}
\caption{Elliptic boundary splitting.}
\label{fig:elliptic}
\end{figure}

\begin{figure}
\labellist
\small\hair 2pt
\pinlabel ${\color{blue} \kappa^0}$ at 0 5
\pinlabel ${\color{blue} \kappa^1}$ at 93 60
\endlabellist
	\centering
	\includegraphics[width=.3\linewidth]{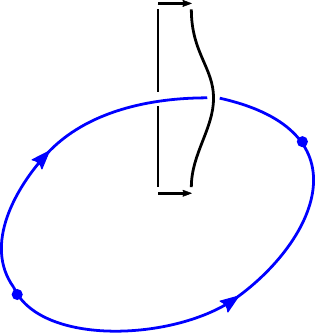}
	\caption{Crossing the stable manifold of $\kappa^{1}$.}
	\label{fig:stablecross}
\end{figure}

\begin{figure}
\centering
\includegraphics[width=.7\linewidth]{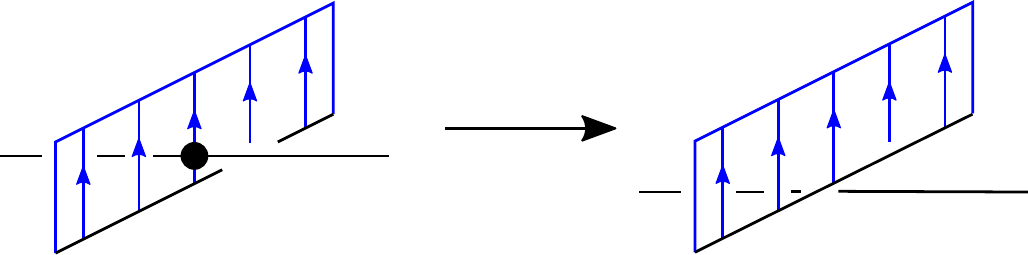}
\caption{Boundary crossing.}
\label{fig:cross}
\end{figure}

\begin{figure}
\labellist
\small\hair 2pt
\pinlabel ${\color{blue} C_K}$ at 124 53
\pinlabel ${\color{blue} C_K}$ at 284 53
\pinlabel ${\color{red} u}$ at 58 33
\pinlabel ${\color{red} u}$ at 219 1
\pinlabel $L_K$ at 109 9
\pinlabel $L_K$ at 265 9
\endlabellist
\centering
\includegraphics[width=.8\linewidth]{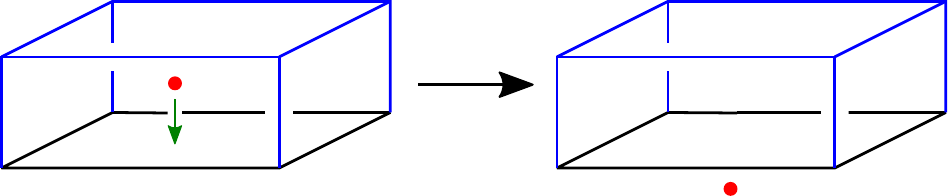}
\caption{Interior crossing.}
\label{fig:out}
\end{figure}

\begin{figure}
\labellist
\small\hair 2pt
\pinlabel ${\color{blue} \nabla f}$ at -2 30
\pinlabel ${\color{blue} u \cap C_K}$ at 178 60
\pinlabel $\partial u$ at 142 60
\pinlabel $\partial u$ at 65 49
\endlabellist
\centering
\includegraphics[width=.5\linewidth,]{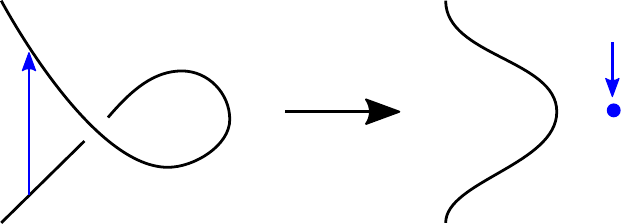}
\caption{From left to right: boundary kink; from right to left: interior kink.}
\label{fig:encomplex}
\end{figure}

\noindent
We must also consider boundary phenomena near the Reeb chord endpoints where we have fixed capping paths:
\begin{itemize}
	\item[$(9)$] The leading Fourier coefficient at a positive puncture vanishes.
\end{itemize}

\begin{lma}\label{l:codim1walls}
For generic data, $(1)-(9)$ is the complete list of degenerations in 1-parameter families of generalized curves.
\end{lma}

\begin{proof}
Codimension one degenerations of holomorphic curves with boundary in a Lagrangian are well-known and correspond to $(1)-(3)$. 

Consider the boundary of the curve. For generic data this is a 1-parameter family in general position with respect to the gradient vector field of the Morse function $f$. The corresponding degenerations are $(4)$ and $(7)$. Also, the family of boundary curves is generic as a family of smooth curves in $L_{K}$. The corresponding degeneration is $(5)$.

Next consider the interior of the curve. We have a family of surfaces with boundary in general position with respect to $C_{K}$ and its boundary $L_{K}$. The corresponding degenerations are then $(6)$ and $(8)$.

Finally, we must consider general position with respect to capping paths. Near the capping path endpoint the curve admits a Fourier expansion and in a generic 1-parameter family, degenerations correspond to transverse vanishing of the Fourier coefficient in the direction of leading asymptotics. The corresponding degeneration is $(9)$. 
\end{proof}

\subsection{Invariance in 1-parameter families}
In this subsection we argue that the generating function for generalized holomorphic curves at generic instances is independent of the particular instant. This also leads to invariance of the open Gromov--Witten potential. More precisely, we aim to justify the following result (see however Remark~\ref{rmk:parallelchains} for a discussion of a missing piece of the argument).

\begin{thm}\label{t:mainmodulispace}
$\quad$
\begin{itemize}
\item[$(a)$]
Let $\mathbf{c}^{+}$ be a word of Reeb chords of total grading 1. Let $\mathcal{M}(\mathbf{c}^{+})$ denote the moduli space of generalized holomorphic curves in $X$ with boundary on $L_{K}$, with positive punctures at $\mathbf{c}^{+}$. Then $\mathcal{M}(\mathbf{c}^{+})$ is a weighted branched oriented 1-manifold with boundary given by the moduli space of two-level generalized holomorphic curves of the following form: one $\R$-invariant family of generalized curves in the symplectization $\R\times ST^{\ast}S^{3}$, along with rigid generalized holomorphic curves in $X$ attached at Reeb chords and at bounding chains.
\item[$(b)$]
Let $s\in I$ be a generic 1-parameter family of perturbations for holomorphic curves in $(X,L_{K})$. Let 
\[ 
\mathbf{F}_{K}(s)=\sum_{j=0}^{\infty} \mathbf{F}^{j}_{K}(s)
\]
denote the generating function for generalized holomorphic curves in $(X,L_{K})$. Then $\mathbf{F}_{K}(s)$ is independent of $s\in I$.
\end{itemize}
\end{thm} 

\begin{proof}
Lemma \ref{l:codim1walls} implies that it suffices to show that the boundary degenerations $(2)-(9)$ cancel out. We show this below in a sequence of lemmas that together then establish the theorem. 
\end{proof}

We next consider the lemmas needed to demonstrate the invariance result in Theorem \ref{t:mainmodulispace}. We first consider the degeneration $(4)$:

\begin{lma}\label{l:crossunstable}
The moduli space of generalized holomorphic curves does not change under degeneration $(4)$, i.e., when the boundary of a holomorphic curve crosses the unstable manifold of $\kappa^{1}_{j}$.
\end{lma}

\begin{proof}
Recall the definition \eqref{eq:defboundingchain1} of the bounding chain of a generic holomorphic curve $u$:
\[ 
\sigma_{u}=\sigma'_{u} - m\cdot W^{\mathrm{u}}(\kappa^{1}),
\]
where $\sigma'_{u}$ is the chain of flow lines of $\nabla f$ starting on $\pa u$, and where this chain intersects the torus at infinity in a curve $\gamma$ of homology class $n\cdot x + m\cdot p$. Note that as $\pa u$ crosses $W^{\mathrm{u}}(\kappa^{1}_{j})$, $\sigma'_{u}$ changes by $\pm W^{\mathrm{u}}(\kappa^{1}_{j})$ and $m_{j}$ changes by $\pm 1$. These two changes cancel out in $\sigma_{u}$, leaving the bounding chain and hence the moduli space of generalized holomorphic curves unchanged. 
\end{proof}

We next consider the case of hyperbolic boundary splitting $(2)$, which cancels with boundary crossing $(5)$. The next two results, Lemmas~\ref{l:boundarysplitandcross} and~\ref{l:elliptic}, depend on the fine points of the perturbation scheme we use; accordingly, the arguments given here should be considered as outlines rather than complete proofs, see Remark \ref{rmk:parallelchains}.  

\begin{lma}\label{l:boundarysplitandcross}
A curve with a boundary node in a generic 1-parameter family appears both as a boundary splitting $(2)$ and as a boundary crossing $(5)$. The moduli space of generalized holomorphic curves gives a cobordism between the moduli spaces before and after the instant with the singular curve. 
\end{lma}

\begin{proof}
At the hyperbolic boundary splitting we find a holomorphic curve with a double point that can be resolved in two ways, $u_+$ and $u_-$. Consider the two moduli spaces corresponding to $m$ insertions at the corresponding intersection points between $\pa u_+$ and $\sigma_{u_-}$ and between $\pa u_-$ and $\sigma_{u_{+}}$. 

To obtain transversality at this singular curve for any Euler characteristic we must separate the intersection points. To this end, we use an abstract perturbation scheme that time-orders the crossings. We then employ usual gluing at the now distinct crossings. 

Consider gluing at $m$ intersection points as $\pa u_{-}$ crosses $\sigma_{u_+}$. This gives a curve of Euler characteristic decreased by $m$ and orientation sign $\epsilon^{m}$, $\epsilon=\pm 1$. Furthermore, at the gluing, the ordering permutation acts on the gluing strips and each intersection point is weighted by $\frac12$ since we count pairs of intersections between boundaries and bounding chains twice (i.e., both $\partial u\cap \sigma_v$ and $\partial v\cap \sigma_u$ contribute). This gives a moduli space of additional weight
\[
\epsilon^{m}\frac{1}{2^{m}m!} g_{s}^{m}.
\]  
The only difference between these configurations and those associated with the opposite crossing is the orientation sign. Hence the other gluing when $\pa u_{+}$ crosses $\sigma_{u_{-}}$ gives the weight
\[
(-1)^{m}\epsilon^{m}\frac{1}{2^{m}m!} g_{s}^{m}.
\]

Noting that the original moduli space is oriented towards the crossing for one configuration and away from it for the other we find that the two gluings cancel if $m$ is even and give a new curve of Euler characteristic decreased by $m$ and of weight $\frac{2}{2^{m}m!}$ if $m$ is odd. The two resulting moduli spaces without boundary are depicted in Figure \ref{fig:multiplegluing}. Counting ends of moduli spaces we find that the curves resulting from gluing at the crossing count with a factor
$e^{\frac12 g_{s}}-e^{-\frac12 g_{s}}$.
\end{proof}

\begin{figure}
\labellist
\small\hair 2pt
\pinlabel $1$ at 39 87
\pinlabel $\frac{1}{2}$ at 18 34
\pinlabel $\frac{1}{2}$ at 43 34
\pinlabel {$m$ odd} at 31 6
\pinlabel {$m$ even} at 117 6
\endlabellist
	\centering
	\includegraphics[width=.6\linewidth]{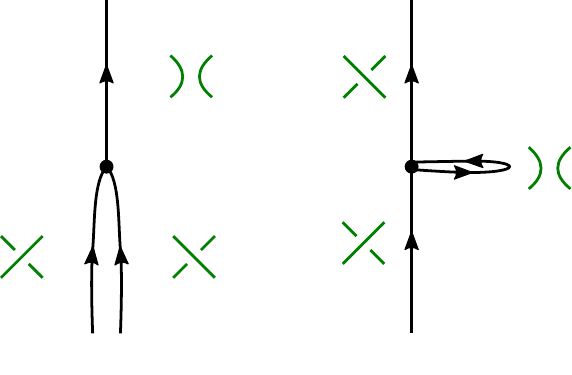}
	\caption{1-dimensional moduli spaces near boundary crossings/boundary splittings.}
	\label{fig:multiplegluing}
\end{figure}
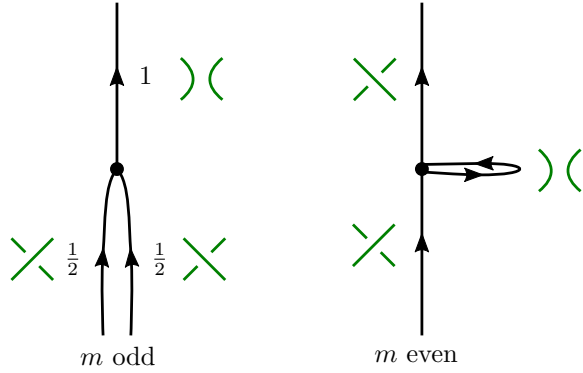

Next we consider the case of elliptic boundary splitting $(3)$, which cancels with boundary crossing $(6)$.   

\begin{lma}\label{l:elliptic}
	A curve which intersects $L_{K}$ in an interior point in a generic 1-parameter family appears both as an elliptic splitting $(3)$ and as a interior crossing $(6)$. The moduli space of generalized holomorphic curves gives a cobordism between the moduli spaces before and after the instant with the curve that intersects $L_{K}$. 
\end{lma}

\begin{proof}
The proof is similar to that of Lemma \ref{l:boundarysplitandcross}. The curve with an interior point mapping to $L_K$ can be resolved in two ways, one curve $u_+$ that intersects $C_{K}$ at a point in the direction $+J\nabla f$ and one $u_-$ that intersects $C_{K}$ at a point in the direction $-J\nabla f$. 

We also have a gluing problem: a constant disk at the intersection point can be glued to the family of curves at the intersection. As in the hyperbolic case, in order to get transversality at any Euler characteristic we must allow for this to happen many times. To that end we use an abstract perturbation that time orders $C_{K}$ and intersection points. We then apply usual gluing.
Since the intersection sign is part of the orientation data for the gluing problem, the calculation of weights is exactly as in the hyperbolic case above, where this time the $\frac12$-factors come from the boundary of $C_K$ being twice $L_K$, $\pa C_{K}=2[L_K]$. As there, we conclude that gluing corresponds to multiplication by
$e^{\frac12 g_{s}}-e^{-\frac12 g_{s}}$.
The same factor appears in the difference of counts when $u_{+}\cdot C_{K}$ is replaced with $u_{-}\cdot C_{K}$. The lemma follows.
\end{proof}

\begin{rmk}\label{rmk:parallelchains}
Lemmas \ref{l:boundarysplitandcross} and \ref{l:elliptic} use certain properties of the perturbation scheme for holomorphic curves. Except for usual general position properties we use time ordering of intersections to derive the contribution at the gluing. To complete the argument one would need to show that there actually exists such a perturbation scheme that also satisfies all the usual general position properties. 

From the point of view of \cite{ESh}, the above treatment can be understood as follows. In \cite{ESh} curve components of symplectic area zero are left unperturbed and only so called bare curves (curves with no components of symplectic area zero) are counted, but in a way that takes into account contributions from constant curves attached. The arguments above correspond to keeping constants unperturbed, turning a perturbation on near the degenerate instance, and then turning them back off.   
\end{rmk}

Next we consider tangencies $(7)$ of the boundary to the gradient vector field that cancel with an interior intersection with $C_{K}$ moving to the boundary $(8)$.
\begin{lma}\label{l:encomplex}
A curve with boundary tangent to $\nabla f$ has a degenerate intersection with $C_{K}$. The moduli space of generalized holomorphic curves gives a cobordism between the moduli spaces before and after the tangency instant. 
\end{lma}

\begin{proof}
Here the change does not involve gluing of holomorphic curves. It is simply exchanging an intersection in $\partial u_{\nu}\cap \sigma_u$ with one in $u_{J\nu}\cap C_K$. More precisely, pick orientation so that the curve right before the tangency moment has an intersection between $\partial u_{\nu}$ and $\sigma_{u}$ that disappears after the tangency. Then there is a corresponding intersection between $u_{J\nu}$ and $C_{K}$ born at the tangency moment, see Remark \ref{r:qubicmodel}. The contribution of both these configurations corresponds to multiplication by $e^{\epsilon \frac12 g_{s}}$, where $\epsilon=\pm1$. The lemma follows.  
\end{proof}

\begin{rmk}\label{r:qubicmodel}
To get a local model for Lemma \ref{l:encomplex} consider local coordinates 
\[ 
(z_{1},z_{2},z_{3})=(x_{1}+iy_{1},x_{2}+iy_{2},x_{3}+iy_{3})\in\C^{3}
\]
on $X$ with $L_{K}$ corresponding to $\R^{3}$. Assume that the gradient of $f$ is $\nabla f=\partial_{x_{3}}$. Then $C_{K}$ is locally given by $C_{K}=C_{K}^{+}+C_{K}^{-}$, where 
\[ 
C_{K}^{\pm}=\pm\{y_{2}=y_{3}=0\}.
\]
A generic family of holomorphic curves with a tangency with $\partial_{x_{3}}$ is given by the map $u_{\pm}\colon H\to\C^{3}$ ($H$ is the upper half plane),
\[ 
u_{\pm}(z)=(z^{2},z(z^{2}+s),\pm z).
\]
For $s<0$ the projection of $u|_{\partial H}$ to the $(x_{1},x_{2})$ has a double point at $z=\pm\sqrt{-s}$ that contributes to linking according to the sign of $u_{\pm}$. At $s=0$ the boundary has a tangency with $\partial_{x_{3}}$ and at $s>0$, $u_{\pm}(H)$ intersects $C_{K}^{\pm}$ at $u(i\sqrt{s})$ with the sign that agrees with the liking sign before the tangency.

\end{rmk}

We finally consider the degeneration $(9)$ when the holomorphic curve becomes tangent to the capping path. 

\begin{lma}\label{l:capping}
The moduli space of generalized holomorphic curves gives a cobordism between the moduli spaces before and after an instant where the Fourier coefficient of the leading asymptotic vanishes (corresponding to a tangency with a capping path). 
\end{lma}

\begin{proof}
As in the proof of Lemma \ref{l:encomplex} there is no gluing of holomorphic disks involved. We show that the count remains invariant by a local calculation.

At moments of type $(9)$ there are two scenarios: either an intersection with the capping path disappears or not, depending on which quadrant the capping path lies in, see Figure \ref{fig:realcapcross}.

\begin{figure}
\labellist
\small\hair 2pt
\pinlabel {Leading asymptotic} at 96 182
\pinlabel {direction} at 96 175
\pinlabel {\color{red} Capping path} at 108 160
\pinlabel {\color{blue} Boundary of} at 87 132
\pinlabel {\color{blue} holomorphic curve} at 87 123
\pinlabel {\color{black} Subleading asymptotic} at 100 95
\pinlabel {direction} at 100 86
\endlabellist
	\centering
	\includegraphics[width=.7\linewidth]{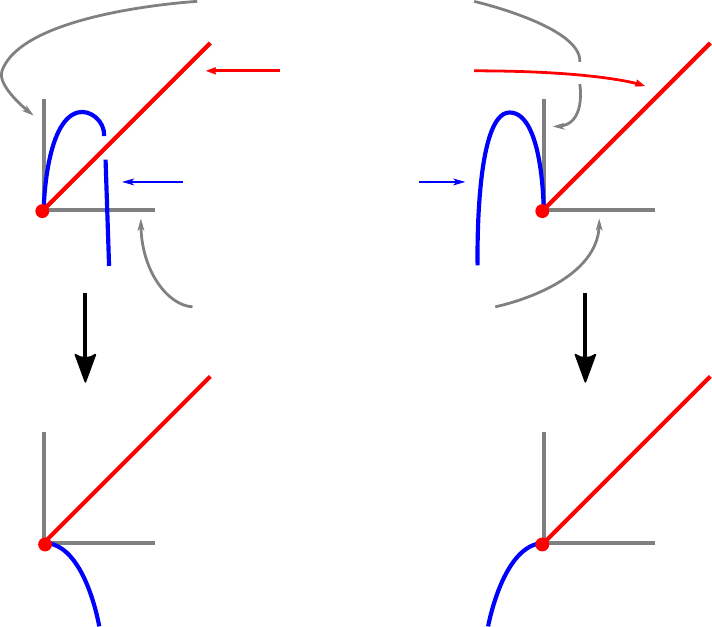}
	\caption{The crossings at a Reeb chord endpoint. The change or invariance of the real crossing are compensated by the change or invariance of imaginary crossings.}
	\label{fig:realcapcross}
\end{figure}
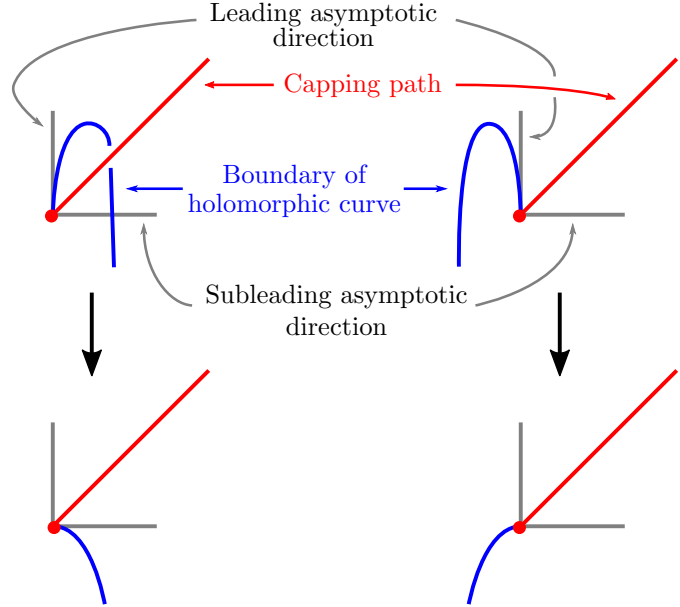

We have a similar boundary phenomenon for interior intersections that cross the boundary. To see this we carry out the calculation in coordinates adapted to the leading and sub-leading directions in Figure \ref{fig:realcapcross}. Up to exponentially small error we can write the holomorphic curve near the puncture as follows:
\[ 
s+it \mapsto \left(\begin{matrix}
b e^{\beta(s+it)} \\
\sigma e^{\alpha(s+it)} \\
c(s+it)
\end{matrix}\right), 
\]
where $\sigma\in (-\epsilon,\epsilon)$.
The imaginary part of the curve is thus given by
\[ 
\left(\begin{matrix}
b e^{\beta s}\sin t \\
\sigma e^{\alpha s}\sin t \\
ct
\end{matrix}\right).
\]
The $4$-chain filling $\R\times\Lambda_{K}$ is locally given by
\[ 
\R^{3}+i\lambda(\epsilon_{1},\epsilon_{2},1),
\]
where $\lambda$ is real and $\epsilon_{j}$ are small. Dividing the imaginary part by $t$ we find that the intersection pattern between the 4-chain and the curve is exactly as in Figure \ref{fig:realcapcross}.  
\end{proof}

Lemmas \ref{l:crossunstable}--\ref{l:capping} show that splitting into a two-level curve is effectively the only codimension one boundary for a 1-dimensional moduli space of generalized holomorphic curves. Combined, these establish Theorem \ref{t:mainmodulispace}.

\subsection{The SFT equation}
The above description of the boundary of 1-dimensional moduli spaces of generalized holomorphic curves leads to the SFT-equation \eqref{eq:master}. 

We let $\mathbf{H}_{K}$ denote the count of generalized rigid holomorphic curves $\Gamma_{\mathbf{u}}$ that appear in the upper level of a two-level curve in the boundary. Such a generalized curve lies over a graph that has a main vertex corresponding to a curve of dimension 1, which in this case is a curve that is rigid up to $\R$-translation; at all other vertices there are trivial Reeb chord strips.

Consider such a generalized holomorphic curve $\Gamma_{\mathbf{u}}$, rigid up to translation in the symplectization. We write $\mathbf{c}^{+}(\mathbf{u})$ and $\mathbf{c}^{-}(\mathbf{u})$ for the monomials of positive and negative punctures of $\Gamma_{\mathbf{u}}$, write $w(\mathbf{u})$ for the weight of $\Gamma_{\mathbf{u}}$, $n(\mathbf{u})\cdot x+m(\mathbf{u})\cdot p+r(\mathbf{u})t$ for its homology class, $\chi(\mathbf{u})$ for the Euler characteristic of the generalized curve of $\Gamma_{u}$. Define
\[
\mathbf{H}_{K}=\sum_{\dim(\Gamma_{\mathbf{u}})=1} w(u)\;  g_{s}^{-\chi(\mathbf{u})+\ell(\mathbf{c}^{+}(\mathbf{u}))}\; e^{n(\mathbf{u})\cdot x+m(\mathbf{u})\cdot p+r(\mathbf{u})t}\;\partial_{\mathbf{c}^{-}(\mathbf{u})}\mathbf{c}^{+}(\mathbf{u}),  
\]
where the sum ranges over all generalized holomorphic curves. As above this formula can be simplified to a sum over simpler graphs with more elaborate weights on edges. For example we can rewrite it as a sum over graphs $\Gamma_{\mathbf{u}}'$ without edges connecting the main vertex, corresponding to the 1-dimensional curve $u$ that contains the positive puncture of degree 1, as follows: 
\[
\mathbf{H}_{K}=\sum_{\dim(\Gamma_{\mathbf{u}}')=1} w(u)  g_{s}^{-\chi(\mathbf{u})+\ell(\mathbf{c}^{+}(\mathbf{u}))}\;e^{\frac12\slk(u)g_s}\; e^{n(\mathbf{u})\cdot x+m(\mathbf{u})\cdot p+r(\mathbf{u})t}\;\partial_{\mathbf{c}^{-}(\mathbf{u})}\mathbf{c}^{+}(\mathbf{u}),  
\]
where $\slk(u)$ for the self-linking number of the curve $u$ at the main vertex.

\begin{rmk}
If we require special properties of the perturbation scheme related to avoiding self-linking between trivial Reeb chord strips, this formula can likely be further simplified. We leave such matters to future studies and work out the relevant contributions here only in the examples we study. This problem is related to the algebraic problem of finding out how detailed a knowledge of the Hamiltionian is needed to extract the recursion relation. In the examples of the trefoil knot and the Hopf link only a small piece of the Hamiltonian is used.   
\end{rmk}

\begin{lma}\label{l:quantization}
Consider a curve $C$ at infinity in class $n\cdot x+m\cdot p+rt$. The count of the corresponding generalized curves with insertion of bounding cochains along $C$ equals
\[ 
e^{-\mathbf{F}_{K}} e^{n\cdot x}Q^{r} e^{m\cdot g_{s}\frac{\partial}{\partial x}}e^{\mathbf{F}_{K}},
\]
where $m\cdot\frac{\partial}{\partial x}=\sum_{j=1}^{k} m_{j}\frac{\partial}{\partial x_{j}}$.
\end{lma}
\begin{proof}
To see this note that contributions from bounding chains of curves inserted $k$ times along $m_{j}p_{j}$ corresponds to multiplication by
\[ 
m_{j}^{k}\frac{1}{k!}g_{s}^{-k}\sum_{k_{1}+\dots+k_{j}=k}
\frac{\partial^{k_{1}} \mathbf{F}_{K}}{\partial x_{j}^{k_{1}}}
\dots
\frac{\partial^{k_{j}} \mathbf{F}_{K}}{\partial x_{j}^{k_{j}}}.
\]	
Here a factor
$\frac{\partial^{s} \mathbf{F}_{K}}{\partial x_{j}^{s}}$ corresponds to attaching the bounding chain of a curve $s$ times. The lemma follows. 	
\end{proof}	

With this lemma established we obtain the SFT equation for conormals $L_{K}$ in the resolved conifold $X$. More precisely we have the following.

\begin{thm}\label{t:sftmastereq}
If $K$ is a link and $L_{K}\subset X$ its conormal Lagrangian then the SFT equation
\begin{equation}\label{eq:sft}
e^{-\mathbf{F}_{K}} \ \mathbf{H}_{K}|_{p_{j}=g_s\frac{\partial}{\partial x_{j}}} \ e^{\mathbf{F}_{K}} =0
\end{equation}
holds.
\end{thm}

\begin{proof}
Theorem \ref{t:mainmodulispace}(a) and Lemma \ref{l:quantization} show that the terms in the left hand side of \eqref{eq:sft} count the ends of an oriented branched 1-manifold. The theorem follows.
\end{proof}    
  
\begin{rmk}
We point out that counting insertions of bounding cochains gives an enumerative geometric meaning to the standard quantization scheme $p_{j}=g_{s} \frac{\partial}{\partial x_{j}}$. See \cite[Section 3.3]{Ekholmoverview} for a related path integral argument.
\end{rmk}

\subsection{Framing and Gromov--Witten invariants}
Theorem \ref{t:mainmodulispace} implies that the Gromov--Witten potential is independent of the data used to define the moduli space of generalized holomorphic curves up to homotopy (e.g., the potential does not depend on the specific choice of almost complex structure or perturbation). As mentioned previously, large $N$ duality predicts that if $K$ is a link $K=K_{1}\cup\dots\cup K_{k}$ and $F=\mathbf{F}_{K}^{0}(x,Q)$, $x=(x_{1},\dots, x_{k})$ denotes its Gromov--Witten potential then
\[ 
\Psi_{K}(x,Q)=e^{F(x,Q)}=\sum_{n=(n_{1},\dots,n_{k})} H_{K;n}(e^{g_{s}},Q)e^{n\cdot x},
\]
where $H_{K;n}$ is the (unnormalized) HOMFLY-PT polynomial with the component $K_{j}$ colored by the $m_{j}^{\rm th}$ symmetric representation. It is well-known that the colored HOMFLY-PT polynomial depends on framing. We derive this dependence here using our definition of generalized holomorphic curves. 

Assume that $\Psi_{K}$ above is defined for a framing $(x,p)=(x_{1},p_{1},\dots,x_{k},p_{k})$ of $\Lambda_{K}$. Then other framings are given by
\[ 
(x',p')=\bigl(x_{1}+r_{1}p_{1},p_{1}, \ \dots \ , x_{k}+r_{k}p_{k},p_{k}\bigr),
\] 
where $r=(r_{1},\dots,r_{k})$ is a vector of integers. Let $\Psi_{K}^{r}(x',Q)$ denote the wave function defined using the framing $(x',p')$.  
\begin{thm}\label{t:framing}
If $\Psi_{K}(x,Q)$ is as above then
\[ 
\Psi_{K}^{r}(x',Q)=\sum_{n=(n_{1},\dots,n_{k})} H_{K;m}(e^{g_{s}},Q)\,  e^{(\sum_{j=1}^{k} n_{j}^{2}r_{j})g_{s}} e^{n\cdot x'}.
\]
\end{thm}

\begin{proof}
Note first that the actual holomorphic curves are independent of the framing. In the perturbation scheme used, the change comes from correcting the boundaries at infinity $\partial_{\infty}\sigma_{u}$ to lie in the correct class. Following the perturbation scheme, this means that for a curve that goes $n_{j}$ times around the generator of $H_{1}(L_{K_{j}})$ we must correct the bounding chain by adding $n_{j}r_{j} W^{\rm u}(\kappa_{1}^{j})$. This means that the linking number in $L_{K_{j}}$ in this class changes by $n_{j}^{2}r_{j}$ which explains the factor $e^{(\sum_{j=1}^{k} n_{j}^{2}r_{j})g_{s}}$.
\end{proof}

\section{Recursive calculation of the open Gromov--Witten potential of the Lagrangian conormal}

In this section we show how to use Theorem \ref{t:mainmodulispace} to determine $\Psi_{K}$ by induction on the Euler characteristic. The inductive step is closely related to the tangent space of the augmentation variety expressed in terms of linearized contact homology. Although this induction is not useful in practice for computing the wave function, individual steps are interesting in themselves. For example, the first step in the recursion gives the annulus amplitude along the augmentation curve, which is the central ingredient in Eynard--Orantin topological recursion, see \cite{remodelB, eynardorantin}.    

\subsection{Regularity properties of the disk potential of the conormal}
Let $K$ be a link, $\Lambda_{K}\subset ST^{\ast}S^{3}$ its conormal Legendrian, and $L_{K}$ its conormal Lagrangian. We will think of $L_K$ either as a Lagrangian submanifold in $T^{\ast} S^{3}$ (when $Q=1$) or in the resolved conifold $X$ (when $Q\ne 1$).

Write $W_K=W_K(e^{x_{1}},\dots,e^{x_{k}},Q)$ for the disk potential of $L_K\subset X$. Our first result states that $W_K$ is an analytic function. (It is a priori not clear that the generating function for holomorphic disks has any convergence properties.)

\begin{lemma}\label{l:analytic}
	The potential $W_K$ is analytic as a function of $x=(x_{1},\dots,x_{k},Q)$.
\end{lemma}

\begin{proof}
	The relation between the disk potential and the augmentation variety implies that $p_{j}=\frac{\partial W_K}{\partial x_{j}}$ gives a local branch of the augmentation variety. On the other hand the augmentation variety is an algebraic variety and determines $e^{p_{j}}$ as an algebraic function of $(e^{x_{1}},\dots,e^{x_{k}},Q)$. The lemma follows. 
\end{proof}

\subsection{Properties of linearized contact homology for $2$-component links}\label{sec:linlink}
If $K=K_{1}\cup\dots\cup K_{k}$ is a link, $\Lambda_{K}$ its conormal Legendrian, and 
\[ 
\epsilon\colon CE(\Lambda_{K})\to \C[e^{\pm x_{j}},Q]_{j=1,\dots,k}
\]
an augmentation, then we define the linearized contact homology complex at $\epsilon$:
\[ 
CE^{\rm lin}_{\epsilon}(K)=\ker(\epsilon)/\ker(\epsilon)^{2},
\]
with differential $d^{\rm lin}_{\epsilon}$ induced by the differential $d$ on $CE(\Lambda_{K})$. If the augmentation $\epsilon$ takes all mixed Reeb chords (i.e., Reeb chords with endpoints on distinct components of $\Lambda_{K}$) to $0$, then the linearized complex $CE^{\rm lin}_{\epsilon}(K)$ decomposes as
\[ 
CE^{\rm lin}_{\epsilon}(K)=\bigoplus_{i,j} CE^{\rm lin}_{\epsilon}(K_{i},K_{j}),
\]
where the summand $CE^{\rm lin}_{\epsilon}(K_{i},K_{j})$ is generated by Reeb chords starting on $\Lambda_{i}$ and ending on $\Lambda_{j}$.

For the remainder of this subsection, we specialize to the 2-component case. Let $K_1$ and $K_2$ be disjoint knots and $L_{K_{1}}$ and $L_{K_{2}}$ their conormal Lagrangian submanifolds in $T^{\ast} S^{3}$. Let $\epsilon_{0}$ be the augmentation induced by the exact Lagrangian filling $L_{K_{1}}\cup L_{K_{2}}$. Then $\epsilon_{0}$ acts trivially on mixed chords. Consider $CE^{\rm lin}_{\epsilon_{0}}(K_{1},K_{2})$. Note that on coefficients, $\epsilon_{0}(e^{p_{j}})=1$ and $\epsilon_{0}(e^{x_{j}})=e^{x_{j}}$; also, $Q=1$ since we work in $T^{\ast}S^{3}$.

Write $P(K_1,K_2)$ for the space of paths starting on $K_1$ and ending on $K_2$ and let $C(K_1,K_2)$ denote the singular chain complex $C_{\ast}(P(K_1,K_2))$. Note that $P(K_1,K_2)$ fibers over the torus $K_1\times K_2$ with fiber at $(q_1,q_2)$ equal to $\Omega(q_1,q_2)$, the space of paths connecting $q_1$ to $q_2$. Using this fibration, we consider the singular chain complex $C_{\ast}(P(K_1,K_2))$ with local coefficients in $\pi_1(K_1)\times\pi_{1}(K_2)$. We write the group ring variables as $e^{x_1}$ and $e^{x_2}$ since the generators can be identified with the generators of the first homology of $L_{K_1}$ and $L_{K_{2}}$.

There is a natural chain map 
\[  
\Theta\colon CE^{\rm lin}_{\epsilon_{0}}(K_1,K_2)\to C(K_1,K_2)
\]
which maps a mixed chord $a$ to the singular chain $\Theta(a)$ defined as follows. Let $\mathcal{M}(a;K_1,K_2)$ denote the moduli space of holomorphic disks $u\colon (D,\partial D)\to (T^{\ast}S^{3},L_{K})$, with one positive puncture mapping to $a$ and two Lagrangian intersection punctures mapping to $K_1$ and $K_2$. Then evaluation along the boundary segment between the two Lagrangian intersection punctures gives a path connecting $K_1$ to $K_2$ and we let $\Theta(a)$ be the chain of paths carried by the moduli space, see Figure \ref{fig:pssmap}.

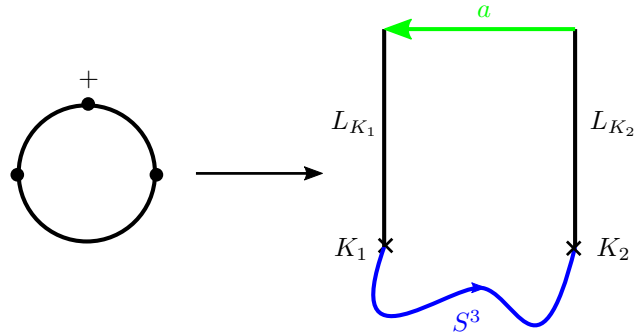
\begin{figure}
\labellist
\small\hair 2pt
\pinlabel $+$ at 23 73
\pinlabel $L_{K_1}$ at 101 60
\pinlabel $L_{K_2}$ at 177 60
\pinlabel $K_1$ at 100 23
\pinlabel $K_2$ at 177 23
\pinlabel ${\color{blue} S^3}$ at 134 2
\pinlabel ${\color{green} a}$ at 139 93
\endlabellist
	\centering
	\includegraphics[width=.6\linewidth]{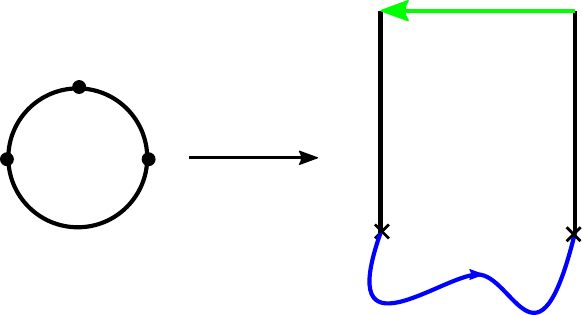}
	\caption{Holomorphic curve for $\Theta(a)$.}
	\label{fig:pssmap}
\end{figure}

\begin{lemma}\label{l:pssmix}
	The map $\Theta$ is a chain map that induces an isomorphism on homology.
\end{lemma}

\begin{proof}
	The proof is similar to the proof of \cite[Theorem 1.1]{CELN}. By SFT compactness, the terms of the chain map equation,
	\[ 
	\partial \circ \Theta -\Theta\circ d^{\rm lin}_{\epsilon_{0}}=0
	\]
	can be identified with the endpoints of the 1-dimensional moduli space of curves with one positive puncture at a mixed Reeb chord and two Lagrangian intersection punctures at $K_{1}$ and $K_{2}$, see Figure \ref{fig:psstwolevel}.
	
	For each binormal geodesic connecting $K_1$ to $K_2$, we have a corresponding Reeb chord from $\Lambda_{K_{1}}$ to $\Lambda_{K_{2}}$ and the $\R$-invariant trivial strip over this chord is a minimal action holomorphic strip. Using a Morse theoretic model of the space $P(K_{1},K_{2})$ and the action filtration, we find that the above chain map is an isomorphism on the first page of the corresponding spectral sequence and hence a quasi-isomorphism. 
\end{proof}

\begin{figure}
\labellist
\small\hair 2pt
\pinlabel $\R \times \Lambda_{K_1}$ at -15 128
\pinlabel $\R \times \Lambda_{K_2}$ at 81 128
\pinlabel $L_{K_1}$ at -8 58
\pinlabel $L_{K_2}$ at 72 58
\pinlabel ${\color{blue} S^3}$ at 27 1
\pinlabel ${\color{green} a}$ at 32 160
\pinlabel ${\color{green} b}$ at 32 94
\endlabellist
	\centering
	\includegraphics[width=.2\linewidth]{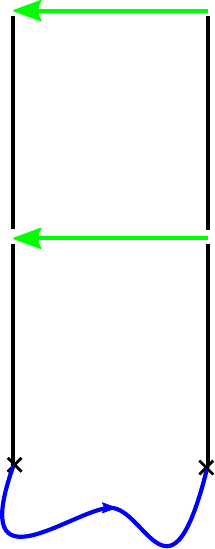}
	\caption{The singular boundary $\partial \Theta(a)$ of $\Theta(a)$ corresponds to two-level curves.}
	\label{fig:psstwolevel}
\end{figure}
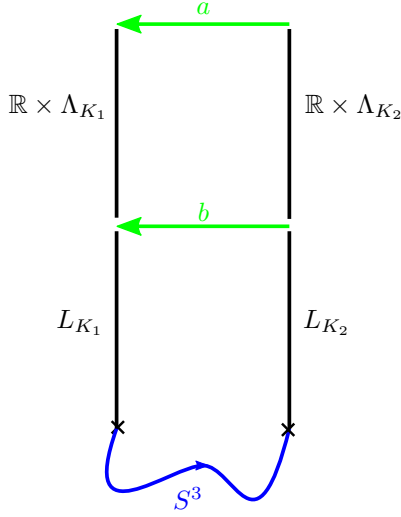

We next compute the homology of $C(K_{1},K_{2})$.
\begin{lma}\label{l:acyclic}
If $x_{1}\ne 0$ or $x_{2}\ne 0$ then the homology of $C(K_{1},K_{2})$ vanishes.
\end{lma}

\begin{proof}
Using the fibration $P(K_1,K_2)\to K_1\times K_2$ we compute the homology $H(C(K_{1},K_{2}))$ via the Leray-Serre spectral sequence  with second page
\[ 
E^{2}_{p,q}= H_{p+q}(K_1\times K_{2};H_{q}(\Omega(k_{1},k_{2}))).
\]
The homology of the fiber is the homology of the based loop space of $S^{3}$, which has rank $1$ in even degrees $0,2,4,\dots$ and rank $0$ otherwise. Thus there can be no higher differentials and we find that the homology is computed on this page. The chain complex is then generated by $M$ of degree 2, $s_1,s_2$ of degree 1, and $m$ of degree $0$, with
\begin{equation}
\begin{aligned}
d M &= (1-e^{x_1})s_1 + (1-e^{x_2})s_2\\
ds_1 &= (1-e^{x_2})m\\
ds_2 &= -(1-e^{x_1})m\\
d m &= 0.
\end{aligned}
\label{eq:morsedistinct}
\end{equation}
This complex is acyclic if $x_1\ne 0$ or $x_2\ne 0$.
\end{proof}

Consider now the transition to the resolved conifold $X$ and the augmentation variety $p_{j}=\frac{\pa W_{K}}{\pa x_{j}}$, $j=1,2$. Let $\epsilon$ denote the augmentation induced by the non-exact Lagrangian filling $L_{K}\subset X$. 

\begin{lemma}\label{l:mixedlinvanish}
	The mixed linearized contact homology $H_*(CE^{\rm lin}_{\epsilon}(K_{1},K_{2}))$ vanishes for $\epsilon$ in a Zariski open subset of the branch of the augmentation variety corresponding to the parameterization $p_{j}=\frac{\partial W_{K}}{\partial x_{j}}$, $j=1,2$.
\end{lemma}

\begin{proof}
%
	The condition of the differentials in the complex $CE^{\rm lin}_{\epsilon}(K_1,K_2)$ being surjective is stable under small perturbations and hence we find that the homology is as claimed in an open subset. The lemma follows.
\end{proof}

\subsection{Properties of linearized contact homology for knots}\label{sec:linknot}
We next consider the counterpart of Lemma \ref{l:mixedlinvanish} for a single knot. The discussion from Section \ref{sec:linlink} needs only small modifications.

The space $P=P(K,K)$ of paths starting and ending on $K$ is the analogue of $P(K_{1},K_{2})$.  For $K_{1}\ne K_{2}$ we had local coefficients in $\pi_{1}(K_{1})\times\pi_{1}(K_{2})$. Here $K_{1}=K_{2}$; the coefficients still sit at the endpoints of the paths and now give a total coefficient in $\pi_{1}(K)$.
Write $C_{\ast}(P)$ for singular chains with these coefficients. Repeating the argument in Lemma \ref{l:acyclic}, we find that the homology $H(C_{\ast}(P))$ equals $0$ for $x\ne 0$. 

Write $P_0\subset P$ for the subspace of constant paths from $K$ to $K$. 
The exact sequence for relative homology (with coefficients in $\pi_{1}(K)$) then gives 
\[
\cdots \longrightarrow H_*(P) \longrightarrow H_*(P,P_0) \longrightarrow H_*(P_0) \longrightarrow \cdots.
\]
Since $H_{\ast}(P)=0$, the quotient complex is isomorphic via the connecting homomorphism to the subcomplex with degree shifted by $1$.

Let $\epsilon_{0}\colon CE(\Lambda_{K})\to\C[e^{\pm x}]$ be the augmentation induced by the exact filling $L_{K}\subset T^{\ast}S^{3}$, $\epsilon_{0}(e^{p})=1$. Denote the corresponding linearized chain complex $CE^{\rm lin}_{\epsilon_{0}}(K)$. As in the two component case, consider the map 
\[ 
\Theta\colon CE^{\rm lin}_{\epsilon_{0}}(K)\to C_{\ast}(P,P_0),
\]
where $C_{\ast}(P,P_0)$ denotes the quotient complex of singular chains $C_{\ast}(P)/C_{\ast}(P_{0})$ and
where, in direct analogy with the two component case, $\Theta(a)$ is the chain of paths carried by the moduli space $\mathcal{M}(a;K,K)$ of holomorphic disks with positive puncture at $a$, and two Lagrangian intersection punctures of $K$. 
\begin{lma}\label{l:qisoknot}
The map $\Theta$ is a chain map and a quasi-isomorphism, inducing an isomorphism
\[ 
H_{\ast}(CE^{\rm lin}_{\epsilon_{0}}(K))\to H_{\ast}(P,P_0)\cong H_{\ast+1}(P_{0}).
\]
\end{lma}

\begin{proof}
To see that the chain map equation holds we note that the codimension one boundary of the chain carried by the moduli space $\mathcal{M}(a;K,K)$ has two parts. The first part, exactly as in the 2-component case, consists of two-level curves with a curve of dimension $1$ in the symplectization. The second part is the locus where the component in the boundary of a map $u\in\M(a;K,K)$ that maps to $S^{3}$ shrinks to a constant. The second degeneration thus gives a chain of constant paths, and since we divide out by chains of constant paths the desired chain map equation follows. 

The quasi-isomorphism statement then follows from existence and uniqueness of trivial Reeb chord strips and an action filtration argument exactly as in the 2-component case.  
\end{proof}

As in the 2-component case we will transfer Lemma \ref{l:qisoknot} to the linearized contact homology for other augmentations $\epsilon$ that can be viewed as small perturbations of $\epsilon_{0}$ induced by the exact filling $L_{K}$. To that end we need a chain complex which is stable under small perturbation. We define it as follows.

Add the Morse complex of $K$, i.e., introduce two additional generators $\xi_{0}$ of degree 0 and $\xi_{1}$ of degree 1. We define the differential $d^{\rm tot}$ on 
\[ 
C^{\rm tot}_{\epsilon_{0}}(K)=C^{\rm lin}_{\epsilon_{0}}(K)\oplus C_{\ast}(K)
\]
as follows:  $d^{\rm tot}\xi_{j}=0$, $j=0,1$, and for Reeb chords $c$, $d^{\rm tot}c=d^{\rm lin}_{\epsilon_{0}} c+ d'_{\epsilon_{0}}c$ where $d^{\rm lin}_{\epsilon_{0}}$ is the differential on $CE^{\rm lin}_{\epsilon_{0}}(K)$ and where 
\[ 
d'_{\epsilon_{0}}\colon CE^{\rm lin}_{\epsilon_{0}}(K)\to C_{\ast}(K)
\] 
is the map that counts holomorphic disks with boundary in $L_K$ as follows. The coefficient of $\xi_{0}$ is the count of curves that pass through any point in $K\subset L_K$ and the coefficient of $\xi_{1}$ the count of curves passes through a specific point in $K$. 
This allows us to define a new chain map
\[ 
\Theta^{\rm tot}\colon C^{\rm tot}_{\epsilon_{0}}(K)\to C_{\ast}(P),
\]
which is defined as before on Reeb chords and takes chains on $K$ to the corresponding chains of constant paths.

\begin{lma}\label{l:qisoknot2}
	The map $\Theta^{\rm tot}$ is a chain map and a quasi-isomorphism. It follows in particular that if $x\ne 0$ and if $b$ is a generator of $H_{1}(CE^{\rm lin}_{\epsilon_{0}}(K))$ then the count of holomorphic disks with positive puncture at $b$ that pass through $K$ is nonzero.
\end{lma}

\begin{proof}
	To see that the chain map equation holds, we note that the map $d'_{\epsilon_{0}}$ followed by the inclusion exactly accounts for the locus in the boundary of $\mathcal{M}(a;K,K)$ where the part of the boundary mapping to $S^{3}$ shrinks to a constant. The chain map equation follows, and the quasi-isomorphism statement then follows as before.   
	
	To see the last statement, note that since $C_{\ast}(P)$ is acyclic, so is $C^{\rm tot}_{\epsilon_{0}}(K)$, which implies that 
	$d'_{\epsilon_{0}}b=\xi_{0}$.
\end{proof}

Lemma \ref{l:qisoknot2} is stable under small perturbations. We use it to prove the following, which is the main result underlying our recursion. 

\begin{lemma}\label{l:generic}
	For augmentations $\epsilon$ in a Zariski open subset of the branch of the augmentation variety corresponding to the conormal filling $L_{K}$ of a knot $K$, we have
	\[ 
	\rk(H_1(CE^{\rm lin}_{\epsilon}(K)))=1\quad \text{and} \quad \rk(H_2(CE^{\rm lin}_{\epsilon}(K)))=1.
	\]
	Furthermore, if we consider disks with positive puncture at a linear combination of Reeb chords that represents the generator of $H_1(CE^{\rm lin}_{\epsilon}(K))$, and if $\xi$ is a parallel of $x$, then the count of these disks that pass through $\xi$ is generically nonzero.   
\end{lemma} 

\begin{proof}
	Let $\epsilon$ denote the augmentation induced by $L_{K}\subset X$ for small $x$. As in the proof of Lemma \ref{l:mixedlinvanish} we see that the differential on $CE^{\rm tot}_{\epsilon}(K)=CE^{\rm lin}_{\epsilon}(K)\oplus C_{\ast}(K)$, with $d'_{\epsilon}$ defined by counting disks with insertions passing through $K$, is a small perturbation of the differential on $C^{\rm tot}_{\epsilon_{0}}(K)$. Since the condition that the differential is an isomorphism is stable, it follows that $C^{\rm tot}_{\epsilon'}(K)$ is generically acyclic. For the last statement, we note that since the knot $K$ is homotopic in $L_{K}$ to a parallel $\gamma$ of $x$, the complex $C^{{\rm tot}'}_{\epsilon'}(K)$ obtained by replacing $C_{\ast}(K)$ with $C_{\ast}(\gamma)$ is still acyclic. If the count of $\R$-invariant disks with positive puncture at a generator of $CE^{\rm lin}_{\epsilon}(K)$ through $\gamma$ were equal to $0$, then the same would be true for disks with insertions, and it would follow that the generator of $H_{1}(\gamma)$ survives in the homology of $C^{{\rm tot}'}_{\epsilon'}(K)$. This contradicts vanishing homology, and it follows that the count must be nonzero.  
\end{proof}

\begin{rmk}
If $K$ is a knot and if $a_{1},\dots,a_{r}$ are its Reeb chords of degree 0 then we can consider the full augmentation variety $\tilde V_{K}$ as the subset of $\C^{r}\times (\C^{\ast})^{3}$ given by
\begin{align*}
\tilde V_{K}=\bigl\{&(\epsilon_{1},\dots,\epsilon_{r},e^{x_{0}},e^{p_{0}},Q_{0})\colon\\
 &a_{j}\mapsto \epsilon_{j},e^{x}\mapsto e^{x_{0}},e^{p}\mapsto e^{p_{0}},Q\mapsto Q_{0} \text{ is an augmentation}\bigr\}.
\end{align*}
Then since there are no generators in negative degree, the degree $0$ linearized contact homology $H_0(CE^{\rm lin}_{\epsilon}(K))$ is the tangent space to the subset of $\tilde V_{K}$ lying over $(e^{x_{0}},e^{p_{0}},Q_0)$. By Lemma~\ref{l:generic} and the easily checked fact that $CE^{\rm lin}_{\epsilon}(K)$ has Euler characteristic $0$, $H_0(CE^{\rm lin}_{\epsilon}(K))=0$ over a generic point in $V_{K}$. It follows that the natural projection map $\tilde V_{K}\to V_{K}$ is an immersion over generic points in $V_{K}$.
\end{rmk}

\subsection{The annulus amplitude for a 2-component link}\label{sec:annulus}
In this subsection we explain how the results in Sections \ref{sec:linlink} and \ref{sec:linknot} allow us to compute the annulus amplitude over a generic point in the augmentation variety. This corresponds to the first step in the recursive calculation of the wave function that we present in Section \ref{sec:toprec}. We choose to treat this case separately since the curve counts at infinity involved in this case do not need any abstract perturbations, and therefore, in combination with Lemma \ref{l:nogenus}, they lead to a direct combinatorial formula. See Section \ref{sec:annulusforHopf} for the computation of the annulus amplitude for the Hopf link.

Let $K_{1}$ and $K_{2}$ be two disjoint knots and let $L_{K_{1}}$ and $L_{K_{2}}$ denote their Lagrangian conormals in $X$. Assume that we are at a generic point in the augmentation variety and let $c=\sum_{j} \gamma_{j}c_{j}$ be a linear combination of Reeb chords of $\Lambda_{K_{1}}$, with $|c_{j}|=1$ for each $j$, such that $c$ represents a generator for $H_1(CE^{\rm{lin}}_{\epsilon}(K_{1}))$. Here we take $\gamma_{j}=\gamma_{j}(e^{x_{1}},Q)$. We consider two counts of holomorphic disks with positive puncture at chords in $c$. 

First, define $\Delta^{c_{j}}_{K_{1}}$ to be the count of augmented disks with positive puncture at $c_{j}$:
\[ 
\Delta_{K_{1}}^{c_{j}}=\sum_{|\mathbf{a}|=0}|\mathcal{M}_{l,m,k}(c_{j},\mathbf{a})|e^{lx}e^{mp}Q^{k}\epsilon(\mathbf{a}),
\]
where we view the augmentation $\epsilon$ as an algebraic function of $(e^{x_{1}},Q)$. Noting that $\left.\Delta_{K_{1}}^{c_{j}}\right|_{p_{1}=\frac{\partial W_{K_{1}}}{\partial x_{1}}}$ is the constant term in the dg-algebra differential (i.e., the differential in $\mathcal{A}_{K}$) of $c_{j}$ twisted by the augmentation $\epsilon$, we find that $\Delta_{K_{1}}^{c_{j}}=0$ along the augmentation variety. This means that the augmentation polynomial divides $\Delta_{K_{1}}^{c_{j}}$. 

If we now define
\[ 
\Delta_{K_{1}}^{c}=\sum_{j}\gamma_{j}\Delta_{K_{1}}^{c_{j}},
\]
then for some $G$, we have
\[ 
\Delta_{K_{1}}^{c} = G(e^{x_{1}},e^{p_{1}},Q)\cdot A_{K_{1}}(e^{x_{1}},e^{p_{1}},Q),
\]
where $A_{K_{1}}$ is the augmentation polynomial of $K_{1}$. 

Second, we define $\Theta^{c_{j}}$ to be the count of augmented holomorphic disks with two mixed negative punctures:
\[ 
\Theta^{c_{j}}=\sum_{|\mathbf{a}|=0}|\mathcal{M}_{l,m,k}(c_{j},\mathbf{a})|e^{lx}e^{mp}Q^{k}\epsilon(\mathbf{a}')\partial_{a_{12}}\partial_{a_{21}},
\] 
where $\mathbf{a}$ runs over all degree 0 words with exactly two mixed negative punctures $a_{12}$ and $a_{21}$, and $\mathbf{a}'$ denotes the subset of pure punctures in $\mathbf{a}$.

Finally, for mixed Reeb chords, let 
\[ 
B = g_{s}\sum_{a_{12},a_{21}}|\mathcal{M}_{l_1,1_2,k}(a_{12}a_{21})|e^{l_{1}x_{1}}e^{l_{2}x_{2}}Q^{k}a_{12}a_{21}
\]
be the generating function for disks with two positive punctures at mixed chords, and let the annulus amplitude be
\[ 
R_{K_{1}K_{2}} = \sum_{l_{1},l_{2},k} R_{l_{1},l_{2},k} e^{l_{1}x_{1}}e^{l_{2}x_{2}}Q^{k}.
\]
Similarly let
\[ 
R_{K_{1}K_{2}}^{c} = \sum_{l_{1},l_{2},k} R_{l_{1},l_{2},k}^{c} e^{l_{1}x_{1}}e^{l_{2}x_{2}}Q^{k}
\]
denote the count of $\R$-invariant annuli with positive puncture at $c$. We point out that Lemma \ref{l:nogenus} below shows that we can perturb $\Lambda_{K}$ so that $R^{c}_{K_{1}K_{2}}=0$.
\begin{thm}\label{t:annulus}
The following equation holds:
\[ 
\left.\frac{\partial}{\partial p}\Delta^{c}_{K_{1}}\right|_{p_{1}=\frac{\partial W_{K_{1}}}{\partial x_{1}}} \cdot \frac{\partial}{\partial x_{1}} R_{K_{1}K_{2}} +
\Theta^{c}B + R^{c}_{K_{1}K_{2}}=0.
\]
\end{thm}

\begin{rmk}
Note that
\[ 
\left.\frac{\partial}{\partial p}\Delta^{c}_{K_{1}}\right|_{p_{1}=\frac{\partial W_{K_{1}}}{\partial x_{1}}}= G\left(e^{x_{1}},e^{\frac{\partial W}{\partial x_{1}}},Q\right)\cdot\frac{\partial A_{K_{1}}}{\partial p_{1}}\left(e^{x_{1}},e^{\frac{\partial W}{\partial x_{1}}},Q\right).
\]
\end{rmk}

\begin{rmk}
A similar result starting from $K_{2}$ implies that
\begin{align*} 
&\left(\left.\frac{\partial}{\partial p}\Delta^{c}_{K_{1}}\right|_{p_{1}=\frac{\partial W_{K_{1}}}{\partial x_{1}}}\right)^{-1}
\left(\frac{\partial}{\partial x_{1}}\Theta^{c}B+R^{c}_{K_{1}K_{2}}\right)\\
&\quad = \left(\left.\frac{\partial}{\partial p}\Delta^{c'}_{K_{2}}\right|_{p_{2}=\frac{\partial W_{K_{2}}}{\partial x_{2}}}\right)^{-1}
\left(\frac{\partial}{\partial x_{2}}\Theta^{c'}B+R^{c'}_{K_{1}K_{2}}\right).
\end{align*}
\end{rmk}

\begin{proof}[Proof of Theorem~\ref{t:annulus}]
The equation accounts for the boundary points of the oriented manifold of annuli with one boundary component on each Lagrangian and positive puncture at $c$. To see this, note that after using bounding chains only SFT splitting remains. The fact that $c$ is a cycle in linearized contact homology implies that the term with a strip in the cylindrical region equals zero.  
\end{proof}

We next explain how to compute $B$ from data at infinity. Consider the coefficient $B(a_{12},a_{21})$ of $a_{12}a_{21}$ in $B$ counting disks with two positive punctures, at $a_{12}$ and $a_{21}$.

\begin{lma}
If $d$ denotes the differential in $CE^{\rm lin}_\epsilon(K_{1},K_{2})$ and $a_{12}$ is a degree $0$ generator, then there exists $b_{12}$ such that $d b_{12}=a_{12}$, and if $B'(b_{12},a_{21})$ is the generating function for augmented disks in the symplectization then
\[ 
B(a_{12},a_{21})= B'(b_{12},a_{21}).
\]  
\end{lma}

\begin{proof}
The first statement follows from Lemma \ref{l:mixedlinvanish}. To establish the second, identify the two sides as counting boundary components of the 1-dimensional moduli space of disks with two positive punctures at $b_{12}$ and $a_{21}$.
\end{proof}

\begin{rmk}
The above discussion gives the following scheme for determining the annulus amplitude. First, determine the augmentation on Reeb chords by elimination theory. Second, find the differential for linearized homology and preimages of all mixed degree 0 chords. Next, count disks contributing to $\Theta^{c}$, $\Delta^{c}$, and $B'(b_{12},a_{21})$. Together with the augmentation polynomial, this gives the annulus amplitude.
\end{rmk}

\subsection{A-model recursion for the full wave function.}\label{sec:toprec}
In this subsection, we describe the nature of A-model recursion at infinity, which shows how to recover counts of closed curves (i.e.~curves without punctures) of arbitrary Euler characteristic on the conormal from rational curves at infinity. One of the reasons for this is that it shows that the SFT-formalism indeed recovers the wave function and thus the $D$-module that gives the recursion relation. Another is that it gives an inductive scheme for actually computing amplitudes for curves of higher negative Euler characteristic. 

Let $K=K_1\cup\dots\cup K_k$ be a $k$-component link. Let 
\[ 
\mathbf{F}_{K}=F(x,Q)= \sum C_{K;\chi,n,k} \,g_s^{-\chi}\, e^{n\cdot x}\, Q^k
\]  
denote the holomorphic curve amplitude, so that
\[ 
\Psi_{K}=\Psi(x)=\exp(F_{K})
\]
is the wave function, counting all disconnected curves. Note that
\[ 
F_{K}= g_{s}^{-1}\left(F_{K;0} + g_{s}F_{K;1} + g_{s}^{2}F_{K;2} + \dots \right),
\]
where $F_{K;j}=F_{j}(x)$ counts curves of Euler characteristic $\chi=1-j$.

The disk amplitude $F_{K;0}$ determines (an irreducible component of) the augmentation variety and can be computed from it via
\[ 
p_j=\frac{\partial F_{K;0}}{\partial x_j},
\]
where $e^{p_j}=e^{p_j(x)}$ is a local parameterization of the augmentation variety. Note that if $L_K$ is the conormal Lagrangian then $e^{p_j}=e^{p_j(x_j)}$ is a function of $x_j$ only and consequently
\[ 
F_{K;0}(x)=\sum_{j=1}^{k} F_{K_j;0}(x_j),
\]
where $F_{K_j;0}(x_{j})$ is the disk potential of $L_{K_{j}}$.

We next turn to the recursion. We consider curves with several positive and several negative punctures. Either one or none of the positive punctures will have grading 1, and all other positive punctures have grading 0. All negative punctures have grading $0$. We call a curve with one positive puncture of grading 1 an \emph{index 1 curve} and other curves under consideration \emph{index 0 curves}. 

We say that a curve has \emph{type} $(n,\chi)$ if it has $n$ positive grading $0$ punctures and if it has Euler characteristic $\chi$. We say that an index 0 curve attached to an index $1$ curve has \emph{attached type} $(n_0,n_1,\chi)$ if it is attached via $n_0$ bounding chain insertions and positive punctures, has  $n_1$ auxiliary positive punctures (not attached to any negative puncture), and has Euler characteristic $\chi$.

We observe that an index 1 curve of type $(n^{0},\chi^0)$ with $m$ index 0 curves of attached types $(n_0^{j},n_1^{j},\chi^{j})$, $j=1,\dots,m$ attached gives a two-level index 1 curve of type 
\[
(n,\chi)=\left(\sum_{j=0}^{m} n_1^{j} \ , \ \sum_{j=0}^{m}\chi^{j}-\sum_{j=1}^{m} n_0^{j}\right).
\]

The key step for the inductive calculation of the amplitudes $F_{j}$ is the following result.
\begin{lma}
The amplitudes of all index 0 curves of type $(n,\chi)$ with $-\chi+n=r$ is determined by the amplitudes of the curves of index 1 of type $(n,\chi)$ at infinity, together with the amplitudes of the index 0 curves of type $(n,\chi)$ with $-\chi+n<r$.  
\label{l:indexamp}
\end{lma}

\begin{rmk}
We show in Section \ref{sec:curvesatinfinity} that the amplitudes of curves of index 1 at infinity can be expressed in terms of rational curves only, after introducing extra Reeb chords.
\end{rmk}

\begin{proof}[Proof of Lemma~\ref{l:indexamp}]
 Let $b$ be the generator of $H_1(CE^{\rm lin}_{\epsilon}(K))$ and consider the moduli space of holomorphic curves of index 1 and type $(0,r)$ which have a positive puncture at $b$. The boundary of this moduli space consists of two-level curves of type $(0,r)$. Here there is only one type of broken configuration that contains index $0$ curves of attached type $(1,0,r)$: these are augmented disks with one bounding chain insertion. If $F_r$ is the amplitude of index $0$ curves of type $(0,r)$ then this gives
\[ 
B(e^{x_1},Q)\cdot F_r,
\] 
where $B$ counts $\R$-invariant disks with positive puncture at $b$. By Lemma \ref{l:generic}, $B$ is nonzero.

Furthermore there is one broken configuration that contains an attached curve of type $(0,1,r)$. The upper level of such a curve is a strip with positive puncture at $b$ and one negative puncture. Since  $b$ is a cycle for the linearized differential, we find that the total contribution from such two-level curves equals $0$ (since the augmented curves in the upper level of the two-level configurations already cancel out). Consequently, counting ends of the 1-dimensional moduli space of curves with one positive puncture at $b$, we find:
\[ 
B(e^{x_1},Q)\cdot F_r + R^{2} + R^{\infty}=0,
\]
where $R^{2}$ is the count of two-level curves with both components of type $(n,\chi)$ with $-\chi+n<r$, and $R^{\infty}$ is the count of $\R$-invariant curves of type $(n,\chi)$ with $-\chi+n=r$.
Thus we can solve for $F_{r}$ in terms of index 0 curves of type $(n,\chi)$ with $-\chi+n<r$ and $\R$-invariant curves of index 1 of type $(n,\chi)$ with $-\chi+n\le r$.

In order to complete the proof we then have to show also how to compute the amplitudes of all other curves with $\chi+n=r$. The argument is similar: any such curve has a positive puncture of grading $0$ and by Lemma \ref{l:generic} we can find a linear combination $b'$ of degree $1$ chords such that the image of $b'$ under the linearized differential is the positive puncture of degree $0$. Studying breakings of the moduli space of index 1 curves with positive puncture at $b'$ and $n-1$ other positive grading 0 punctures and arguing exactly as above, we can solve for the desired amplitude. (Note that we use also the knowledge of $F_{K;r}$ in this calculation.) This finishes the proof.
\end{proof}

\subsection{Curves at infinity}\label{sec:curvesatinfinity}
In this subsection we discuss the curves at infinity. We recall the strategy for describing the holomorphic curves in the $\R$-invariant region, see \cite{EENS}. Represent $K$ as a braid around the unknot. Then $\Lambda_{K}$ lies in a $1$-jet neighborhood of $\Lambda_{U}$, where $U$ is the unknot. Furthermore, if we shrink $K$ toward $U$, holomorphic curves with boundary on $\Lambda_{K}$ converge to holomorphic curves on $\Lambda_{U}$ with flow lines attached. As in the calculation of the knot contact homology differential, we choose the almost complex structure so that the projection of holomorphic curves into $T^{\ast} S^{2}$ remain holomorphic. Furthermore, we take the link to lie in a small ball in $S^{3}$, and we call Reeb chords contained in the unit cotangent bundle restricted to this ball ``small''. It is straightforward to check that any non-small Reeb chord has index at least $2$.

\begin{lma}\label{l:nogenus}
Let $K$ be any link. Then there exists a deformation of $\Lambda_{K}$ such that any rigid holomorphic curve on $\R\times\Lambda_{K}$ is rational and such that all small Reeb chords of $\Lambda_{K}$ have degrees $0$, $1$, or $2$. 
\end{lma}

\begin{proof}
 Since the only holomorphic curves with boundary on $\Lambda_{U}$ are disks, it follows that holomorphic curves on $\Lambda_{K}$ must limit to these holomorphic disks with flow lines attached. In particular, curves of Euler characteristic $\ne 1$ must either contain a flow tree connecting such a disk to itself or contain a flow graph which is not a tree. Note however that such a configuration lifts to a holomorphic curve in the symplectization if and only if the lifts of the disk and the flow tree  or the flow graph close up. 
 
 We consider first the case of a flow tree connecting the big disk to itself. In the limit the flow line is very thin and therefore nearly horizontal (in the symplectization direction). It follows that if we find a perturbation of $\Lambda_{K}$ such that no flow line connects points of the disks of the unknot that lift to same height, then there are no such non-rational curves. Figure \ref{fig:height} shows the points of equal heights and it is straightforward to check the representation of the braid shown in Figure \ref{fig:noannuli} has no flow lines connecting points at the same height.
 
\begin{figure}
	\centering
	\includegraphics[width=0.6\linewidth]{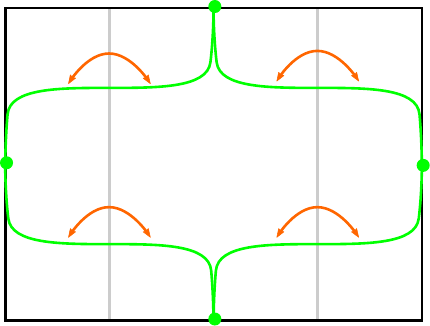}
	\caption{Boundaries of big disks on the conormal of the unknot, compare Figure~\ref{fig:unknotdisk} below. The gray lines subdivide the disk boundaries in two components. The height function is symmetric in these components, decreases in one, and increases in the other.}
	\label{fig:height}
\end{figure}

\begin{figure}
\labellist
\small\hair 2pt
\pinlabel braiding at 54 3
\endlabellist	
	\centering
     \includegraphics[width=0.6\linewidth]{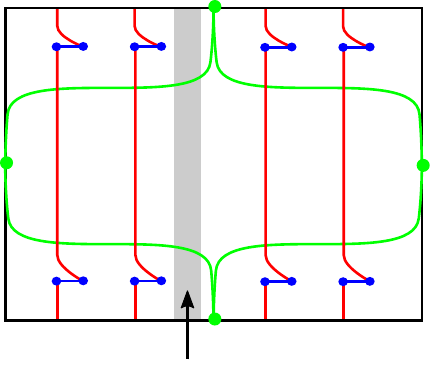}
	\caption{Additional short Reeb chords subdivide the Morse flow so that no points of equal height on the big disks are connected.}
	\label{fig:noannuli}
\end{figure}

 We next consider flow graphs other than trees. Here our argument uses the specific form of braid we use. We separate the strands in the braid by an increasing amount in each step. This means that flow graphs of $n$ sheets can be viewed as flow graphs of $(n-1)$ sheets with flow lines attached. Furthermore, from the point of view of the lifts of curves in earlier steps, the flow lines in the last step lift to disks in the symplectization that are virtually constant in the symplectization direction until they go vertically down (negative puncture) or up (positive puncture) at the Reeb chord. It follows from this that there are no higher genus flow graph contributions from flow graphs that are not trees. 
\end{proof}

We point out that even though the rational curves of Lemma \ref{l:nogenus} are disks, they may give rise to generalized curves with $\chi<1$ via insertion of bounding chains which here correspond to linking. Basic rational curves are straightforward to describe: they are as easy to find as the curves in the contact homology differential. The exact contributions at higher genus from curves with several positive punctures involve the nature of the actual perturbation scheme. We next give a conjectural description of this.

We first discuss what generalized curves there are. The obvious generalized curves are the 1-vertex graphs with edges connecting this vertex. Here there is a disk of dimension 1 at the vertex and the edges correspond to self linking. As we will see in the examples below the generalized curves that contribute to the Hamiltonian include also graphs where there are trivial strips at the vertices that are not the main vertex. 

Furthermore, at the main vertex there can be only transversely cut out 1-dimensional curves. 
To see this, we must discuss possible contributions from 1-dimensional curves with branched covers of trivial Reeb chord strips attached. Consider a branch cover of degree $d$. We use an obstruction bundle argument. If we fix the location of the branch points, the linearized holomorphic curve equation has index $-2d$ and zero kernel. We can find a section of the obstruction bundle which is nonzero over the interior of the space branch points. When a branch point moves to the boundary, the curve breaks in two branched covers of trivial strips. Continuing this way, we eventually break the curve into only trivial strips, and the contribution to the moduli space is controlled by the usual linking and intersection with $C_{K}$.

\begin{rmk}\label{r:factorsforextrapositive}
We conjecture that rational curves in the $\R$-invariant region with $m$ positive punctures should be counted with a factor of
\[ 
\pm e^{rg_{s}}(e^{\frac12g_{s}}-e^{-\frac12g_{s}})^{m},
\]
where $r$ is determined by self-linking (including interior intersections with $C_{K}$) as usual, and where the other factors come from extension of the abstract perturbation and limits to the original disk with constant lower Euler characteristic curves attached. 

To motivate this, we consider the construction of the perturbation scheme. The scheme starts from indecomposable curves (minimal energy and disks in the symplectization with only one positive puncture) and makes them regular. These curves then appear in various configurations at the boundary of curves of the next complexity. Here we pick a perturbation near the boundary and extend in order to make curves of the next complexity regular. In this case the boundary configurations correspond to breaking at Reeb chords, interior crossing, and boundary crossing. Near the latter two, the gluing arguments of Lemmas \ref{l:boundarysplitandcross} and \ref{l:elliptic} give a neighborhood of the boundary in the moduli space. Consider a disk with $m$ interior boundary crossings. Such a disk lies at the boundary of a moduli space of disks with $m$ holes, which can further split into a disk with $m$ positive punctures and a disk with $1$ positive and $m$ negative punctures. Gluing a strip with one positive and one negative puncture to the lower half we find a disk with $m$ positive punctures in the symplectization. The gluing analysis at the interior crossing gives a count:
\[ 
(e^{\frac12 g_{s}}-e^{-\frac12 g_{s}})^{m},
\] 
and we claim that there exists a perturbation scheme such that these small curves can be concentrated near positive punctures. One can study boundary crossings in a similar spirit with the same result.

We point out that this conjecture asserts more than the mere existence of a perturbation scheme: it says that there is such a scheme for which the contribution from rational curves at infinity has a specific form.
\end{rmk}

We end this section with a discussion of the 2-component link case. Here we claim that if we choose the conormal $\Lambda_{K_{1}}\cup \Lambda_{K_{2}}$ so that Lemma \ref{l:nogenus} holds, then there are no generalized holomorphic annuli at infinity: 

\begin{lma}\label{l:noformalannuli}
Let $K=K_{1}\cup K_{2}$ be such that Lemma \ref{l:nogenus} holds. Then $R^{c}_{K_{1}K_{2}}=0$ (see Theorem \ref{t:annulus} for notation).
\end{lma}

\begin{proof}
Formal annuli come from disks with a self insertion. Since the augmentation equals 0 on all mixed chords, such an annulus has homotopically trivial boundary in at least one $\Lambda_{K_{j}}$ and hence does not contribute to the amplitude considered.	
\end{proof}

\section{Knot contact homology in quotients of $S^{3}$}
\label{sec:quotients}
In this section we briefly turn to the study of more general ambient spaces than $S^{3}$. We use similar notation: let $M$ be a 3-manifold, $T^{\ast}M$ its cotangent bundle, and $ST^{\ast}M$ its unit cotangent bundle. If $K\subset M$ is a knot then denote its Lagrangian conormal $L_{K}$ and Legendrian conormal $\Lambda_{K}$. 

Here we consider the setting where $M$ is a quotient of $S^3$ by a discrete group, which we assume is a group of isometries of $S^3$ with the round metric.
Iit is known how to construct a large $N$ dual to $T^{\ast}M$ using toric geometry, see e.g. \cite{Marino}. In particular, the K\"ahler parameters in the large $N$ dual correspond to the free homotopy classes of loops in $M$. We will discuss the augmentation variety in this setup, and show in particular how deformation parameters associated to these K\"ahler parameters arise in this context from the orbit contact homology dg-algebra. 

In general, if $\Lambda\subset Y$ is a Legendrian submanifold of a contact manifold then the Chekanov--Eliashberg dg-algebra $CE(\Lambda)$ discussed above, generated by Reeb chords, is an algebra over the orbit contact homology dg-algebra $\mathcal{Q}(Y)$ of $Y$. Here the orbit contact homology algebra is the closed string counterpart of the Chekanov--Eliashberg algebra generated by Reeb orbits and with differential counting holomorphic spheres with one positive and several neagtive punctures.   

When the first Chern class of $Y$ equals $0$, then $\mathcal{Q}(Y)$ admits a grading and one can restrict the coefficient algebra of the Chekanov--Eliashberg algebra to the degree 0 orbit contact homology. In particular, when $M$ is a geometric quotient of the 3-sphere and $Y = ST^{\ast}M$, we have the following.

\begin{lma}\label{l:deg0CH}
If $M$ is a closed $3$-manifold with constant sectional curvature $1$ (universal cover $S^{3}$) and $Y = ST^{\ast}M$ then the homology of the orbit dg-algebra in degree 0, $H\mathcal{Q}_{0}(Y)$, is a commutative algebra generated by the free homotopy classes of loops in $M$.	
\end{lma} 

\begin{proof}
For the standard action from $p\, dq$ as the contact form on 
$Y=ST^{\ast}M$, closed Reeb orbits are lifts of geodesic loops. Consider a closed geodesic in $M$ that lifts to a geodesic arc of length $\ell< \pi$ in $S^{3}$ with respect to the standard metric. Then any other geodesic in its homotopy class has index $\ge 2$. For geodesics of length equal to $\pi$ the same argument applies after small perturbation. This means that there is a contact form with one Reeb orbit for each free homotopy class and all other orbits of grading $\ge 2$. Thus the contact homology orbit algebra is $0$ in degrees $1$ and $\leq -1$, and in degree $0$ it is the commutative algebra generated by free homotopy classes of loops. The result follows.
\end{proof}

As a special case of Lemma~\ref{l:deg0CH}, for $M=S^{3}$ there is a unique homotopy class of free loops, and so the orbit algebra is just the ground ring in degree $0$, which explains why we can disregard the orbit algebra in this case. 

Now let $K$ be a link in $M$ with Lagrangian conormal $L_{K}$. Exactly as in $S^{3}$ we can use a closed form in a neighborhood of $K$ to shift $L_{K}$ off the zero section and then consider $L_{K}\subset X_{M}$ as a Lagrangian in the large $N$ dual $X_{M}$ of $T^{\ast}M$. Consider now the Chekanov--Eliashberg dg-algebra $CE(\Lambda_{K})$ and restrict the coefficients to the degree 0 homology $H\mathcal{Q}_{0}(ST^{\ast}M)$. Using Lemma \ref{l:deg0CH} we then find that the resulting
knot contact homology $H_*(CE(\Lambda_{K}))$ in this case can be considered as an algebra over
\[ 
\C[e^{x},e^{p},Q, H\mathcal{Q}_{0}(ST^{\ast}M)].
\]
In particular we get the expected deformation variables for the augmentation variety. Note also that if $\gamma$ is a free homotopy class then the value on $\gamma$ corresponds to the count of curves in $X_{M}$ with one positive puncture at $\gamma$:
\[ 
\epsilon(\gamma)=\sum_{k_{0},k_{1},\dots, k_{m}} C_{k} Q_{0}^{k_{0}}\cdots Q^{k_{m}}_{m}.
\]
Here $C_{k}$ counts curves with one positive puncture at $\gamma$ representing the class $\sum_{j}k_{j}t_{j}$, where $t_{j}$ is the $j^{\rm th}$ K\"ahler parameter.

\section{Examples}

In this section, we outline some illustrative computations for a few simple examples: the unknot and Hopf link in $S^3$ and the line in $\mathbb{R}P^3$. 

\begin{rmk}
Some of the details in the below calculations (signs, the effect of choices of capping paths, etc.) are not made explicit here. It is an interesting open question to make these choices in some systematic manner in order to get a direct combinatorial formula for the SFT Hamiltonian in terms of a braid presentation. We leave these details to future work and view the examples worked out here as good indications that this is indeed possible.
\end{rmk}

\subsection{The unknot}
Let $U$ denote the unknot in $S^3$. The Chekanov--Eliashberg dga of $\Lambda_U$, $CE(\Lambda_U)$, was worked out in \cite{EENS}.
In the conormal of the unknot, there is one generator $c$ of degree $1$ and no generators of degree 0, and there are four disks contributing to $\partial c$; see Figure \ref{fig:unknotdisk}.
We can choose capping paths such that (possibly after a change of variables) the Hamiltonian is
\[ 
\mathbf{H}= \mathbf{H}^{c} = 1 - e^{x} -e^{p} + Qe^{x}e^{p}.
\]  
Setting $p=g_{s}\frac{\partial}{\partial x}$, we get the recursion relation for the unnormalized colored HOMFLY-PT polynomial of the unknot.

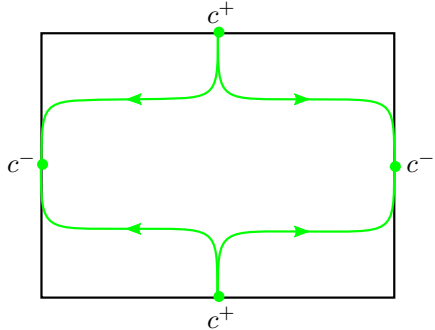
\begin{figure}[htp]
\labellist
\small\hair 2pt
\pinlabel $c^+$ at 66 102
\pinlabel $c^+$ at 66 -2
\pinlabel $c^-$ at -2 51
\pinlabel $c^-$ at 134 51
\endlabellist
	\centering
	\includegraphics[width=0.4\linewidth]{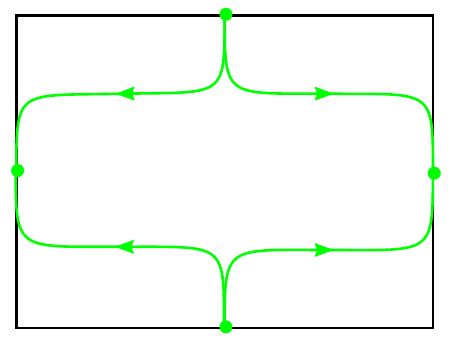}
	\caption{Rigid disks on the conormal of the unknot $U$. Here the Legendrian torus $\Lambda_U$ is the depicted square with sides identified, and the curves denote the boundaries of the four rigid disks. The labels $c^+,c^-$ denote the ending and beginning points of the Reeb chord $c$. Compare \cite[Figure 20]{EENS}.}
	\label{fig:unknotdisk}
\end{figure}

We also illustrate how to reconstruct the wave function recursively. We write the wave function as
\[ 
\Psi_{U}= \exp\left(g_{s}^{-1}F_{U;0}+F_{U;1}+g_{s}F_{U;2}+\dots\right).
\]
Since there are only disks at infinity we find that the disks with insertion must cancel as follows:
inserting $F_{U;m}$ once must cancel with all possible insertions of lower genus curves that give the same Euler characteristic. This means that
\[ 
g_{s}\frac{d F_{U;m}}{dx} = \sum \frac{1}{s_{1}!\dots s_{n}!}\left(\frac{d^{s_{1}}F_{U;m_{1}}}{dx^{s_{1}}}\right)^{t_{1}}\dots \left(\frac{d^{s_{n}}F_{U;m_{n}}}{dx^{s_{n}}}\right)^{t_{n}},
\] 
where the sum ranges over all products for which $\sum s_{j}m_{j}t_{j}=m$ and all $m_{j}\ge 2$. It then follows that
\[ 
e^{g_{s}\frac{d}{dx}} e^{F_{U}}= e^{\frac{dF_{U;0}}{dx}}e^{F_{U}}= \frac{1-e^{x}}{1-Qe^{x}}e^{F_{U}};
\]
to see this note that all terms in the right hand side that includes $\frac{d^{s}F_{U;m}}{dx^{s}}$ where $m+s>1$ cancel out. We conclude that
\[ 
\left(1 - e^{x} -e^{g_{s}\frac{d}{dx}} + Qe^{x}e^{g_{s}\frac{d}{dx}}\right)e^{F_{U}}=0,
\]
as expected.

\subsection{The line in $\R P^{3}$}
In this subsection we carry out the calculation of the knot contact homology of a projective line $\ell\subset \R P^{3}$. This is a special case of the setup discussed in Section~\ref{sec:quotients}. The large $N$ dual of $\R P^{3}$ is local $\C P^{1}\times\C P^{1}$, i.e., the total space of the $\mathcal{O}(-2,-2)$ line bundle over $\C P^{1}\times\C P^{1}$. 

We use the constant sectional curvature $1$ metric on $\R P^{3}$. In this metric there is a Bott-family of closed geodesics of length $\pi$ where the Bott-manifold is $G_{2,4}$, the Grassmann manifold of 2-planes in $\R^{4}$. Furthermore, $\ell$ can be represented by a closed geodesic of length $\pi$ and there is a $T^{2}$ Bott-family of binormal geodesics of $\ell$ of length $\pi$ (where in fact all Reeb chords are also Reeb orbits). Morsifying this situation gives one Reeb chord of index 0, corresponding to the minimum on $G_{2,4}$, and a shortest Reeb chord of index 1, corresponding to the minimum of the $T^{2}$-family. (To see that the index of the chord equals 1, note that we can shrink it to a constant by moving the endpoints along $\ell$.) Let $\gamma$ denote the index 0 Reeb orbit and $c$ the index 1 Reeb chord. 

\begin{thm}
	The contact homology differential $\partial$ on $CE(\Lambda_{\ell})$ satisfies
	\[ 
	\partial c = e^{x} - e^{-x} + e^{p} + Q e^{-p} + \gamma.
	\]
\end{thm}

\begin{proof}
	There are again four disks corresponding to the disks of the unknot after lifting, see Figure \ref{fig:disksforline}. The disk with positive puncture at $c$ and negative at $\gamma$ corresponds to shrinking a based loop to a free loop. In the Bott-Morse picture this corresponds to holomorphic cylinder with a slit at the location of the minimum on $T^{2}$ followed by a Morse flow line to the minimum in $G_{2,4}$. A gluing argument gives exactly one such disk for generic perturbation, see Figure \ref{fig:MorseBott}. 
\end{proof}

\begin{figure}
\labellist
\small\hair 2pt
\pinlabel $x'$ at 47 0
\pinlabel $p'$ at 68 50
\endlabellist	
	\centering
	\includegraphics[width=0.3\linewidth]{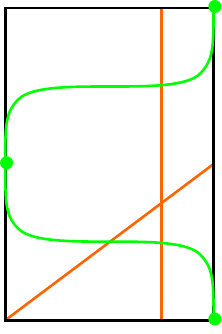}
	\caption{Rigid disks on the conormal of the line (boundaries of these disks are drawn in green). The generators of $H_1(\Lambda_\ell)$ are $x$ dual to $x'$ and $p$ dual to $x'+p'$.}
	\label{fig:disksforline}
\end{figure}

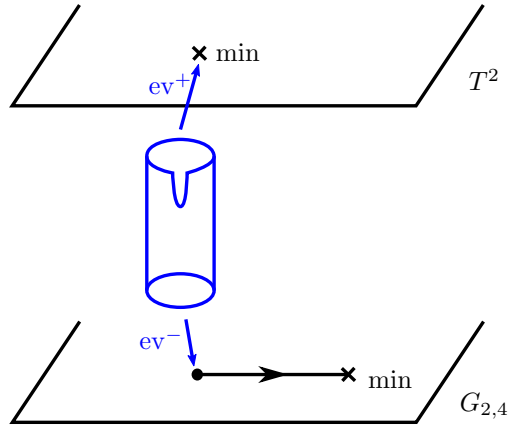
\begin{figure}
\labellist
\small\hair 2pt
\pinlabel min at 51 83
\pinlabel min at 85 10
\pinlabel $T^2$ at 106 77
\pinlabel $G_{2,4}$ at 106 4
\pinlabel ${\color{blue} \ev^+}$ at 36 76 
\pinlabel ${\color{blue} \ev^-}$ at 34 19
\endlabellist	
	\centering
	\includegraphics[width=0.5\linewidth]{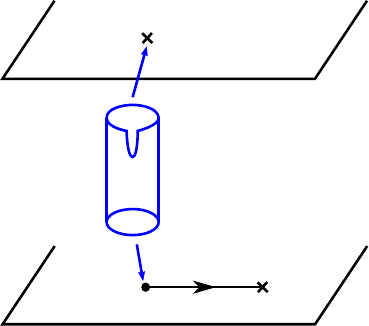}
	\caption{The disk with interior puncture from the shortest chord to the shortest orbit. There is a $T^{2}$-family of holomorphic disks with positive puncture at the chords of $\Lambda_{\ell}$. Evaluation at the boundary puncture gives a map $\ev^{+}$ into the $T^{2}$ Bott-family of chords and at the interior puncture a map $\ev^{-}$ into the $G_{2,4}$ Bott-family of orbits. The configuration that contributes to the differential of $c$ corresponds to a curve in this family with $\ev^{+}$ restricted to lie at the minimum in $T^{2}$. The map $\ev^{-}$ at this curve then gives a Reeb orbit in $G_{2,4}$ at which a flow line starts, and the flow line ends at the minimum in $G_{2,4}$.}
	\label{fig:MorseBott}
\end{figure}

\begin{rmk}
It follows as for the unknot in $S^{3}$ that the augmentation polynomial of $\ell$ is simply
\[ 
A_{\ell}(e^{x},e^{p},Q,\gamma)=
e^{x} - e^{-x} + e^{p} + Q e^{-p} + \gamma.
\]
Note that the conic bundle
\[ 
A_{\ell}(e^{x},e^{p},Q,\gamma)=uv
\]
gives a mirror to local $\C P^{1}\times\C P^{1}$.
\end{rmk}

\subsection{The annulus amplitude for the Hopf link}\label{sec:annulusforHopf}
In this subsection we compute the annulus amplitude for the conormal of the Hopf link in two different ways. Let us first describe with more details what we actually compute. Let $K=K_{1}\cup K_{2}$ denote the Hopf link. Let $L_{j}$, $j=1,2$ denote the Lagrangian conormal of $K_{j}$ in the resolved conifold $X$, and let $x_{j}$ denote the longitude class in $L_{j}$. Let 
$F_{1}(e^{x_{1}},e^{x_{2}},Q)$ denote the count of generalized holomorphic annuli with boundary on $L_{1}\cup L_{2}$. We compute the quantity
\[ 
\frac{\partial F_{1}}{\partial x_{1}\partial x_{2}}=\sum_{k_{1},k_{2},r} C_{k_{1},k_{2},r}\, Q^{r}e^{k_{1}x_{1}}e^{k_{2}x_{2}},
\]
where the sum ranges over $k_{1},k_{2}\ne 0$.
\begin{rmk}
Note that $\frac{\partial F_{1}}{\partial x_{1}\partial x_{2}}$ picks up contributions only from annuli with one boundary component on $L_{1}$ and one on $L_{2}$. Since $L_{1}\cap L_{2}=\varnothing$ it then follows that there are contributions only from formal curves $\Gamma_{\mathbf u}$ of the following form: the graph $\Gamma$ has exactly one vertex where there is a regular holomorphic annulus, there are disks at all other vertices, each disk vertex is connected to the annulus vertex with exactly one edge, and there are no other edges. This means that the count is defined with reference only to bounding chains of the rigid disks.
\end{rmk}

We consider first the count from the perspective of topological strings. The Hopf link consists of two linked unknots. The above calculation for the unknot shows that there is no contribution from large annuli. Consider the contribution to the topological string amplitude from strings stretching between two branes on the conormal of the unknot. Here only annuli contribute and the amplitude is
\[ 
\sum_{n=1}\frac{1}{n} e^{nx_{1}}e^{-nx_{2}}= \log(1-e^{x_{1}}e^{-x_{2}}).
\]
There are many ways to see this; consider for example a small shift along the closed form $d\theta$ and count the annuli arising over the unique closed geodesic in $S^{1}\times\R^{2}$, see Figure \ref{fig:annuli}. Alternatively, see \cite[Section 2.2]{AENV}. 

\begin{figure}[htp]
\labellist
\small\hair 2pt
\pinlabel $T^*S^1$ at 59 137
\pinlabel ${\color{blue} \Gamma_{d\theta}}$ at 140 90
\pinlabel $0$ at 140 53
\endlabellist
		\centering
	\includegraphics[width=0.2\linewidth]{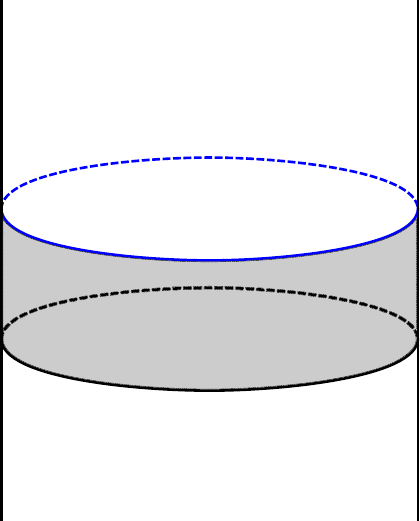}
	\caption{Annuli in $T^{\ast} S^{1}$. The annulus that is an $n$-fold cover contributes by $\frac{1}{n}$.}
	\label{fig:annuli}
\end{figure}

In order to replace the second brane with the conormal of the linked unknot we note that the shift $x_{2}$ of the curve $x_{1}$ is the meridian $p_{2}$ of the second component $K_{2}$ in the limit $K_{2}\to K_{1}$, i.e. it is the curve $x_{2}$ that is filled in by $L_{2}$.
Our annulus amplitude is then expressed as
\[ 
\frac{\partial^{2}}{\partial x_{1}\partial x_{2}} \log(1-e^{x_{1}}e^{-p_{2}(x_{2})}),
\] 
where $p_{2}(x_{2})=\frac{\partial W_{K_{2}}}{\partial x_{2}}$ parameterizes the augmentation variety $\{1-e^{x_{2}}-e^{p_{2}}+Qe^{x_{2}}e^{p_{2}}=0\}$ determined by the disk potential. Here 
\[ 
e^{p_{2}}=\frac{1-e^{x_{2}}}{1-Qe^{x_{2}}}
\]
and the amplitude is
\begin{align*}
&\frac{\partial^{2}}{\partial x_{1}\partial x_{2}}
\Bigl(\log(1-e^{x_{2}}-e^{x_{1}} +Qe^{x_{1}}e^{x_{2}})-\log(1-e^{x_{2}})\Bigr)\\
&\quad=\frac{\partial^{2}}{\partial x_{1}\partial x_{2}}
\log(1-e^{x_{2}}-e^{x_{1}} +Qe^{x_{1}}e^{x_{2}}).
\end{align*}

We consider next the recursive calculation of the same amplitude from the SFT point of view. The Reeb chords of $\Lambda_{K}$ are
\begin{itemize}
\item $a_{12}$ and $a_{21}$ of degree 0;
\item $c_{11}$, $c_{12}$, $c_{21}$, $c_{22}$, $b_{12}$, and $b_{21}$ of degree 1. 
\end{itemize}

The generator of $CE^{\rm lin}(\Lambda_{K_{1}})$ is $c_{11}$ and we depict relevant moduli spaces and amplitudes in Figure \ref{fig:amplitudes}; see Proposition~\ref{prp:HamHopf} below for a justification of these counts.

\begin{figure}[htp]
\labellist
\small\hair 2pt
\pinlabel \underline{Moduli space} at 24 180
\pinlabel \underline{Count} at 122 180
\pinlabel \underline{Moduli space} at 230 180
\pinlabel \underline{Count} at 330 180
\pinlabel $c_{11}$ at 23 158
\pinlabel $c_{11}$ at 23 74
\pinlabel $a_{12}$ at 10 0
\pinlabel $a_{21}$ at 40 0
\pinlabel $c_{21}$ at 220 158
\pinlabel $a_{12}$ at 248 158
\pinlabel $c_{21}$ at 234 74
\pinlabel $a_{21}$ at 234 0
\pinlabel $1-e^{x_1}-e^{p_1}+Qe^{x_1}e^{p_1}$ at 122 120
\pinlabel $1$ at 122 35
\pinlabel $(1-e^{x_1})e^{p_1}e^{-p_2}$ at 330 120
\pinlabel $Qe^{x_1}e^{p_1}-e^{p_1}e^{-p_2}$ at 330 35
\endlabellist
	\centering
\includegraphics[width=0.8\linewidth]{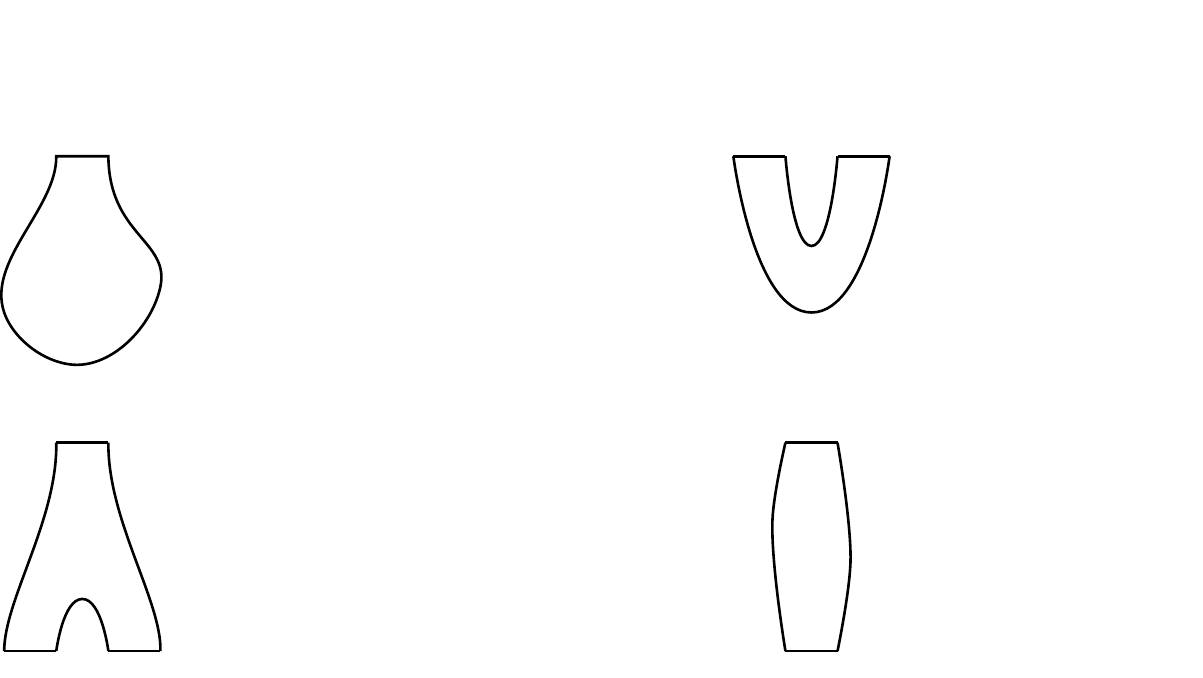}
	\caption{Moduli spaces used for computing the annulus amplitude for the Hopf link.}
	\label{fig:amplitudes}
\end{figure}

From the two rightmost counts in Figure \ref{fig:amplitudes}, the amplitude for disks with two positive punctures at $a_{12}$ and $a_{21}$ equals
\[ 
\frac{(1-e^{x_{1}})e^{-p_{2}}}{Qe^{x_{1}}-e^{-p_{2}}}.
\]

Looking now at the boundary of the 1-dimensional space of annuli with positive puncture at $c_{11}$, we find that
\[
\frac{\partial F_{1}}{\partial x_{1}}=\frac{1}{Qe^{x_{1}}e^{p_{1}}-e^{p_{1}}}
\frac{(1-e^{x_{1}})e^{-p_{2}}}{Qe^{x_{1}}-e^{-p_{2}}}
= \frac{e^{-p_{2}}}{Qe^{x_{1}}-e^{-p_{2}}}=
\frac{Q^{-1}e^{-x_{1}}e^{-p_{2}}}{1-Q^{-1}e^{-x_{1}}e^{-p_{2}}},
\]
where the second equality uses the equation for the augmentation variety of $K_{1}$.
Changing variables in the homology coefficients, $Q^{-1}e^{-x_{1}}\mapsto e^{x_{1}}$, we find that
\[ 
\frac{\partial F_{1}}{\partial x_{1}}=\frac{e^{x_{1}}e^{-p_{2}}}{1-e^{x_{1}}e^{-p_{2}}},
\]
and hence the amplitude is
\[ 
\frac{\partial^{2}}{\partial x_{1}\partial x_{2}}\log(1-e^{x_{1}}e^{-p_{2}})=
\frac{\partial^{2}}{\partial x_{1}\partial x_{2}}\log(1-e^{x_{1}}-e^{x_{2}}+Qe^{x_{1}}e^{x_{2}})
\] 
in agreement with the above.

\begin{rmk}
The change of variables reflect our choice of orientations, and a natural change of lift of this change of variables to a relative homology class, where the longitude $x$ on the unknot is capped by a half-sphere, so changing from $x$ to $-x$ gives a whole sphere with negative orientation, $Q^{-1}$.
\end{rmk}

\subsection{Legendrian SFT and recursion for the Hopf link}
\label{sec:Hopfrec}

In this subsection we will use a computation of Legendrian SFT to deduce the recursion for the Hopf link. We will find that the answer agrees with the physics computation from \cite{AENV}. Our computation can be viewed as a precursor to the more involved computation for the trefoil that occupies Section~\ref{sec:trefoil}. We note that the computation we give here is a sketch and we leave out some details; what we present is an argument for how to use Legendrian SFT to obtain a result that agrees with the known Hopf recursion.

Let $K$ denote the Hopf link, represented by a braid around the unknot $U$ with two (half) twists. Since $K$ lies very close to $U$, its Legendrian conormal $\Lambda_{K}$ lies in a small neighborhood of $\Lambda_{U}$ that can be identified with the $1$-jet space $J^{1}(\Lambda_{U})$. It is easy to check that the projection to the base in the 1-jet space gives a degree two covering map $\Lambda_{K}\to\Lambda_{U}$. 

We next determine the Reeb chords of $\Lambda_{K}$. These are of two types: short, which are entirely contained in $J^{1}(\Lambda_{U})$, and long, which are not. 
We can arrange that there are only four short Reeb chords, see \cite{EENS}, and the long chords are close to the chords of the unknot. This then results in the following Reeb chords of $\Lambda_{K}$:  
$a_{12},a_{21}$ in degree $0$; $b_{12},b_{21},c_{11},c_{12},c_{21},c_{22}$ in degree $1$; and the remainder in degree $2$. For our computation, as with computations of augmentation polynomials, it suffices to calculate the Hamiltonian of the degree $1$ Reeb chords. Each of these is a differential operator acting on the space of power series in $a_{12},a_{21}$ with coefficients in $\lambda_1,\mu_1,\lambda_2,\mu_2,Q,q$. Here we have
\[
\lambda_{j} = e^{x_{j}} \hspace{3ex}
\mu_{j} = e^{p_{j}} \hspace{3ex}
p_{j} = g_s \frac{\partial}{\partial x_{j}} \hspace{3ex}
q = e^{g_s}
\]
and $\mu_{j}\lambda_{j} = q\lambda_{j}\mu_{j}$. 

The following result lays out the parts of the Hamiltonian that we will use in our computation. If $c$ is a degree $1$ Reeb chord we write
\[ 
H(c)=\mathbf{H}^{c}
\]
for the portion of $\mathbf{H}$ corresponding to curves whose positive punctures are $c$ and some collection of degree $0$ Reeb chords (see Section~\ref{ssec:lsft}).

\begin{prp}
We have
\label{prp:HamHopf}
\begin{align*}
H(c_{11}) &= 1-\lambda_1-\mu_1+Q\lambda_1\mu_1+q^{-1} \d_{a_{12}} \d_{a_{21}} + \mathcal{O}(a) \\
H(c_{22}) &= 1-\lambda_2-\mu_2+Q\lambda_2\mu_2+q^{-1} Q \lambda_2\mu_2 \d_{a_{12}}  \d_{a_{21}} 
+ \mathcal{O}(a)\\
H(c_{12}) &= (\mu_2\mu_1^{-1}-Q\lambda_2\mu_2)\d_{a_{12}} + (1-q)(1-\lambda_2)a_{21}+ \mathcal{O}(a^2) \\
H(c_{21}) &= (q^{-1}Q \lambda_1 \mu_1 - q^{-1}\mu_1 \mu_2^{-1})\d_{a_{21}}
+ (q-1)(1-q\lambda_1)\mu_1\mu_2^{-1} a_{12} \\
&\qquad+ (q-q^{-1}) Q \lambda_1 \mu_1 a_{12} \d_{a_{12}} \d_{a_{21}}+ \mathcal{O}(a),
\end{align*}	
where $\mathcal{O}$ represents total combined order in $a_{12}$ and $a_{21}$.
\end{prp}

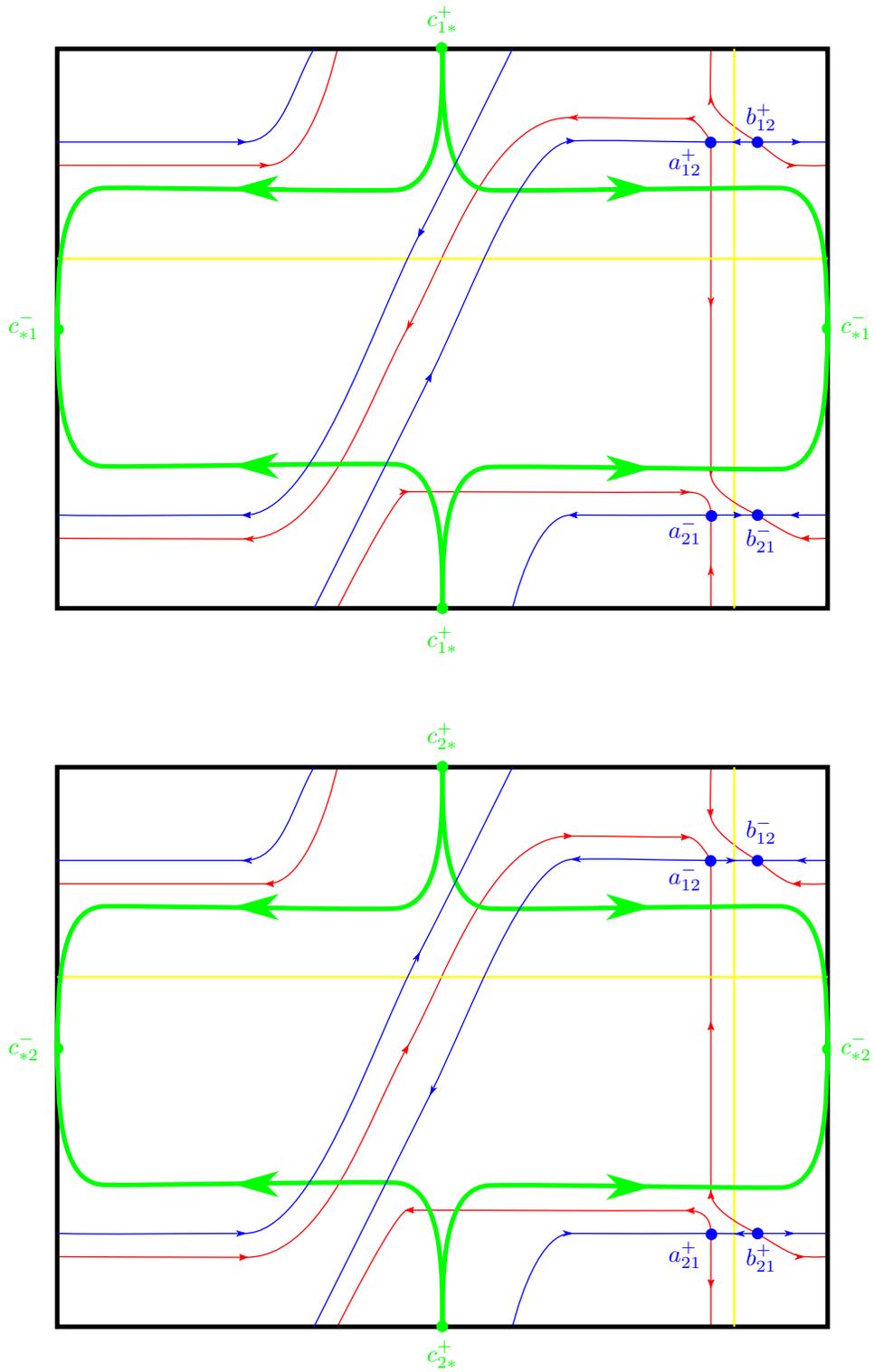
\begin{figure}[htp]
\labellist
\small\hair 2pt
\pinlabel ${\color{green} c_{\ast 1}^-}$ at -20 692
\pinlabel ${\color{green} c_{\ast 1}^-}$ at 515 692
\pinlabel ${\color{green} c_{1\ast}^+}$ at 249 889
\pinlabel ${\color{green} c_{1\ast}^+}$ at 249 490
\pinlabel ${\color{green} c_{2 \ast}^+}$ at 249 428
\pinlabel ${\color{green} c_{2\ast}^+}$ at 249 30
\pinlabel ${\color{green} c_{\ast 2}^-}$ at -20 228
\pinlabel ${\color{green} c_{\ast 2}^-}$ at 515 228
\pinlabel ${\color{blue} a_{12}^+}$ at 405 797
\pinlabel ${\color{blue} b_{12}^+}$ at 454 826
\pinlabel ${\color{blue} a_{21}^-}$ at 405 560
\pinlabel ${\color{blue} b_{21}^-}$ at 454 555
\pinlabel ${\color{blue} a_{12}^-}$ at 405 337
\pinlabel ${\color{blue} b_{12}^-}$ at 454 367
\pinlabel ${\color{blue} a_{21}^+}$ at 405 94
\pinlabel ${\color{blue} b_{21}^+}$ at 454 92
\endlabellist	
\centering
	\includegraphics[width=0.9\linewidth]{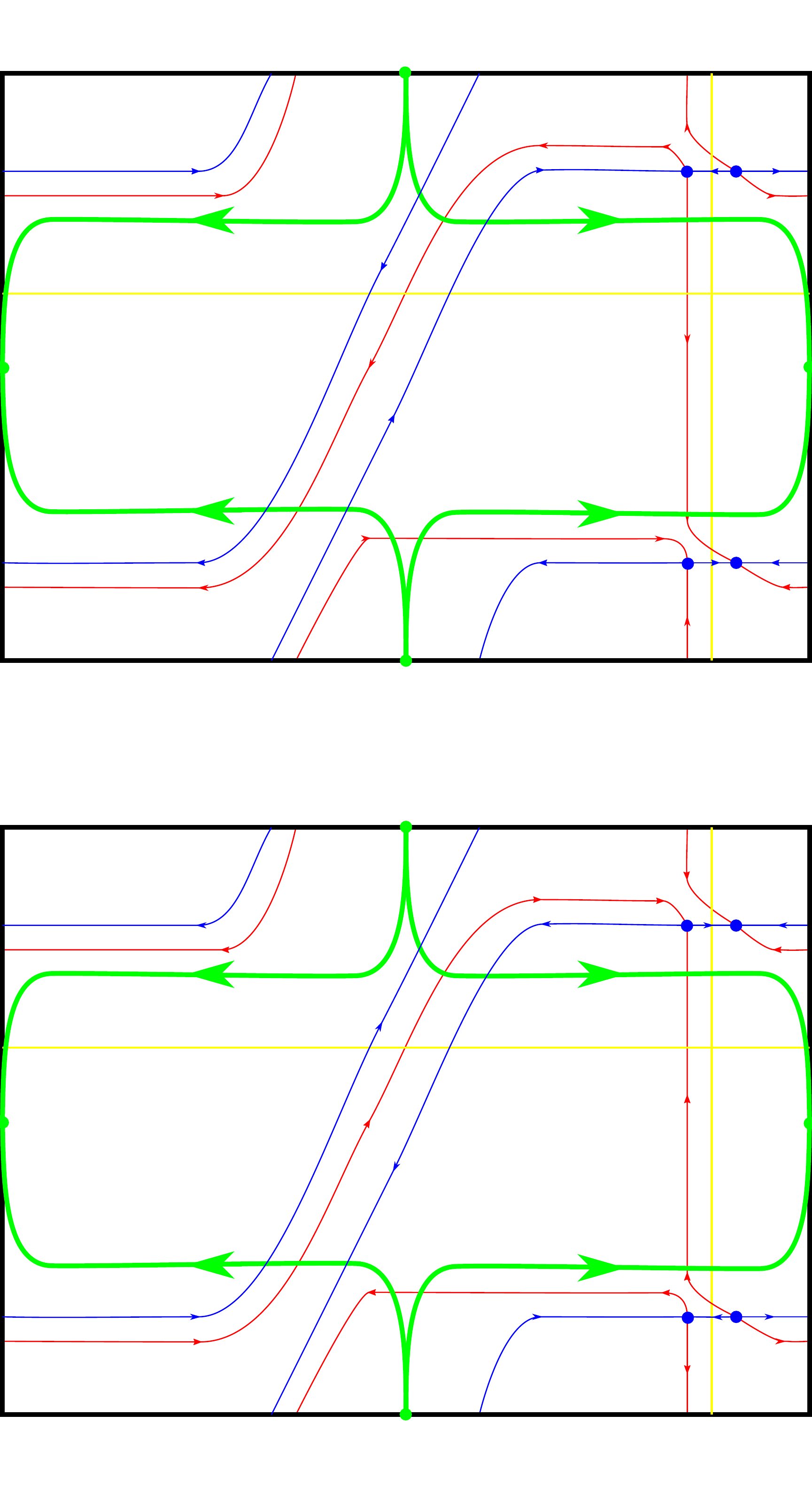}
	\caption{Flow lines and big disk boundaries for the Hopf link.}
	\label{fig:Hopfflow}
\end{figure}

\begin{rmk}
\label{rmk:Hopf}
The proof of Proposition \ref{prp:HamHopf} depends on Figure \ref{fig:Hopfflow}, which lays out the geometry of the terms contributing to the Hamiltonian for the Hopf link. Before we proceed to the proof, we describe what is depicted in this figure.	
The two squares (with opposite sides identified) are the Legendrian tori $\Lambda_{K_1}$ and $\Lambda_{K_2}$. The green dots are the endpoints of the Reeb chords $c_{ij}$, labeled by $+$ and $-$ depending on whether they are the positive or negative endpoint; for instance, $c_{1\ast}^+$ denotes the positive endpoint of (both) the Reeb chords $c_{11}$ and $c_{12}$.
The blue dots are the endpoints of the Reeb chords $b_{ij}$ and $a_{ij}$, as labeled. 
 The green curves are the boundaries of big disks. The yellow curves are reference curves; intersection numbers with the vertical (resp.\ horizontal) curves give the exponents of $e^{x_i}$ (resp.\ $e^{p_i}$) associated to each holomorphic curve.
 
The blue flow lines correspond to boundaries of disks with a negative puncture at $a_{ij}$ (stable manifold for the negative gradient flow of the positive function difference) and the red flow lines to boundaries of disks with a positive puncture at $a_{ij}$ (unstable manifold for the negative gradient flow of the positive function difference). As explained in \cite{EENS}, the stable manifolds follow the vector joining the two components of the braid; since the Hopf link consists of one full twist on two strands, the stable manifolds wind one full time around the meridional direction of the Legendrian tori.
For the unstable manifolds, note that there are flows between the endpoints for Reeb chords $a_{ij}$ and $b_{ij}$ that roughly follow the (vertical) fiber direction.

Figure \ref{fig:Hopfflow} gives rise to generalized holomorphic curves as listed in Figures \ref{fig:HHopf11and22}, \ref{fig:HHopf12}, and \ref{fig:HHopf21}. The curves drawn in these figures are the boundaries of the holomorphic curves that arise from big disks along with flow lines from Figure \ref{fig:Hopfflow}. Self intersections of the boundaries correspond to linking numbers and contribute factors $q^{\frac12\slk}$ to the corresponding curve count, see Section \ref{sec:sftpotential}. Note that there are also contributions from intersections with capping paths. In order not to clutter the pictures too much we have drawn the capping paths (shown in magenta) only when there are such contributions and left them out otherwise.

We also recall from Remark~\ref{r:factorsforextrapositive} that we count disks with two positive punctures in general with a factor of $\pm e^{rg_s}(e^{\frac12 g_s}-e{-\frac12 g_s})^m$. For our computation here we have $m=1$ (this counts the additional positive puncture beyond the input for $H$) and $r=\frac12\slk$, and $\slk = \pm 1$ for the disks that we will consider, for an overall factor of $\pm (q-1)$ or $\pm (1-q^{-1})$.
\end{rmk}

\begin{rmk}\label{r:signextrapos}
We emphasize that the proof of Proposition~\ref{prp:HamHopf} is a sketch, with some details omitted. In particular, exact $q$-contributions are sometimes hard to see directly from pictures, since they depend on choices of capping disks and how these intersect $C_{K}$. Such choices affect terms in the Hamiltonian but after basic choices are made other terms are determined. We plan to study particular choices suitable for algorithmic computations elsewhere. Also, we do not justify the overall signs of some terms; in particular, the disks with two positive punctures contribute (a power of $q$ times) $\pm (q-1)$ but we do not justify the individual choices of $\pm$.
\end{rmk}

\begin{proof}[Sketch of proof of Proposition~\ref{prp:HamHopf}]
Flow trees and big disk boundaries for the Hopf link are depicted in Figure \ref{fig:Hopfflow}.
We first note that the disks without negative punctures in $H(c_{11})$ and $H(c_{22})$ are the disks of the unknot components. Taking capping disks exactly as for the unknot components, we find that these disks contribute exactly as for the unknot. We therefore concentrate on the other disks. The subtlest point is determining the powers of $q$ attached to each term, and we focus on this.

For $H(c_{11})$ and $H(c_{22})$, there is one additional disk apiece, depicted in Figure \ref{fig:HHopf11and22}. The disk on the left corresponds to the $\partial_{a_{12}} \partial_{a_{21}}$ term in $H(c_{11})$, while the disk on the right corresponds to the $\lambda_2\mu_2\partial_{a_{12}} \partial_{a_{21}}$ term in $H(c_{22})$. In both cases there are two crossings with capping paths which cancel, leaving no contribution to powers of $q$ from intersections with capping paths. However, each disk has two contributions of $q^{-\frac12}$ near the endpoints of $a_{12}$ and $a_{21}$, compare the proof of Lemma \ref{l:capping}, and this results in an overall factor $q^{-1}$ for both disks.

\begin{figure}
\labellist
\small\hair 2pt
\pinlabel $\partial_{a_{12}}\partial_{a_{21}}$ at 72 0
\pinlabel $Q\lambda_2\mu_2\partial_{a_{12}}\partial_{a_{21}}$ at 278 0
\endlabellist
	\centering
	\includegraphics[width=0.6\linewidth]{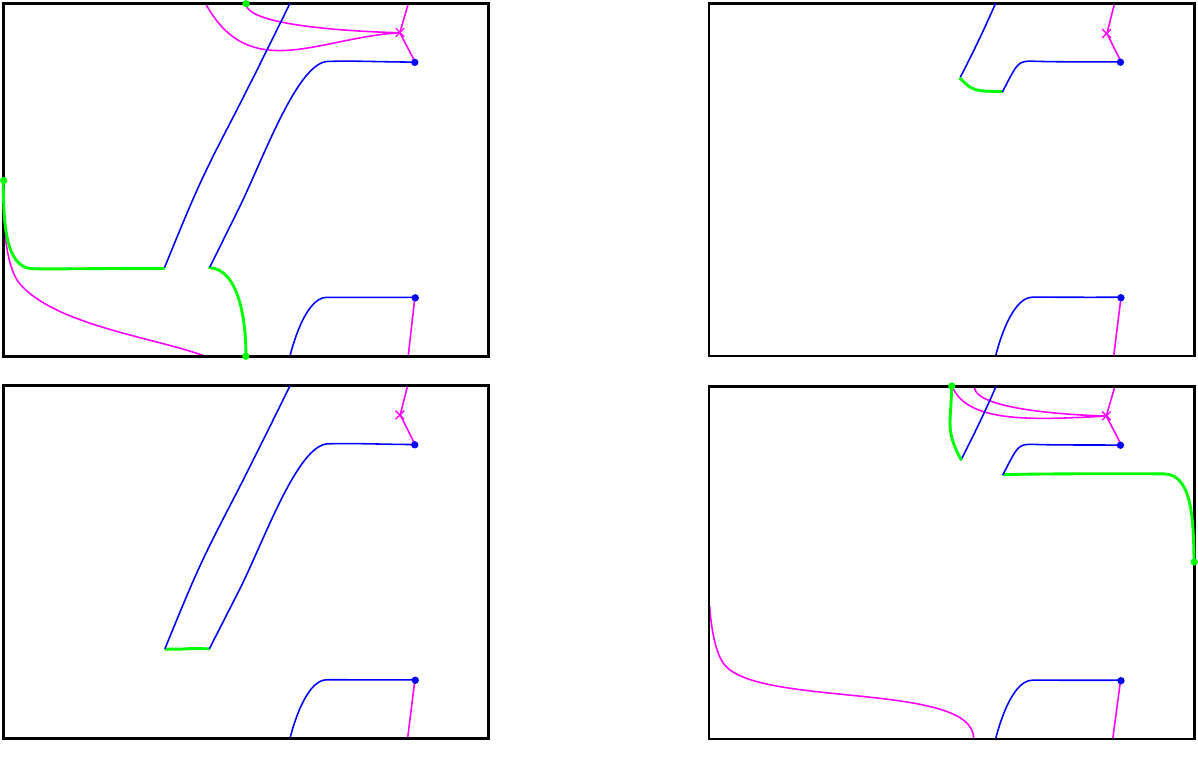}
	\caption{
	Disks for the terms $\partial_{a_{12}} \partial_{a_{21}}$ in $H(c_{11})$ (left) and $\lambda_2\mu_2\partial_{a_{12}} \partial_{a_{21}}$ in $H(c_{22})$ (right).	}
	\label{fig:HHopf11and22}
\end{figure} 

\begin{figure}
\labellist
\small\hair 2pt
\pinlabel  $\mu_2\mu_1^{-1}\partial_{a_{12}}$ at 72 246
\pinlabel $Q\lambda_2\mu_2\partial_{a_{12}}$ at 258 246
\pinlabel $a_{21}$ at 72 0
\pinlabel $\lambda_2a_{21}$ at 258 0
\endlabellist	
\centering
	\includegraphics[width=0.6\linewidth]{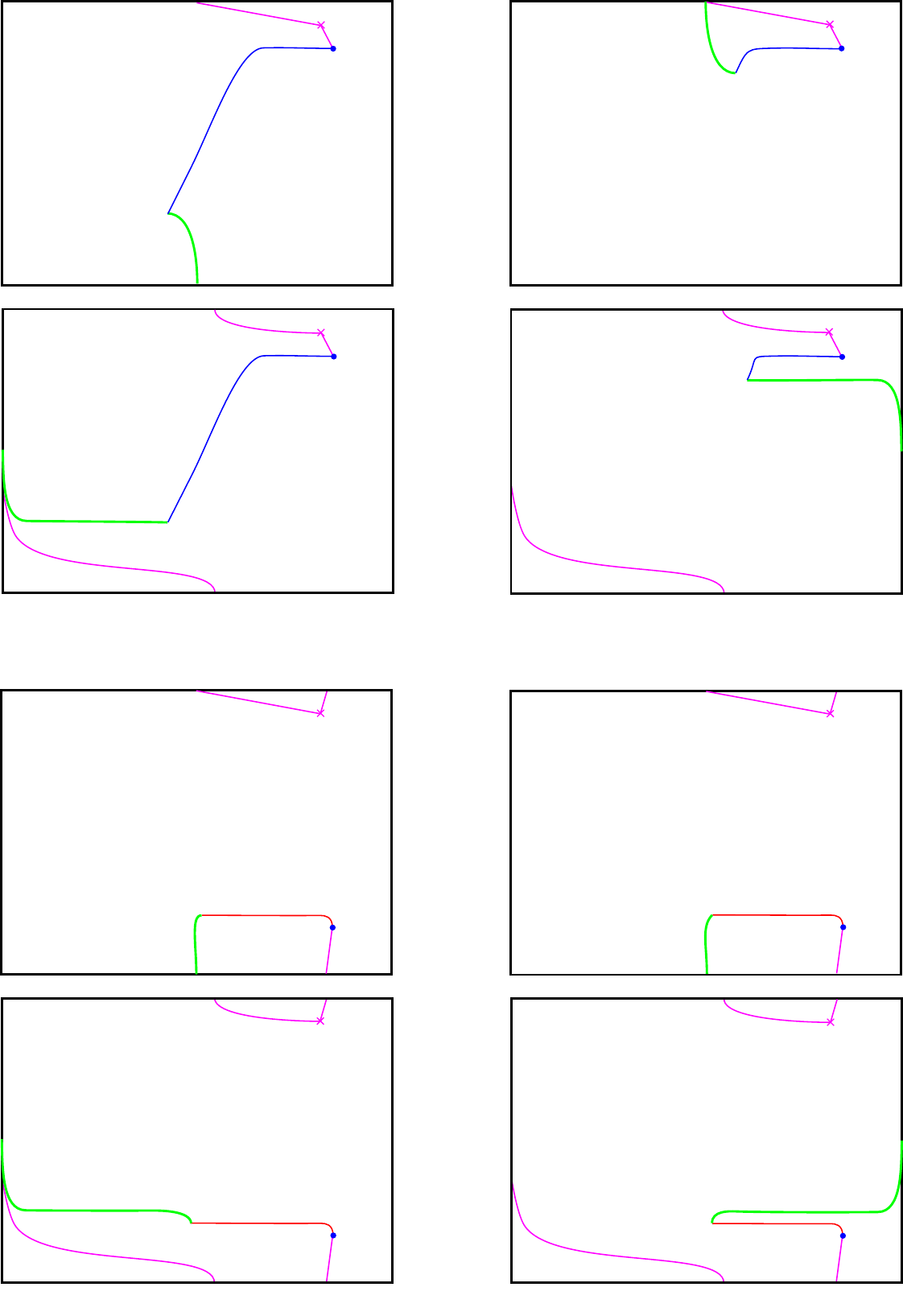}
	\caption{Disks for $H(c_{12})$.
}
	\label{fig:HHopf12}
\end{figure}
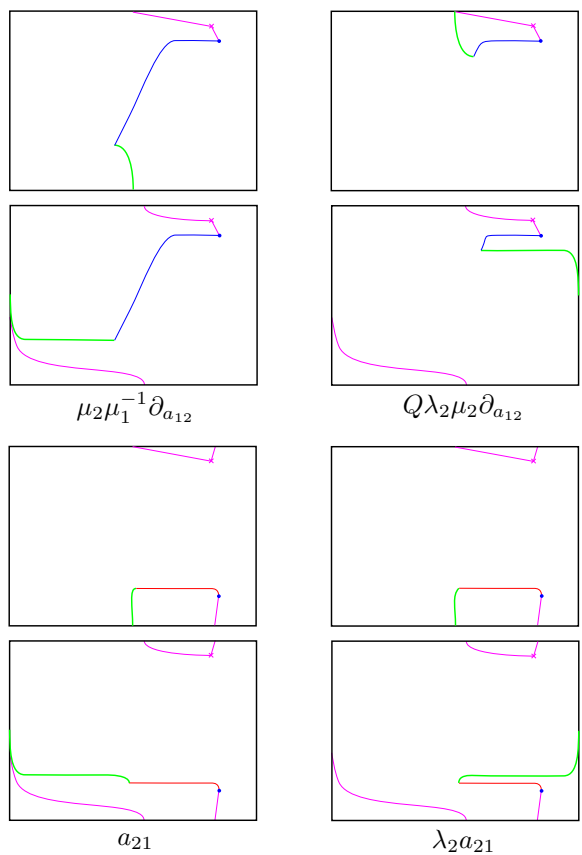

The disks for $H(c_{12})$ are shown in Figure \ref{fig:HHopf12}; the two disks with one positive puncture are to the left and the two disks with two positive punctures are to the right. There are no crossings with capping paths. The curves with two positive punctures both come with the factor $1-q$. (Comparing to the calculations of $H(c_{11})$ and $H(c_{22})$, the contribution $q^{-\frac12}$ at the negative puncture $a_{12}$ can be considered an overall factor and the factor of the disks with two positive punctures becomes $q^{\frac12}(q^{\frac12}-q^{-\frac12})$, see Remark \ref{r:factorsforextrapositive}.)

\begin{figure}[htp]
\labellist
\small\hair 2pt
\pinlabel $Q\lambda_1\mu_1\partial_{a_{21}}$ at 256 914
\pinlabel $\mu_1\mu_2^{-1}\partial_{a_{21}}$ at 908 914
\pinlabel $\mu_1\mu_2^{-1}a_{12}$ at 1546 914
\pinlabel $\lambda_1\mu_1\mu_2^{-1}a_{12}$ at 256 46
\pinlabel $Q\lambda_1\mu_1a_{12}\partial_{a_{12}}\partial_{a_{21}}$ at 908 46 
\pinlabel $Q\lambda_1\mu_1a_{12}\partial_{a_{12}}\partial_{a_{21}}$ at 1546 46
\endlabellist
	\centering
	\includegraphics[width=\linewidth]{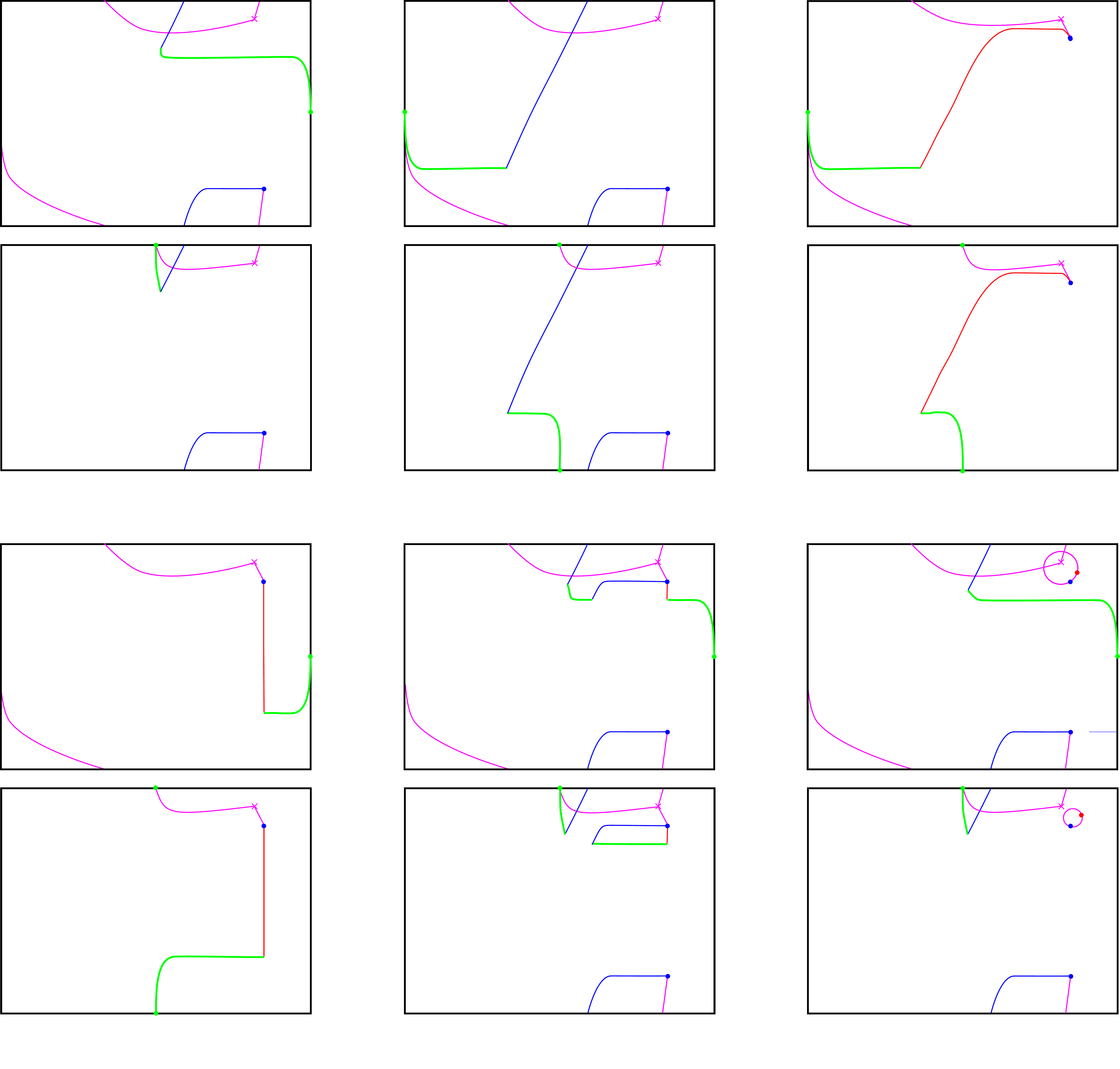}
	\caption{Disks for $H(c_{21})$.}
	\label{fig:HHopf21}
\end{figure}
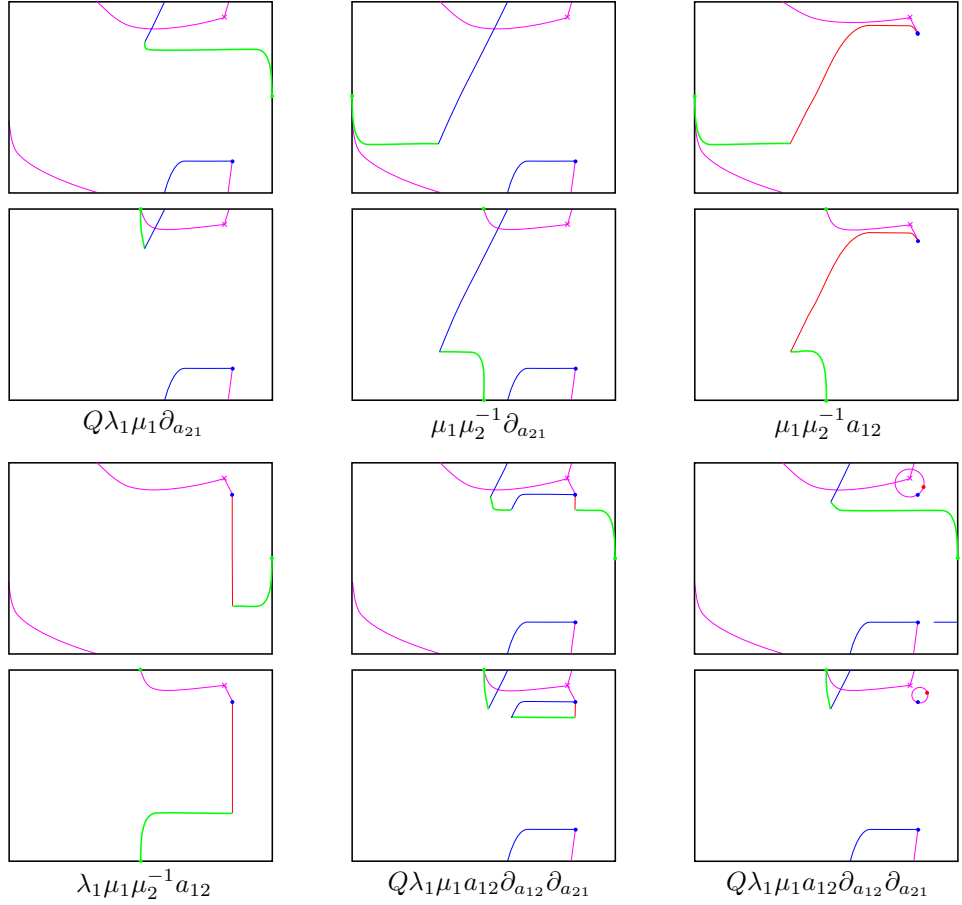
Finally, consider $H(c_{21})$. The $6$ curves that contribute are depicted in Figure \ref{fig:HHopf21}. 
There are two disks with one positive puncture, corresponding to terms $Q\lambda_1\mu_1\partial_{a_{21}}$ and $\mu_1\mu_2^{-1}\partial_{a_{21}}$. For each of these, the crossings together with the contribution from the negative puncture at $a_{21}$ give coefficients $q^{-1}$. The remaining disks have two positive punctures. There are two disks with degree $0$ positive puncture at $a_{12}$, contributing $\mu_1\mu_2^{-1}a_{12}$ and $\lambda_1\mu_1\mu_2^{-1}a_{12}$. The former comes with coefficient $q-1$ as in the calculation of $H(c_{12})$. Comparing the latter disk to the former disk, we find that near the punctures the disk boundary enters from the opposite directions, and in combination with capping paths this adds a factor $q^{1}$; thus the $\lambda_1\mu_1\mu_2^{-1}a_{12}$ disk comes with total coefficient $q(1-q)$.

The two remaining contributions to $H(c_{21})$ both correspond to terms $Q\lambda_1\mu_1a_{12}\partial_{a_{12}}\partial_{a_{21}}$. Of these two, the left one in Figure~\ref{fig:HHopf21} is a single disk with one extra positive puncture at $a_{12}$ and two negative punctures at $a_{12}$ and $a_{21}$, while the right one is a disk with a negative puncture at $a_{21}$ linked with a trivial strip over $a_{12}$. 
Both disks have two crossings with capping paths, contributing a total of $q$. In the first disk there are two negative punctures and one positive with a total capping contribution $q^{\frac12-\frac12-\frac12}$; overall the first disk comes with coefficient $q^{\frac12}(q^{\frac12}-q^{-\frac12})=q-1$. 
The second disk is obtained by inserting a trivial strip that links the disk into which it is inserted with linking number $-1$. 
The lower right picture in Figure~\ref{fig:HHopf21} shows the boundary of the trivial strip as two dots connected by the shorter arc in the circle; the longer arcs in the circle represent capping paths of the trivial strip, which overall link with the capping path of the big disk once. 
The total capping contribution from the trivial negative punctures is $q^{-1}$, the perturbation chosen then gives no capping contribution at the positive puncture, and overall the total coefficient for the second disk is $qq^{-1}(1-q^{-1})=1-q^{-1}$. When we add the contributions from the first and second disks, we conclude that $Q\lambda_1\mu_1a_{12}\partial_{a_{12}}\partial_{a_{21}}$ appears with coefficient $q-q^{-1}$.
\end{proof}

We now use Proposition~\ref{prp:HamHopf} to calculate the recursion for the Hopf link. The recursion consists of not a single relation, as for a single-component knot, but several relations: in this case, three relations $A,B,C$ in $\lambda_1, \mu_1, \lambda_2, \mu_2, Q, q$. Recall that these coefficients commute except that $\mu_1\lambda_1 = q\lambda_1\mu_1$ and $\mu_2\lambda_2 = q\lambda_2\mu_2$.

The relations are obtained by eliminating $\d_{a_{12}}$ and $\d_{a_{21}}$ in the relations $H(c_{11}), H(c_{22}), H(c_{12}), H(c_{21})$ given by the Hamiltonian in Proposition~\ref{prp:HamHopf}. For instance, a linear combination of $H(c_{11})$ and $H(c_{22})$ will cancel the $\d_{a_{12}} \d_{a_{21}}$ terms in each: define
\begin{align*}
A &:= H(c_{11}) - q Q^{-1}\lambda_2^{-1}\mu_2^{-1}H(c_{22}) \\
&= -\lambda_1+q Q^{-1}\lambda_2^{-1}-(1-Q\lambda_1)\mu_1+(1-q\lambda_2^{-1})Q^{-1}\mu_2^{-1} + \mathcal{O}(a).
\end{align*}
Next note that
\[
\d_{a_{21}} H(c_{12}) = (1-q)(1-\lambda_2)+(\mu_2\mu_1^{-1}-Q\lambda_2\mu_2)\d_{a_{12}}
\d_{a_{21}} + \mathcal{O}(a)
\]
and then define
\begin{align*}
B &:= (\mu_1^{-1}-Q\lambda_2) H(c_{22}) - q^{-1}Q\lambda_2 \d_{a_{21}} H(c_{12}) \\
&= (\mu_1^{-1}-Q\lambda_2)(1-\lambda_2-\mu_2+Q \lambda_2\mu_2)-(q^{-1}-1)Q\lambda_2(1-\lambda_2)
+\mathcal{O}(a) \\
&= (1-\lambda_2-\mu_2+Q\lambda_2\mu_2)(\mu_1^{-1}-q^{-1}Q\lambda_2) + \mathcal{O}(a).
\end{align*}
Similarly define
\begin{align*}
C &:= (q Q \lambda_1-\mu_2^{-1})H(c_{11})-\mu_1^{-1}\d_{a_{12}} H(c_{21}) \\
&= (q-q \lambda_1-\mu_1+Q \lambda_1\mu_1)(Q\lambda_1-\mu_2^{-1}) + \mathcal{O}(a).
\end{align*}

We drop the $\mathcal{O}(a)$ terms in $A,B,C$ to obtain the recurrence relations for the Hopf link:
\begin{align*}
A&=-\lambda_1+q Q^{-1}\lambda_2^{-1}-(1-Q\lambda_1)\mu_1+(1-q\lambda_2^{-1})Q^{-1}\mu_2^{-1}
\\
B &= (1-\lambda_2-\mu_2+Q\lambda_2\mu_2)(\mu_1^{-1}-q^{-1}Q\lambda_2)\\
C &= (q-q \lambda_1-\mu_1+Q \lambda_1\mu_1)(Q\lambda_1-\mu_2^{-1}) .
\end{align*}

In \cite[Section 4.5]{AENV}, the recurrence relations for the Hopf link were calculated using techniques from Chern--Simons and topological string theory, with the following result:
\begin{align*}
\mathcal{A}_1 &= -\lambda_1+\lambda_2-(1-Q\lambda_1)\mu_1+(1-Q\lambda_2)\mu_2 \\
\mathcal{A}_2 &= (1-q^{-1}\lambda_2-(1-Q\lambda_2)\mu_2)(\mu_1-\lambda_2) \\
\mathcal{A}_3 &= (1-q^{-1}\lambda_1-(1-Q\lambda_1)\mu_1)(\mu_2-\lambda_1).
\end{align*}
We now check that our relations agree with the ones from \cite{AENV} up to a change of coordinates. Indeed, replace
\[
\lambda_2 \mapsto Q^{-1}\lambda_2^{-1}, \qquad \mu_2 \mapsto Q^{-1}\mu_2^{-1}
\]
in the expressions for $A,B,C$, to obtain
\begin{align*}
A&=-\lambda_1+q\lambda_2-(1-Q\lambda_1)\mu_1+(1-qQ\lambda_2)\mu_2
\\
B &= -\mu_1^{-1}\lambda_2^{-1}(1-q\lambda_2-q^{-1}\mu_2+Q\lambda_2\mu_2)(q^{-1}\mu_1-\lambda_2) \\
C &= Q(q-q \lambda_1-\mu_1+Q \lambda_1\mu_1)(\lambda_1-\mu_2) .
\end{align*}
If we drop the units $-\mu_1^{-1}\lambda_2^{-1},Q$ in front of $B,C$, change variables
\[
\lambda_1 \mapsto q \lambda_1, \qquad 
\mu_1 \mapsto q \mu_1, \qquad \lambda_2 \mapsto \lambda_2,
\qquad \mu_2 \mapsto q \mu_2, \qquad Q \mapsto q^{-1}Q,
\]
and finally replace $q$ by $q^{-1}$, then $A,B,C$ precisely become $q^{-1}\cdot \mathcal{A}_1,\mathcal{A}_2,q^{-2}\cdot\mathcal{A}_3$ respectively. This verifies that our relations are in agreement with \cite{AENV}.

\section{Legendrian SFT and recursion for the trefoil}
\label{sec:trefoil}

In this section we present an example of a more involved computation. We use Legendrian SFT to produce a $q$-deformed version of the augmentation polynomial for the right-handed trefoil, and then check that this agrees with the recurrence relation for the colored HOMFLY-PT polynomials for the trefoil. 

As in Section~\ref{sec:Hopfrec}, we need to include a disclaimer here. The formulas we write for the Hamiltonian in Legendrian SFT depend on a putative choice of geometric additional data---for example, choices of capping paths and perturbations of coincident flow lines can affect formulas---and we do not carefully justify this part of the computation leading to our formulas. In addition, we do not justify all of the signs of terms in the Hamiltonian; some of them come directly from Legendrian contact homology and are determined by the work in \cite{EENS}, but others are assigned without proof.
Finally, some parts of the calculations 
depend on our conjectures, see e.g.~Remark \ref{r:factorsforextrapositive}. 
Rather than a rigorous derivation, this section should be viewed as a first step towards a combinatorial description of the Legendrian SFT Hamiltonian and a sketch of how to use it to extract the recursion relation.    

\subsection{Hamiltonian for the trefoil}
Let $T$ denote the right-handed trefoil, represented by a braid around the unknot $U$ with three (half) twists. We view this as with the Hopf link in Section~\ref{sec:Hopfrec}: place $\Lambda_T$ in a neighborhood $J^1(\Lambda_U)$ of $\Lambda_U$. 
Exactly as for the Hopf link, $\Lambda_T$ has the following Reeb chords: $a_{12},a_{21}$ in degree $0$; $b_{12},b_{21},c_{11},c_{12},c_{21},c_{22}$ in degree $1$; and the remainder in degree $2$. The coefficients of the Hamiltonian are slightly simpler than for the Hopf link since the trefoil has a single component: they are $\lambda,\mu,Q,q$, where
\[
\lambda = e^x \hspace{3ex}
\mu = e^p \hspace{3ex}
p = g_s \frac{d}{dx} \hspace{3ex}
q = e^{g_s}
\]
and $\mu\lambda = q\lambda\mu$. 

The following result lays out the parts of the Hamiltonian that we will use in our computation. As for the Hopf link, we will only need the Hamiltonian for degree $1$ Reeb chords; if $c$ is a degree $1$ Reeb chord we write
\[ 
H(c)=\mathbf{H}^{c}.
\]

\begin{proposition}
We have \label{prop:trefoilHam}
\begin{align*}
H(b_{12}) &= \lambda^{-1} \d_{a_{12}} - \d_{a_{21}} +\mathcal{O}(a)\\
H(c_{11}) &= 
\lambda\mu - q^{-1}\lambda -((1+q^{-1})Q-\mu) \d_{a_{12}} - Q \d^2_{a_{12}} \d_{a_{21}} + \mathcal{O}(a) \\
H(c_{21}) &= Q - \mu + \lambda\mu \d_{a_{21}} + Q \d_{a_{12}} \d_{a_{21}} +  (q^{-1}-1) \lambda a_{12}  \\
& \qquad+ (q^{-1}-1) Q a_{12} \d_{a_{12}}
+ \mathcal{O}(a^2) \\
H(c_{22}) &= \mu-1-Q \d_{a_{21}}+\mu \d_{a_{12}} \d_{a_{21}} +(q-1) Q a_{12} \\
&\qquad + (q-1) \mu a_{12} \d_{a_{12}} + \mathcal{O}(a^2),
\end{align*}
where $\mathcal{O}$ represents total combined order in $a_{12}$ and $a_{21}$.
\end{proposition}

\begin{proof}[Sketch of proof]
We prove this proposition by reducing the holomorphic curve count to combinatorics. We use the strategy recalled in Section \ref{sec:curvesatinfinity}: consider the limit $T\to U$ in which the braid of $T$ approaches the unknot. In this limit, as mentioned above, we can view $\Lambda_T$ as a subset of $J^1(\Lambda_U)$,
where a neighborhood of the 0-section in $J^{1}(\Lambda_{U})$ is identified with a neighborhood of $\Lambda_{U}$ in $ST^{\ast}S^{3}$. Since $T$ is a $2$-fold cover of $U$, $\Lambda_{T}$ is graphical over $\Lambda_{U}$ with two sheets, and 
there is only one function difference between the two sheets of $\Lambda_{T}$. This function determines a Morse flow along $\Lambda_{U}$. As explained in Section \ref{sec:curvesatinfinity}, the generalized holomorphic curves of $\Lambda_{T}$ can then be determined from knowledge of the disks of $\Lambda_{U}$ this Morse flow determined by the degree $2$ Legendrian $\Lambda_{T}\subset J^{1}\Lambda_{U}$.   

To count the generalized holomorphic curves via flow trees, we refer to Figure~\ref{fig:trefoilflow}. This shows the torus $\Lambda_T$ (the full rectangle, with opposite sides identified) as a $2$-fold cover of $\Lambda_U$ (the two smaller rectangles). The yellow curves are reference curves, where intersections with the vertical (resp.\ horizontal) curve count powers of $\lambda$ (resp.\ $\mu$). As in the calculation for the Hopf link (see Remark~\ref{rmk:Hopf}), the green and blue dots represent endpoints of Reeb chords as labeled. Also as in Remark~\ref{rmk:Hopf}, the blue and red curves are flow lines for the function difference between the two sheets, beginning and ending at endpoints of Reeb chords $a_{12},a_{21},b_{12},b_{21}$, which are critical points of the function difference. Note that the long flow lines between $b$ and $a$ critical points wind $\frac{3}{2}$ times around the meridional direction, corresponding to the $3$ half twists of the trefoil $T$.

\begin{figure}
\labellist
\small\hair 2pt
\pinlabel ${\color{green} c_{\ast 1}^-}$ at 156 60
\pinlabel ${\color{green} c_{1\ast}^+}$ at 78 118
\pinlabel ${\color{green} c_{1\ast}^+}$ at 78 2
\pinlabel ${\color{green} c_{2 \ast}^+}$ at 218 118
\pinlabel ${\color{green} c_{2\ast}^+}$ at 218 2
\pinlabel ${\color{green} c_{\ast 2}^-}$ at 0 60
\pinlabel ${\color{green} c_{\ast 2}^-}$ at 295 60
\pinlabel ${\color{blue} a_{12}^+}$ at 122 89
\pinlabel ${\color{blue} b_{12}^+}$ at 137 100
\pinlabel ${\color{blue} a_{21}^-}$ at 123 21
\pinlabel ${\color{blue} b_{21}^-}$ at 137 30
\pinlabel ${\color{blue} a_{12}^-}$ at 260 89
\pinlabel ${\color{blue} b_{12}^-}$ at 275 100
\pinlabel ${\color{blue} a_{21}^+}$ at 260 21
\pinlabel ${\color{blue} b_{21}^+}$ at 275 30
\endlabellist
	\centering
	\includegraphics[width=1.2\linewidth]{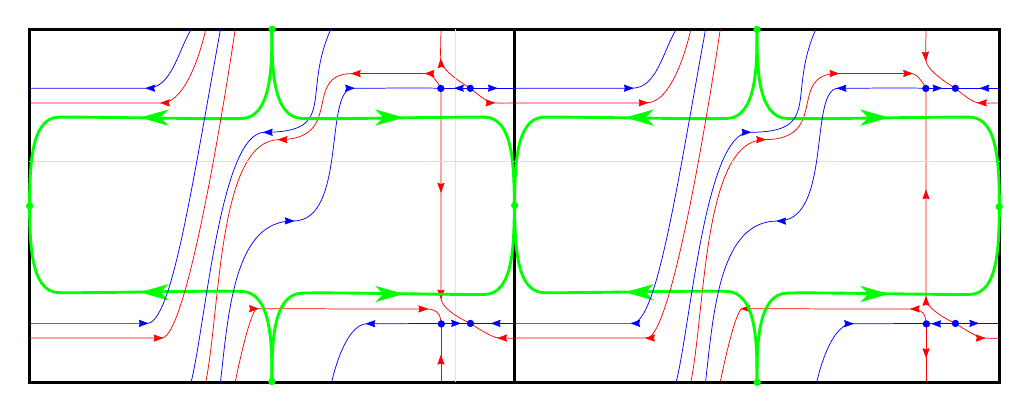}
	\caption{Morse flows and big disks on the conormal of the trefoil.
 }
	\label{fig:trefoilflow}
\end{figure}
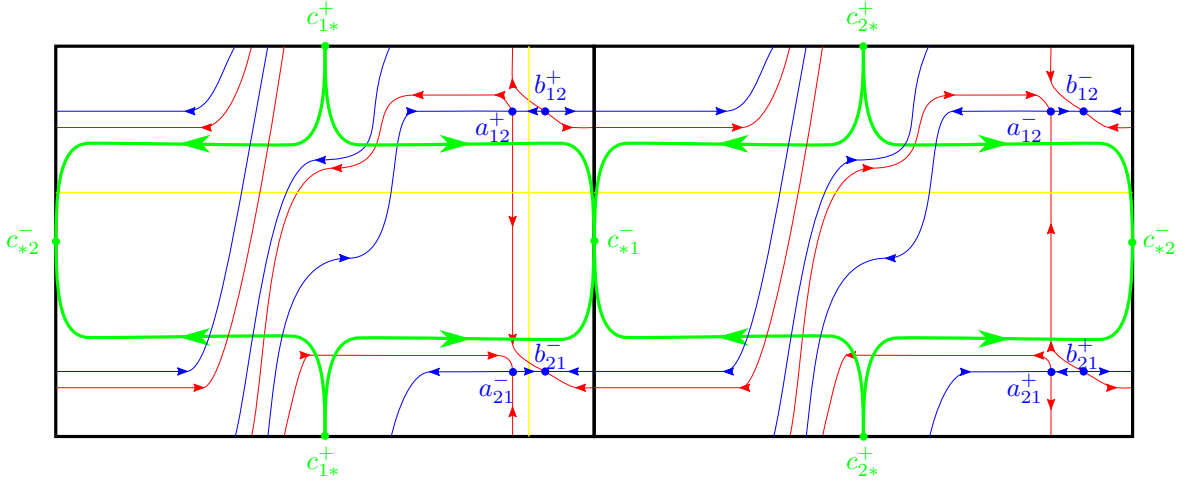

We now proceed to the disks that contribute to $H(b_{12})$, $H(c_{11})$, $H(c_{21})$, and $H(c_{22})$. Most of these disks have a single positive puncture at $b_{12}$ etc., and these can be enumerated as in the explicit calculation of knot contact homology in \cite{EENS}. It is then also straightforward to enumerate the relevant disks with two positive punctures in a similar fashion, and we will list these below without proof. As with the Hopf link calculation in Section~\ref{sec:Hopfrec}, the focus of our calculation is on the powers of $q$ associated to all of these disks.

First consider the order $a^0$ part of $H(b_{12})$ (that is, total order $0$ in $a_{12}$ and $a_{21}$).
This counts disks with a single positive puncture at $b_{12}$. There are two such disks, each given by Morse flow strips, one with negative puncture at $a_{12}$, and one with negative puncture at $a_{21}$. These are precisely the two terms contributing to the differential of $b_{12}$ in Legendrian contact homology; their homology classes are easily read off, and these give the desired expression for $H(b_{12})$, with no powers of $q$ appearing.

\begin{figure}
\labellist
\small\hair 2pt
\pinlabel $\lambda\mu$ at 131 252
\pinlabel $\lambda$ at 426 252
\pinlabel ${Q \partial_{a_{12}}}$ at 131 126
\pinlabel ${Q \partial_{a_{12}}}$ at 426 126
\pinlabel $\mu \partial_{a_{12}}$ at 131 4
\pinlabel ${Q \d^2_{a_{12}} \d_{a_{21}}}$ at 426 4
\endlabellist
	\centering
	\includegraphics[width=\linewidth]{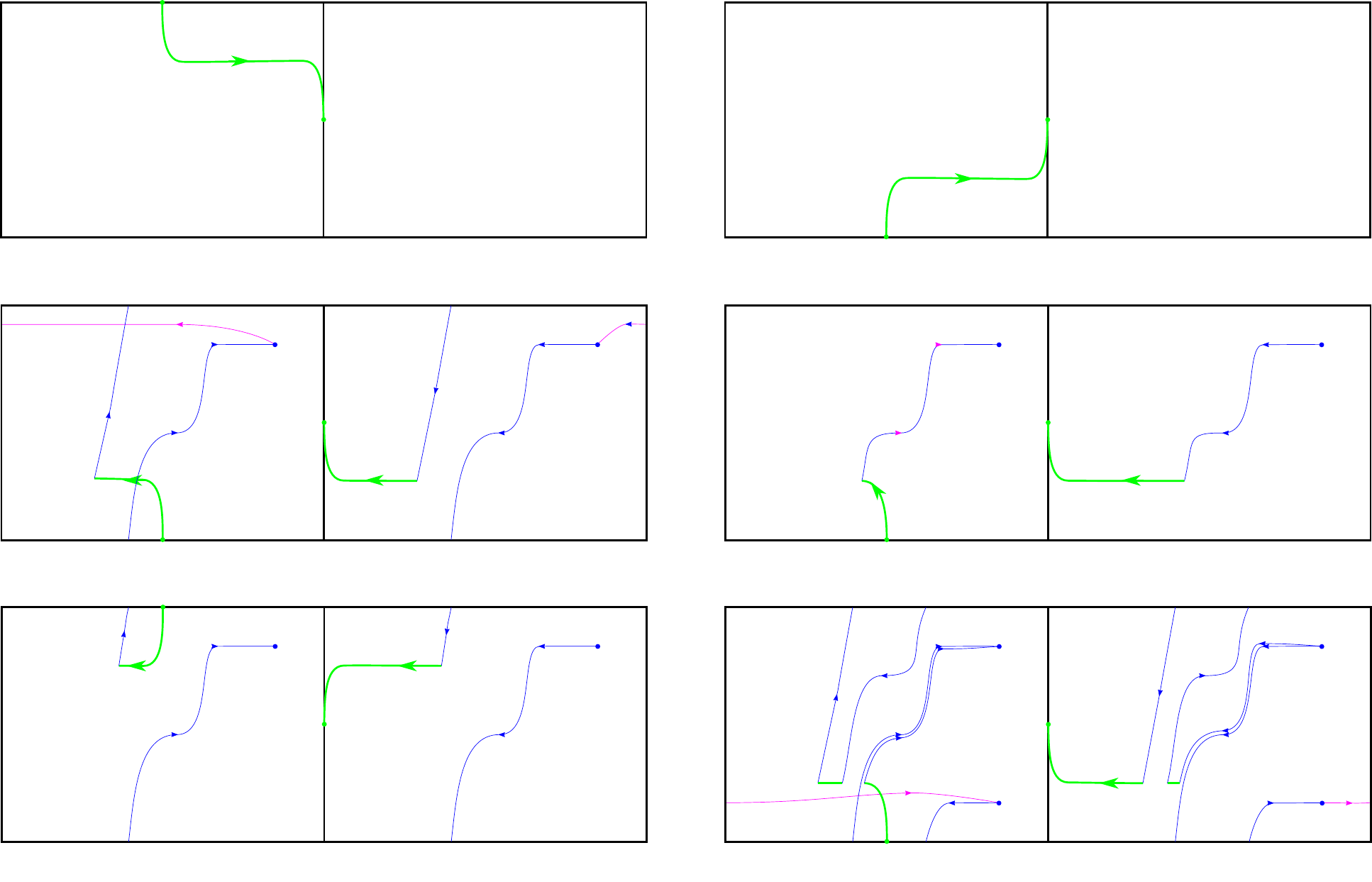}
	\caption{Geometry of coefficients for $H(c_{11})$.}
	\label{fig:H11}
\end{figure}

\begin{figure}
	\centering
	\includegraphics[width=0.8\linewidth]{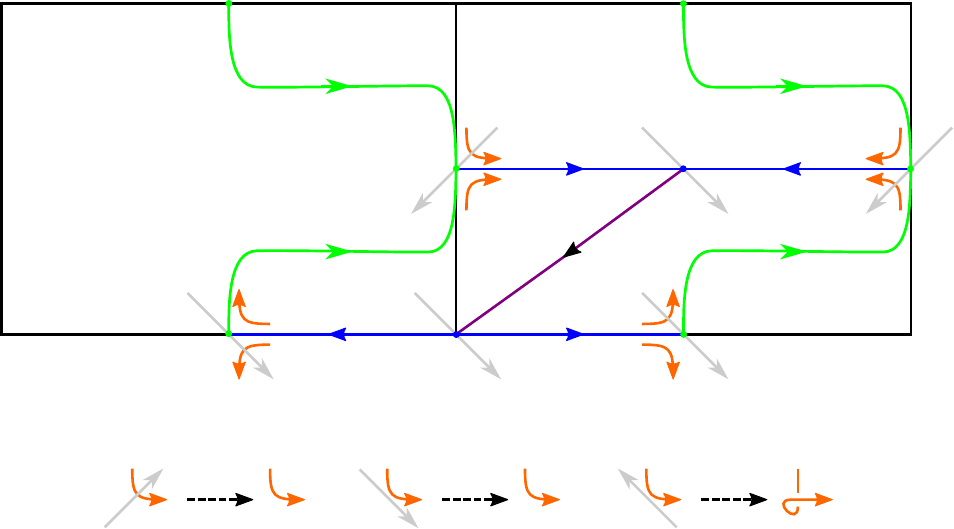}
	\caption{Comparing disks.}
	\label{fig:comparing}
\end{figure}

Next we look at the order $a^0$ part of $H(c_{11})$. There are two disks with positive puncture at $c_{11}$ and no negative punctures, shown in the top row of Figure~\ref{fig:H11}. There are no self-intersections of the boundary of these disks, but the second one comes with a factor of $q^{-1}$. This factor is related to our additional geometric data. To see it we consider the two components of the boundary of the moduli space of disks with positive puncture are the degree $2$ chord $u$ shown in Figure \ref{fig:comparing}. Recall that we need to perturb the gradient $\nabla f$ from the vertical near the Reeb chord endpoints. The arrows in the picture indicate such perturbations. Note that some of the disks near the boundary of the moduli space have tangencies with the perturbed vector field. We choose the perturbations so that negative tangencies give an intersection point with the $4$-chain $C_{T}$ and positive tangencies do not, compare Figure \ref{fig:realcapcross}. Since the algebraic number of points in the boundary of a $1$-dimensional moduli space equals $0$, we find when gluing the strip that the disk contributing to $\lambda$ in $H(c_{11})$ should have an additional $q^{-1}$ factor. Overall these two disks give the terms $\lambda\mu-q^{-1}\lambda$ in $H(c_{11})$.

There are $4$ more disks contributing to $H(c_{11})$ to order $a^0$, shown in the bottom two rows of Figure~\ref{fig:H11}. For the two that contribute $Q\d_{a_{12}}$, one (middle right in Figure~\ref{fig:H11}) has no crossing and hence no linking. For the other, we find two intersection points both contributing negatively to the linking number, combining to give $q^{-1}$: at the intersection with the capping path, the capping path passes under the other branch, while at the other intersection, the Morse flow line lies at the height of the point where it is attached, which here means that it passes over the big disk line. For the disk contributing $\mu\d_{a_{12}}$, there is no crossing and no linking. Finally, for the disk with three negative punctures contributing ${Q \d^2_{a_{12}} \d_{a_{21}}}$, the boundary in the flow tree limit is non-generic with respect to the Reeb flow; perturbing the flow lines, we would find two intersections and these cancel the capping path intersections, leading to no $q$ factor.

\begin{figure}
\labellist
\small\hair 2pt
\pinlabel $Q$ at 131 374
\pinlabel $\mu$ at 426 374
\pinlabel ${\lambda\mu \partial_{a_{21}}}$ at 131 250
\pinlabel ${Q \partial_{a_{12}} \partial_{a_{21}}}$ at 426 250
\pinlabel $\lambda a_{12}$ at 131 128
\pinlabel ${Q a_{12} \d_{a_{12}}}$ at 426 128
\pinlabel ${\lambda a_{12} \d_{a_{12}}}$ at 131 4
\pinlabel ${\lambda a_{12} \d_{a_{12}}}$ at 426 4
\endlabellist
	\centering
	\includegraphics[width=\linewidth]{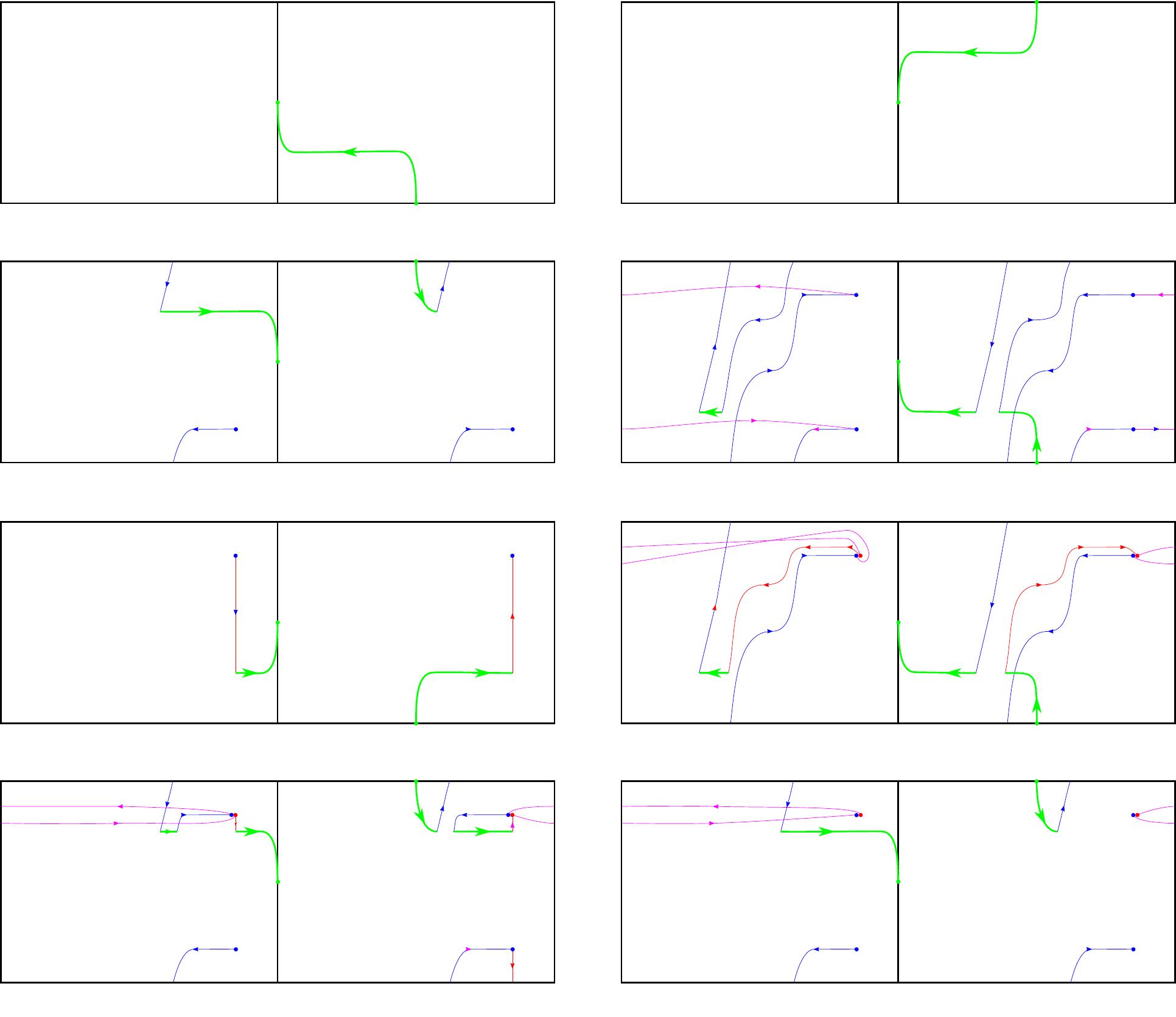}
	\caption{Geometry of coefficients for $H(c_{21})$.}
	\label{fig:H21}
\end{figure}
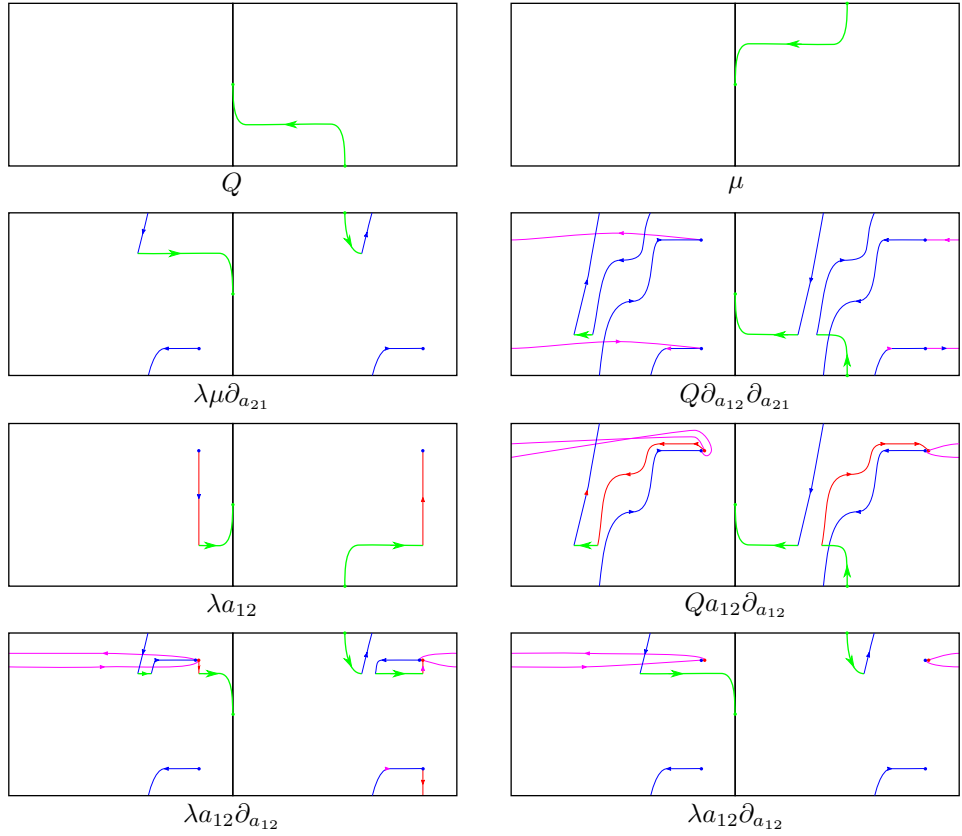

We next turn to $H(c_{21})$ up to order $a^1$. There are four disks with positive puncture at $c_{21}$ and no other positive puncture, shown in the top two rows of Figure~\ref{fig:H21}. For the first three, the boundaries of these disks do not self-intersect and so none of them has a $q$ factor. For the fourth disk contributing to $Q \d_{a_{12}} \d_{a_{21}}$, there are three crossings; the top two cancel and the lower crossing cancels with the contribution at the capping path at the positive end of $a_{21}$, compare Lemma \ref{l:capping}, resulting in no factor of $q$. Calculating orientation signs as determined in \cite{EENS} gives a total contribution of $Q - \mu + \lambda\mu \d_{a_{21}} + Q \d_{a_{12}} \d_{a_{21}}$ from these four disks, and this gives $H(c_{21})$ to order $a^0$.

The remaining contributions to $H(c_{21})$, of order $a^1$, come from disks with two positive punctures, one at $c_{21}$ and another at a degree $0$ Reeb chord. We refer to Remark~\ref{r:factorsforextrapositive} for the scheme to compute powers of $q$ for these disks. There are four such disks, shown in the bottom two rows of Figure~\ref{fig:H21}, and we treat these in order. For the first, contributing $\lambda a_{12}$, we have no crossing and the flow line enters $a_{12}$ from below, and the disk is counted with the factor $-(1-q^{-1})$. For the second, contributing $Qa_{12}\d_{a_{12}}$, there are three visible intersections, two of which cancel, and the flow line enters $a_{12}$ from above. Furthermore, the two copies of the capping path of $a_{12}$ project non-generically; after perturbation we get one extra intersection, and combined with the non-canceled intersection this gives gives a factor $q^{-1}$ and a total of $-q^{-1}(q-1)=q^{-1}-1$. The final two disks in the bottom row of Figure~\ref{fig:H21}, each of which contributes $\lambda a_{12} \d a_{12}$, cancel: the left one, an actual disk, is canceled by the right one, which is the formal disk given by joining a disk in class $\lambda$ with positive puncture at $c_{21}$ and no negative punctures to the trivial strip at $a_{12}$. (To see the latter contribution note the linking between the capping path of the trivial strip and the disk in class $\lambda$ with positive puncture at $c_{21}$.)

\begin{figure}
\labellist
\small\hair 2pt
\pinlabel $\mu$ at 141 272
\pinlabel $1$ at 455 272
\pinlabel ${Q \partial_{a_{21}}}$ at 141 140
\pinlabel ${\mu \partial_{a_{12}} \partial_{a_{21}}}$ at 455 140
\pinlabel ${Q a_{12}}$ at 141 8
\pinlabel ${\mu a_{12} \d_{a_{12}}}$ at 455 8
\endlabellist	\centering
	\includegraphics[width=\linewidth]{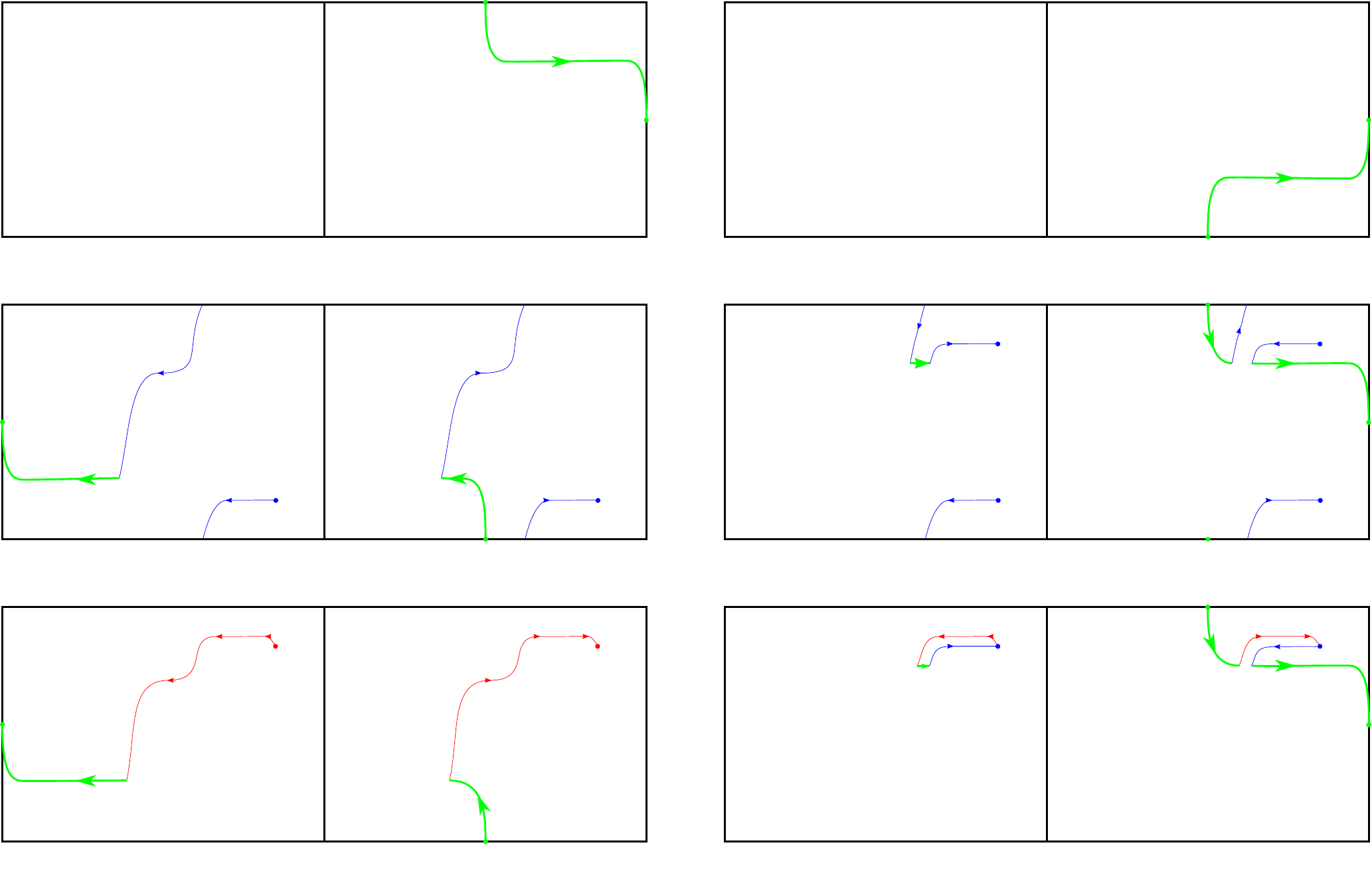}
	\caption{Geometry of coefficients for $H(c_{22})$.}
	\label{fig:H22}
\end{figure}
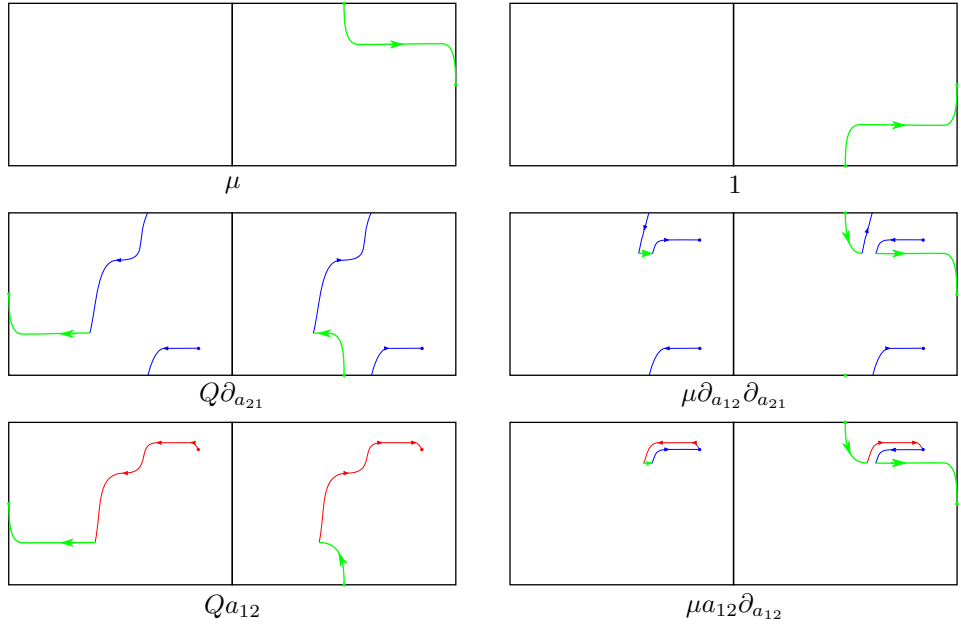

	Finally, we calculate $H(c_{22})$ up to order $a^1$. There are four disks with positive puncture at $c_{22}$ and no other positive puncture, shown in the first two rows of Figure~\ref{fig:H22}. The boundaries of these disks do not self-intersect and so none of them has a $q$ factor, and they give a total contribution of $\mu-1-Q \d_{a_{21}}+\mu \d_{a_{12}} \d_{a_{21}}$. The remaining contributions to $H(c_{22})$ are of order $a^1$ and are given by disks with positive punctures at $c_{22}$ and $a_{12}$. There are two of these, drawn in the bottom row of Figure~\ref{fig:H22}, and they each involve a flow line entering $a_{12}$ from above and without intersection. As with the analogous calculation for $H(c_{21})$, each of these counts with a factor of $q-1$, resulting in a total contribution of $(q-1)Q a_{12} + (q-1)\mu a_{12} \d_{a_{12}}$.
\end{proof}

\subsection{Extracting a quantized polynomial from the Hamiltonian}

In this subsection, we use the Hamiltonian as given in Proposition~\ref{prop:trefoilHam} to produce a $q$-deformed version of the augmentation polynomial of the trefoil, which we will then compare to the known HOMFLY-PT recurrence for the trefoil in Section~\ref{ssec:HOMFLY}.
We first describe how to extract the augmentation polynomial for the trefoil from the differential in knot contact homology; this is the ``classical limit'' of our main computation in this subsection. The relevant portion of the knot contact homology differential (which only counts holomorphic disks with a single positive puncture) can be read off from Proposition~\ref{prop:trefoilHam}, by discarding all $\mathcal{O}(a)$ terms (these represent curves with multiple positive punctures). Algebraically, this corresponds to letting $g_{s}=0$ and $q=e^{g_{s}}=1$, and replacing the operators $\d_{a_{12}}$ and $\d_{a_{21}}$ by the usual DGA generators that we write $a_{12}$ and $a_{21}$, respectively. This gives
\begin{equation}
\begin{aligned}
d(b_{12}) &= \lambda^{-1} a_{12} - a_{21} \\
d(c_{11}) &= \lambda\mu-\lambda-(2Q-\mu) a_{12} - Q a_{12}^2 a_{21} \\
d(c_{21}) &= Q-\mu+\lambda\mu a_{21} + Q a_{12}a_{21} \\
d(c_{22}) &= \mu-1-Q a_{21}+\mu a_{12}a_{21}.
\end{aligned} \label{eq:dgadiffltrefoil}
\end{equation}
Here for simplicity we have passed to the abelianized quotient of the usual Legendrian DGA, in which Reeb chords commute with each other; this quotient suffices for considering augmentations.

To calculate the augmentation polynomial, we want to find the set of $\lambda,\mu,Q$ so that there is some choice of $a_{12},a_{21}$ for all four polynomials in \eqref{eq:dgadiffltrefoil} vanish. This calculation is an exercise in elimination theory. First note that
\[
d(c_{11}) + a_{12} d(c_{21}) - \lambda d(c_{22}) + Q \lambda d(b_{12}) = 0
\]
and so the relation given by $d(c_{11})=0$ is redundant. Using the other differentials, we define
\begin{align*}
A &= \mu d(c_{21}) - Q d(c_{22}) = (Q-\mu^2)+(Q^2+\lambda\mu^2)a_{21} \\
B &= Q d(c_{21}) + \lambda\mu d(c_{22}) = (Q^2+\lambda\mu^2-Q\mu-\lambda\mu)+(Q^2+\lambda\mu^2)a_{12}a_{21} \\
C &= a_{12}A = (Q-\mu^2)a_{12}+(Q^2+\lambda\mu^2)a_{12}a_{21} \\
D &= C - \lambda(Q-\mu^2)d(b_{12}) = \lambda(Q-\mu^2)a_{21}+(Q^2+\lambda\mu^2)a_{12}a_{21}\\
E &= D-B =  -(Q^2+\lambda\mu^2-Q\mu-\lambda\mu)+\lambda(Q-\mu^2)a_{21}.
\end{align*}
Now $A,E$ are linear in $a_{21}$ and we can eliminate $a_{21}$:
\[
\lambda(Q-\mu^2)A-(Q^2+\lambda\mu^2)E = \Aug_T(\lambda,\mu,Q),
\]
where
$\Aug_T$ is the augmentation polynomial of the trefoil:
\[
\Aug_T(\lambda,\mu,Q) = (\mu^4-\mu^3)\lambda^2+(\mu^4-Q\mu^3+2Q^2\mu^2-2Q\mu^2-Q^2\mu+Q^2)\lambda+(Q^4-Q^3\mu).
\]
(Note that this is unnormalized; to obtain the augmentation polynomial for the trefoil as given in \cite{Ngsurvey}, replace
$\lambda,\mu,Q$ by $\lambda\mu^3U^{-1},\mu^{-1},U^{-1}$ respectively.)

We now replicate this calculation non-classically. First, we have
\[
H(c_{11}) +  \d_{a_{12}} H(c_{21}) - \lambda H(c_{22}) + Q \lambda H(b_{12}) = \mathcal{O}(a)
\]
and so $H(c_{11})$ is redundant as before. 
Next, define
\begin{align*}
A &= \mu H(c_{21})-Q H(c_{22}) \\
&= (Q-\mu^2)+(Q^2+q\lambda\mu^2) \d_{a_{21}} + \mathcal{O}(a) \\
B &= Q H(c_{21})+ q\lambda\mu H(c_{22}) \\
&= Q(Q-\mu)+q\lambda\mu(\mu-1)+(1-q)Q\lambda\mu\d_{a_{21}} \\
&\qquad + (Q^2+q\lambda\mu^2)\d_{a_{12}} \d_{a_{21}} + \mathcal{O}(a) \\
C &= \d_{a_{12}} A \\
&= (1-q)\lambda\mu-(q-1)Q^2
+(Q+(q^{-1}-q)Q \mu-\mu^2) \d_{a_{12}} \\
&\qquad
+ (Q^2+q\lambda\mu^2) \d_{a_{12}} \d_{a_{21}} + \mathcal{O}(a).
\end{align*}
In the expression for $C$, for the terms linear in $\d_{a_{12}}$, we replace $\d_{a_{12}}$ by $\lambda \d_{a_{21}}$ by gluing to $\lambda H(b_{12}) = \d_{a_{12}} - \lambda  \d_{a_{21}}$. In doing this, we need to move the capping path for $\lambda$ to the front of the expression, which will introduce a factor of $q^{-1}$ to certain terms in $C = \d_{a_{12}}(\mu H(c_{21})-Q H(c_{22}))$. More precisely, the terms in $C$ that are linear in $\d_{a_{12}}$ are:
\[
(\mu(Q-\mu) +(q^{-1}-1)Q\mu- Q(\mu-1)-(q-1)Q\mu) \d_{a_{12}},
\]
and the terms that are multiplied by $q^{-1}$ are $-\mu^2$, $-Q\mu$, and $-(q-1)Q\mu$, corresponding to the terms $-\mu$ in $H(c_{21})$, $\mu$ in $H(c_{22})$, and $(q-1)\mu a_{12} \d_{a_{12}}$ in $H(c_{22})$, respectively. Thus gluing to $\lambda H(b_{12})$ yields
\begin{align*}
(\mu(Q-q^{-1}\mu) &+(q^{-1}-1)Q\mu- Q(q^{-1}\mu-1)-q^{-1}(q-1)Q\mu) \lambda \d_{a_{21}} \\
&= (Q\lambda+(1-q)Q\lambda\mu-q\lambda\mu^2) \d_{a_{21}}
\end{align*}
and the full relation coming from $C$ is:
\begin{align*}
D &=  (1-q)\lambda\mu-(q-1)Q^2
+(Q\lambda+(1-q)Q\lambda\mu-q\lambda\mu^2) \d_{a_{21}} \\
&\qquad
+ (Q^2+q\lambda\mu^2) \d_{a_{12}} \d_{a_{21}} + \mathcal{O}(a).
\end{align*}
Subtracting $B$ from $D$ gives
\[
E = D-B = 
(-q Q^2+Q\mu+\lambda\mu-q\lambda\mu^2)+(Q-q^{-1}\mu^2)\lambda \d_{a_{21}} + \mathcal{O}(a).
\]

We now want to combine $A$ and $E$ to eliminate the $\d_{a_{21}}$ terms. Note that we have the following identity, which can be verified by breaking the left hand side into a sum of two terms, one for $Q^2$ and one for $q^{-1}\lambda\mu^2$, and then moving these monomials to the left:
\begin{align*}
&(Q-q^{-1}\mu^2)(Q-q^{-3} \mu^2)(Q^2+q^{-1}\lambda\mu^2)\\
&= (Q^2(Q-q^{-3}\mu^2)+q^{-1}\lambda\mu^2(Q-q\mu^2))(Q-q^{-1}\mu^2).
\end{align*}
Thus the following linear combination of $A$ and $E$ eliminates the $\d_{a_{21}}$ terms:
\begin{align*}
&(Q-q^{-1}\mu^2)(Q-q^{-3} \mu^2)\lambda A - (Q^2(Q-q^{-3}\mu^2)+q^{-1}\lambda\mu^2(Q-q\mu^2)) E \\
& \qquad =
(Q-q^{-1}\mu^2)(Q-q^{-3} \mu^2)\lambda(Q-\mu^2)\\
&\qquad \quad - (Q^2(Q-q^{-3}\mu^2)+q^{-1}\lambda\mu^2(Q-q\mu^2))
(-q Q^2+Q\mu+\lambda\mu-q\lambda\mu^2) + \mathcal{O}(a).
 \end{align*}

This last expression (without $\mathcal{O}(a)$) is the $q$-deformed augmentation polynomial for the trefoil, in the framing given by the braid $\sigma_1^3$. We write this as
$\qAug_T^3(\lambda,\mu,Q,q)$, where $3$ denotes the framing number corresponding to $\sigma_1^3$.
To convert to the polynomial for framing number $0$, we use the following result; compare Theorem~\ref{t:framing}.

\begin{proposition}
Let $\qAug_K^f(\lambda,\mu,Q,q) = \sum_{i=0}^k p_i(\mu,Q,q)\lambda^i$ denote the quantized augmentation polynomial for a knot $K$ in framing $f$. Then
\[
\qAug_K^{f-f'}(\lambda,\mu,Q,q) = \sum_{i=0}^k q^{f'i^2/2} p_i(\mu,Q,q)\mu^{-f'i}\lambda^i;
\]
that is, to lower the framing by $f'$, replace each $\lambda^i$ by $q^{f'i^2/2}\mu^{-f'i} \lambda^i$.
\label{prop:framingchange}
\end{proposition}

\begin{proof}
The quantized augmentation polynomial in framing $f$ annihilates the wave function $\Psi_{K}^{f}$ in framing $f$. Theorem \ref{t:framing} shows how the wave function transforms as we change the framing. The tranformation rule for $\qAug_{K}$ is the corresponding transformation of operators. 
\end{proof}

Using Proposition~\ref{prop:framingchange} and the above expression for $\qAug_T^3(\lambda,\mu,Q,q)$, and multiplying on the left by $\mu^3$ to clear denominators, we finally arrive at $\qAug_T = \qAug_T^0$, the quantized augmentation polynomial for the right-handed trefoil in the $0$ framing:
\begin{align*}
\qAug_T(\lambda,\mu,Q,q) &= 
q Q^3 \mu^3(Q-q^{-3}\mu^2)(Q-q^{-1}\mu) \\
&\quad + q^{-5/2}(Q-q^{-2}\mu^2)\left((q^2 \mu^2+q^3 \mu^2-q^3 \mu+q^4)Q^2 \right.\\
&\qquad 
\left.-(q\mu^3+q^3\mu^2+q\mu^2)Q+\mu^4\right)  \lambda \\
&\quad +(Q-q^{-1}\mu^2)(\mu-q)\lambda^2.
\end{align*}

\subsection{Colored HOMFLY-PT recurrence for the trefoil}
\label{ssec:HOMFLY}

We now compare $\qAug_T$ with the recurrence relation for the colored HOMFLY-PT polynomials of the trefoil. From \cite[\S 1.4]{GLL}, this is:
\begin{align*}
P_T(L,M,x,q)&=x^4(x^2M^2-1)(q^6x^2M^4-1) L^0 \\
&\quad +q^7\left(q^4x^2M^4-1)(q^8x^4M^8-q^4x^4M^6+q^2x^4M^4+x^4M^4\right. \\
& \qquad \left.-q^6x^2M^4-q^2x^2M^4-x^2M^2+1\right) L^1 \\
&\quad -q^{18}x^2M^6(q^4M^2-1)(q^2x^2M^4-1) L^2.
\end{align*}
For convenience, we recall from \cite{GLL} what it means for $P_T$ to be the recurrence relation. Let $W_{T;n}(x,q)$ denote the colored HOMFLY-PT polynomial for $T$ colored by the partition with a single row and $n$ boxes. Define operators $M,L$ on the set of all functions from $\mathbb{N}$ to $\mathbb{Q}(x,q)$ by
\[
(Lf)(n) = f(n+1), \qquad (Mf)(n) = q^n f(n).
\]
Then for the function $f$ given by $f(n) = W_{T;n}(x,q)$, we have
\[
P_T f = 0.
\]
Note that $P_T$ is an element of a Weyl algebra, namely the algebra over $\mathbb{Q}(q,x)$ generated by $L$ and $M$, with $L M = q M L$.

\begin{remark}
We can compare the recurrence relation $P_T(L,M,x,q)$ to the expression given in \cite{FGS} for the recurrence relation of the left-handed trefoil. If we start with the expression in \cite[(2.22)]{FGS}, and then set $t=-1$ (this corresponds to taking Euler characteristic in the homological grading) and replace $\widehat{x}$ and $\widehat{y}$ from \cite{FGS} by $M$ and $L$, respectively, we get the following recurrence for the left-handed trefoil:
\begin{align*}
&-a^2q M^3 (M-1)(a M^2 q-1) L^0 \\
&+a(aM^2-1)(-aM^2+q+a^2M^4q-Mq^2+M^2q^2-aM^2q^2-aM^3q^2+M^2q^3)L^1 \\
&+q(aM-1)(aM^2-q)L^2.
\end{align*}
Replacing $a \mapsto a^{-1}$, $q \mapsto q^{-1}$, $M \mapsto M^{-1}$ to pass to the mirror knot, we find precisely the polynomial $P_T(L,M,x,q)$ given above.
\end{remark}

We are now in a position to compare the recurrence $P_T(L,M,x,q)$ with the quantized augmentation polynomial $\qAug_T(\lambda,\mu,Q,q)$ from the previous subsection. The following result shows that they agree, and is proven by simply comparing the explicit equations for $P_T$ and $\qAug_T$.

\begin{proposition}
For the right-handed trefoil, the polynomials $P_T(L,M,x,q)$ and $\qAug_T(\lambda,\mu,Q,q)$ agree after a change of variables, up to an overall unit:
\[
P_T(q^{-1}Q^{-1}\lambda,q^{-1/2}\mu^{-1/2},qQ^{1/2},q^{1/2}) = q^7Q^{-1}\mu^{-6} \qAug_T(\lambda,\mu,Q,q).
\]
\end{proposition}

We note that the commutation relations $L M = q M L$ and $\mu\lambda=q\lambda\mu$ are compatible with each other under this change of variables.

\bibliographystyle{hplain}
\bibliography{myrefsqc}

\end{document}